\documentclass[11pt]{amsart}

\usepackage[marginratio=1:1,width=480pt,tmargin=48pt,bmargin=55pt]{geometry}

\usepackage{amssymb,latexsym,amsmath,amsthm,amscd}
\usepackage{mathrsfs}   
\usepackage{fancyhdr}
 \usepackage{ifthen}

\usepackage{etoolbox}

\newcommand{\im}{\lrcorner\,}
\appto\appendix{\addtocontents{toc}{\protect\setcounter{tocdepth}{1}}}
\appto\listoffigures{\addtocontents{lof}{\protect\setcounter{tocdepth}{1}}}
\appto\listoftables{\addtocontents{lot}{\protect\setcounter{tocdepth}{1}}}
\theoremstyle{plain}
\newtheorem{theorem}{Theorem}[section]
\newtheorem{corollary}[theorem]{Corollary}
\newtheorem{lemma}[theorem]{Lemma}
\newtheorem*{theorem*}{Theorem}
\newtheorem*{corollary*}{Corollary}
\newtheorem{proposition}[theorem]{Proposition}
\newtheorem{definition}[theorem]{Definition}

\theoremstyle{remark}

\newtheorem{remark}[theorem]{Remark}

\newtheorem{example}[theorem]{Example}
\newtheorem*{example*}{Example}

\numberwithin{equation}{section}

\usepackage{array,rotating}  
\usepackage{color}
\usepackage{xcolor}
\definecolor{carmine}{rgb}{0.59, 0.0, 0.09}
\definecolor{mediumpersianblue}{rgb}{0.0, 0.4, 0.65}
\definecolor{persianplum}{rgb}{0.44, 0.11, 0.11}
\usepackage[colorlinks=true,
            linkcolor=persianplum, 
            urlcolor=olive,
            citecolor=mediumpersianblue, 
            backref=page]{hyperref}
\usepackage{amsmath}
\usepackage{amsfonts}
\usepackage{amssymb}
 
\newcommand{\cC}{\mathcal{C}}

\newcommand{\cG}{\mathcal{G}}

\newcommand{\cD}{\mathcal{D}}
\newcommand{\cL}{\mathcal{L}}

\newcommand{\cT}{\mathcal{T}}

\newcommand{\cI}{\mathcal{I}}

\newcommand{\ccA}{\mathcal{A}}

\newcommand{\scD}{\mathcal{D}}

\newcommand{\scK}{\mathcal{K}}

\newcommand{\scH}{\mathcal{H}}

\newcommand{\fg}{\mathfrak{g}}
\newcommand{\fp}{\mathfrak{p}}

\newcommand{\ri}{\operatorname{i}}

\newcommand{\R}{\mathbb{R}}
\newcommand{\RR}{\mathbb{R}}

\newcommand{\CC}{\mathbb{C}}
\newcommand{\AAA}{\mathbb{A}}

\newcommand{\PP}{\mathbb{P}}

\newcommand{\DD}{\mathbb{D}}
\newcommand{\VV}{\mathbb{V}}

\newcommand{\bh}{\mathbf{h}}

\newcommand{\bu}{\mathbf{u}}

\newcommand{\bC}{\mathbf{C}}

\newcommand{\bQ}{\mathbf{Q}}
\newcommand{\bR}{\mathbf{R}}

\newcommand{\bT}{\mathbf{T}}

\newcommand{\w}{{\,{\wedge}\;}}

\newcommand{\exd}{\mathrm{d}}

\newcommand{\ve}{\varepsilon}

\newcommand{\spn}{\operatorname{span}}

\newcommand{\half}{{\textstyle{\frac 12}}}

\newcommand{\biw}{\bigwedge\nolimits}
\newcommand{\what}{\widehat}
\newcommand{\wt}{\widetilde}

\newcommand{\cthetao}{\overline{\theta^1}}
\newcommand{\cthetat}{\overline{\theta^2}}
\newcommand{\cthetai}{\overline{\theta^i}}
\newcommand{\cthetaj}{\overline{\theta^j}}

\newcommand{\cscK}{\overline{\scK}}

\newcommand{\co}{\overline 1}
\newcommand{\ct}{\overline 2}
\newcommand{\ov}{\overline}

\renewcommand{\Re}{\operatorname{Re}}
\renewcommand{\Im}{\operatorname{Im}}
\makeatletter
\renewcommand*{\p@section}{\S\,}
\renewcommand*{\p@subsection}{\S\,}
\renewcommand*{\p@subsubsection}{\S\,}
 
\makeatother
\usepackage{float}

\begin{document}

\author{Omid Makhmali}\author{David Sykes}

 \address{\newline  Omid Makhmali\\\newline
Department of Mathematics and Natural Sciences, Cardinal Stefan Wyszy\'nski University, ul. Dewajtis 5,  Warszawa, 01-815, Poland \\\newline
\textit{Email address: }{\href{mailto:o.makhmali@uksw.edu.pl}{\texttt{o.makhmali@uksw.edu.pl}}}\\\newline
Department of Mathematics and Statistics, UiT The Arctic University of Norway, Troms\o\  90-37,Norway \\\newline\newline
David Sykes\\\newline
Institute for Basic Science, Center for Complex Geometry, Daejeon, 34126, Republic of Korea\\\newline
   \textit{Email address: }{\href{mailto:sykes@ibs.re.kr}{\texttt{sykes@ibs.re.kr}}}
 }

\title[]
{Pre-K\"ahler structures and finite-nondegeneracy}

\begin{abstract}
  Motivated by the geometry of Levi degenerate CR hypersurfaces, we define a \emph{pre-K\"ahler structure} on  a complex manifold as a  pre-symplectic structure compatible with the almost complex structure, i.e. a closed (1,1)-form.  Extending  \emph{Freeman filtration} to the pre-K\"ahler setting, we define  holomorphic degeneration and finite-nondegeneracy and show that the symmetry algebra of a real analytic pre-K\"ahler structure is finite-dimensional if and only if it is finitely nondegenerate.  Concurrently, we extend the classical correspondence between K\"ahler and Sasakian structures  to the pre-K\"ahler setting, i.e.  a one-to-one (local) correspondence between $k$-nondegenerate CR hypersurfaces equipped with a transverse infinitesimal symmetry and $k$-nondegenerate pre-K\"ahler structures.
We additionally formalize a second relationship between the categories, constructing a natural $k$-nondegenerate pre-K\"ahler structure defined on a line bundle over such CR structures via pre-symplectification.

Focusing on the lowest dimensional case, we solve the equivalence problem of non-K\"ahler pre-K\"ahler complex surfaces that are 2-nondegenerate by associating a Cartan geometry to them and explicitly express their local invariants in terms of the fifth jet of a potential function. We describe the vanishing of their basic invariants in terms of a double fibration, which gives a pre-K\"ahler characterization of the twistor bundle of symplectic connections on surfaces.

Lastly, we study the pre-K\"ahler complex surfaces arising as symmetry reductions of homogeneous 2-nondegenerate CR 5-manifolds, which leads to a  characterization of certain \emph{critical} symplectic connections on surfaces. We derive two fundamental results for such pre-K\"{a}hler manifolds:  Their moduli space of geometrically distinct structures contain $2$-dimensional open dense subsets, and they all have nontrivial infinitesimal symmetries. Finally, we show that all locally homogeneous pre-K\"{a}hler complex surfaces are locally flat.
\end{abstract}


\maketitle
  
\vspace{-1.1 cm}

\setcounter{tocdepth}{2} 
\tableofcontents

\newgeometry{marginratio=1:1,height=670pt,width=480pt,tmargin=75pt}

\section{Introduction}
\label{sec:introduction}

A natural complex geometric structure is induced on the space of integral curves of an infinitesimal symmetry on a hypersurface-type CR manifold. When the hypersurface-type CR structure with an infinitesimal symmetry is Levi nondegenerate, locally, such a leaf space is equipped with a unique (pseudo-)K\"ahler structure, which in turn encodes the original hypersurface-type CR structure and the choice of infinitesimal symmetry, establishing the well-known K\"ahler--Sasakian correspondence. The setting of interest in this article is Levi degenerate hypersurface-type CR structures equipped with an infinitesimal symmetry. In this setting, the locally defined leaf space, as a complex manifold, is equipped  with a compatible \emph{pre-symplectic} form $\omega$, that is, a real closed $1$-form of type $(1,1)$ with respect to the underlying complex structure. We refer to such pre-symplectic generalizations of K\"{a}hler geometries as \emph{pre-K\"ahler structures.}

These (possibly non-K\"{a}hler) structures on complex manifolds arise in several contexts aside from their relationship to CR geometry.  Recalling that the Ricci 2-form of a K\"ahler metric is a closed (1,1)-form, one can use local solvability of the complex Monge-Amp\`ere equation to show that every pre-K\"ahler structure is locally defined by the Ricci curvature of a K\"ahler metric. In fact, pre-K\"ahler structures encode the geometry of equivalence classes of K\"ahler metrics with the same Ricci curvature. 
 A pre-K\"{a}hler structure on a complex surface appears with a fundamental role in the recent work \cite{mok2023elliptic} on elliptic surfaces.  Furthermore, pre-K\"ahler structures encode the submanifold geometry of pseudo-K\"{a}hler manifolds, and naturally appear on twistor bundles for a variety of geometric structures. We present an example of this latter relationship, identifying a distinguished sub-class of pre-K\"{a}hler structures with twistor bundles for symplectic connections in dimension two.

While non-K\"{a}hler pre-K\"{a}hler geometries have not been studied to our knowledge, considerable theory has been developed for related Levi degenerate CR structures. So our first aim with this article is to highlight how the two categories are related, establishing basic general results for pre-K\"{a}hler geometry. We then specialize to study everywhere non-K\"{a}hler pre-K\"{a}hler complex surfaces, which is the lowest dimension where such structures have nontrivial local geometry.

\subsection{Outline and main results}\label{sec:outline-main-results}
We develop general fundamentals for pre-K\"{a}hler geometry in  \ref{sec: CR pre-Kahler corresp.}. In \ref{sec: terminology section}, we introduce preliminary terminology, defining pre-K\"{a}hler and pre-Sasakian structures, and establish that pre-K\"{a}hler structures coincide exactly with the induced geometries on complex submanifolds in pseudo-K\"{a}hler manifolds, Proposition \ref{prop: submanifold}. We also show a local one-to-one correspondence between pre-K\"ahler structures and the equivalence class of K\"ahler metrics with the same Ricci 2-forms, Proposition \ref{lem:pre-kahler-Kahler-Ricci-cohom}.

In \ref{sec: Straightenability and k-nondegeneracy} we introduce definitions of finite-nondegeneracy, $k$-nondegeneracy, and holomorphic degeneracy for pre-K\"{a}hler manifolds, structure invariants that are analogues of well-known invariants from CR geometry by the same name. We derive a simple criterion, \emph{total straightenability}, for when the complex structure on a pre-K\"{a}hler manifold canonically descends to a complex structure on the local leaf spaces of the foliation generated by the pre-K\"{a}hler form's kernel,  Lemma \ref{lemma: total straightenability lemma}. We show that finite-nondegeneracy at generic points is necessary and sufficient for a pre-K\"{a}hler manifold to not be locally decomposable into a product of some lower-dimensional pre-K\"{a}hler manifold and a complex manifold, Proposition \ref{prop: holomorphically degenerate and straightenability}.

In \ref{sec: Symmetry reduction and pre-Sasakian structures}, we establish a one-to-one correspondence between $2n$-dimensional pre-K\"{a}hler structures and $(2n+1)$-dimensional pre-Sasakian structures, Theorem \ref{thm: correspondence theorem}. Applying known results from CR geometry through this correspondence, in the theorem below we find a necessary and sufficient condition   for  finite-dimensionality of the symmetry algebra at generic points of a real analytic pre-K\"{a}hler manifold. 
\newtheorem*{thmA}{\bf Theorem \ref{finite-dimensionality theorem}}
\begin{thmA}
\emph{  A real analytic pre-K\"{a}hler structure admitting a Freeman filtration has finite-dimensional infinitesimal symmetry algebras if and only if it is $k$-nondegenerate.}
\end{thmA}
In \ref{sec: presymplectification section}, we establish yet another natural relationship between pre-K\"{a}hler and pre-Sasakian structures canonically associating a $(2n+2)$-dimensional pre-K\"{a}hler structure to each point on a $(2n+1)$-dimensional pre-Sasakian manifold, Proposition \ref{prop: presymplectification prop}. Concluding \ref{sec: CR pre-Kahler corresp.}, in \ref{sec: Low order jets of finitely-nondegenerate potentials} we derive basic conditions that uniform $k$-nondegeneracy imposes on the $(k+1)$-jet of a pre-K\"{a}hler structure's potentials, Proposition \ref{prop: leading terms of k nond def eqn lemma} and Corollary \ref{corol: leading terms of k nond def eqn}.

In \ref{sec:4D-prekahler}, we present a thorough study of  non-(pseudo-)K\"{a}hler pre-K\"{a}hler structures with nonzero constant rank pre-symplectic form in dimension 4, the lowest dimension wherein such structures are nontrivial. In \ref{sec:structure-bundle}, we follow Cartan's method of equivalence to associate a canonical principal $\mathrm{U}(1)$-bundle of coframings adapted to such structures. In \ref{sec:cart-geom-description}, after recalling the definition of a Cartan geometry, we  give a solution of the local equivalence problem for $2$-nondegenerate pre-K\"ahler complex surfaces as a class of Cartan geometries in the following theorem.
\newtheorem*{thmB}{\bf Theorem \ref{thm:cart-geom-descr}}
\begin{thmB}
\emph{There is a canonical equivalence betweeen  2-nondegenerate pre-K\"ahler structures $(g,\omega)$ on a complex surface $M$ and Cartan geometries $(\cG\to M,\varphi)$ of type $(\RR^2\rtimes\mathrm{SL}(2,\RR) ,\mathrm{U}(1))$ whose curvature is given as \eqref{eq:Cartan-conn-curv} and \eqref{eq:cartan-curv-prekahler}   for some  $T_1,T_2\in C^{\infty}(\cG,\CC)$ and $T_3\in C^{\infty}(\cG,\ri\RR).$  The basic invariants for such Cartan geometries are  $\bT_1,\bT_2\in C^\infty(M,\RR)$  where}
  \begin{equation}\label{eq:fund-inv-prekahler4-intro}
\bT_1:=T_1\ov{T_1},\quad   \bT_2:=T_2\ov{T_2},
  \end{equation}
\emph{whose vanishing  characterizes locally flat pre-K\"ahler structures, i.e. $(g,\omega)$ is locally equivalent to the homogeneous space $G/\mathrm{U}(1)$ where $G=\RR^2\rtimes \mathrm{SL}(2,\RR).$}
\end{thmB}
A corollary of this theorem is that the pre-K\"{a}hler structure considered in \cite{mok2023elliptic} can be defined locally as the most symmetric $2$-nondegenerate pre-K\"{a}hler structure on complex surfaces (see Example \ref{exa:homog-surf}).

This  result should be contrasted with studies on $5$-dimensional uniformly $2$-nondegenerate CR hypersurface geometries, e.g. \cite{isaev2013reduction}, as such structures underly the pre-Sasakian structures associated to 2-nondegenerate pre-K\"{a}hler complex surfaces.  For experts familiar with Levi degenerate CR structures, it may be a surprise that our connection construction identifies a natural complement to the pre-K\"{a}hler form's kernel, as there is no such natural complement to Levi kernels in the general CR setting. The complement in our setting has a nice geometric description as the space spanned by $X$ and $\overline{X}$ for a complex vector field $X \in \Gamma( T^{1,0}M)$ such that
\begin{itemize}
\item With respect to the pre-K\"{a}hler metric $g,$ one can normalize $X$ so that
\[
g(\Re(X),\Re(X)) = g(\Im(X),\Im(X)) = \tfrac{1}{2}.
\]
\item There exists  $Y\in \Gamma( T^{1,0}M)$ which is in the kernel of the presymplectic form and satisfies $[Y,\overline{X}]\equiv X\pmod{\langle \overline{X},Y\rangle}$. 
\end{itemize}
Uniqueness and existence of such $\langle X,\overline{X}\rangle$ follows from calculations in \ref{sec:structure-bundle}. Of course, the metric condition is not available in the purely CR setting, which clarifies why the natural complement phenomenon only appears for pre-K\"{a}hler structures.

In \ref{sec:param-expr}, we  start with a choice of (local) potential function, express  an adapted coframing and  the basic invariants in some local complex coordinate system, showing, in particular, that the basic invariants whose vanishing imply flatness depend on the fifth jet of the potential function.
 
In \ref{sec:twist-bundle-characterization}, we interpret the vanishing of each of the basic invariants. In particular, after recalling the definition of a symplectic connection in dimension two and its twistor bundle as a complex surface, in the theorem below we show that these twistor bundles are indeed a distinguished class of pre-K\"ahler complex surfaces.
\newtheorem*{thmC}{\bf Theorem \ref{thm:twistor-bundle-characterization}}
\begin{thmC}
\emph{  There is a one-to-one correspondence between symplectic connections on surfaces and 2-nondegenerate pre-K\"ahler structures on complex surfaces satisfying $\bT_2=0.$}
\end{thmC}
In \ref{sec:remark-equi-surf-tubification}, we remark on the naturally defined symplectic connection on embedded equiaffine surfaces and briefly discuss a map from Levi nondegenerate tube-type CR hypersurfaces in $\mathbb{C}^3$ to a class of 2-nondegenerate CR hypersurfaces with an infinitesimal symmetry in $\mathbb{C}^3$. 
An interpretation of the condition $\bT_1=0$ becomes rather technical and we only briefly describe it in \ref{sec:remark-bt_1=0}.

In \ref{sec:symm-reduct-homog}, we study pre-K\"ahler structures defined by symmetry reductions of locally homogeneous 2-nondegenerate CR 5-manifolds. The symmetry reductions of a flat CR 5-manifold  is carried out in    \ref{sec:symm-reduct-flat} using a Cartan geometric approach. In \ref{sec:crit-spec-sympl}, we show that   symmetry reductions of the flat model define symplectic connections which are characterized in Proposition \ref{prop:Bochner-flat-pre-Kahler} in terms of the vanishing of an invariant trilinear form. Furthermore, after recalling the notion of criticality for symplectic connections, e.g. see \cite{Cahen1, Fox-affine}, in Proposition \ref{prop:crit-spec-sympl-1} we show that such symplectic connections are critical, exhibiting a pre-K\"{a}hler analogue to extremal K\"ahler metrics. In Remark \ref{rmk:sym-red-other-2nondeg-CR}, we highlight extensions of this result and the benefit of our  Cartan geometric approach in carrying out the symmetry reduction. In \ref{sec:pre-sasak-struct}, we describe symmetry reductions of non-flat homogeneous $2$-nondegenerate CR 5-manifolds, which were classified in \cite{FK-CR}. We  discuss their moduli space under two natural equivalence conditions. We describe spaces of germs of pre-Sasakian structures at a point in Proposition \ref{prop: germ equivalence class dimension}. Subsequently, we describe  spaces of geometrically distinct pre-K\"ahler structures, a coarser notion of equivalence,  discussed in Remark \ref{rem: coarse equivalence}, which are then used to find lower bounds on symmetry algebra dimensions.
\newtheorem*{thmD}{\bf Theorem \ref{thm: coarse moduli dim}}
\begin{thmD}
\emph{
The space of geometrically inequivalent pre-K\"{a}hler structures arising from any fixed $5$-dimensional $2$-nondegenerate locally homogeneous CR manifold contains a $2$-dimensional open dense subset. Each such pre-K\"{a}hler structure has nontrivial infinitesimal symmetries.
}
\end{thmD}
 Lastly, in \ref{sec:homog-pre-kahl} we conclude by showing that a  $2$-nondegenerate pre-K\"ahler complex surface is locally homogeneous if and only if it is flat.

\subsection{Conventions}
\label{sec:conventions}

In this article we will work locally over smooth  manifolds. Given a distribution $\scD\subset TM,$  its derived system is the distribution whose sheaf of sections is given by $\Gamma(\scD)+[\Gamma(\scD),\Gamma(\scD)]$ and, by abuse of notation, is denoted as $[\scD,\scD].$ Similarly, given two distributions $\scD_1$ and $\scD_2,$ we denote by $[\scD_1,\scD_2]$ the distribution whose sheaf of sections is   $\Gamma(\scD_1)+\Gamma(\scD_2)+[\Gamma(\scD_1),\Gamma(\scD_2)].$  We denote by $\Omega^k(M)$ and $\Omega^k(M,\fg)$ the sheaf of sections of  $\biw^k T^*M$ and $\biw^k T^*M\otimes\fg,$ respectively.   The symmetric product of  1-forms is denoted by their multiplication, e.g. for two 1-forms $\alpha$ and $\beta$ we define $\alpha\beta=\half(\alpha\otimes\beta+\beta\otimes\alpha)$ and $\alpha^k$ denotes the $k$th symmetric power of $\alpha.$ The span of vector fields  $v_1,\dots,v_k\in\Gamma(T M)$ over $C^\infty(M)$ is denoted by $\langle v_1,\dots,v_k\rangle$ and the algebraic ideal generated by 1-forms $\omega^0,\dots,\omega^n\in\Omega^1(M)$ is denoted by $\{\omega^0,\dots,\omega^n\}.$

  Given a principal bundle $\mu\colon\cC\to M,$ let $(\omega^i,\omega^i_j)$ be a coframe on $\cC$ such that $\omega^i$'s are semi-basic with respect to  $\mu\colon\cC\to M,$ i.e. $\omega^i(V)=0$ for all $V\in\ker\mu_*.$  We denote the frame associated to this coframe by $(\partial_{\omega^i},\partial_{\omega^i_j}).$ Moreover, given  a  function $f$ on $\cC,$ its \emph{coframe derivatives}  are defined by
  \[f_{;i}=\partial_{\omega^i}\im\exd f.\]
 We use the summation convention over repeated indices. Lastly, given a complex-valued     function (or 1-form)  $f=f_1+\ri f_2$ its  real and imaginary parts are denoted by $\Re f=f_1$ and $\Im f=f_2.$

\section{General theory for pre-K\"ahler and pre-Sasakian structures}\label{sec: CR pre-Kahler corresp.} 
In this section we define pre-K\"ahler and pre-Sasakian stuctures, as well as  the notions of  straightenability,  holomorphic degeneracy, Freeman filtration, and $k$-nondegeneracy. We show that pre-K\"ahler structures encode local submanifold geometry of pseudo-K\"ahler structures, and that, locally, they always arise this way. We formalize the pre-K\"ahler--Sasakian correspondence and show that, assuming real analyticity, the symmetry algebra of finitely-nondegenerate pre-K\"ahler structures is finite-dimensional. We also associate a pre-K\"ahler structure to any pre-Sasakian structure using a pre-symplectification construction, which extends the well-known metric cone construction in Sasakian geometry. We finish by characterizing the low order  jets of potential functions for uniformly finitely-nondegenerate pre-K\"ahler structures.

\subsection{Pre-K\"{a}hler, pre-Sasakian, and straightenable structures} \label{sec: terminology section}

(Pseudo-)K\"{a}hler structures on a complex manifold with almost complex structure $J$ are characterized by a real closed symplectic $(1,1)$-form $\omega$, which determines a (pseudo-)Riemannian metric
\begin{equation}\label{Kahler compatibility}
g(v,w)=\omega(v,Jw).
\end{equation}
This naturally generalizes by removing the assumption that $\omega$ be symplectic. In other words, using the term \emph{pre-symplectic}  as in \cite[Chapter 3, Section 7]{libermann2012symplectic}, we assume $\omega$ is a real pre-symplectic $(1,1)$-form. We call such structures \emph{pre-K\"{a}hler}.
\begin{definition}\label{pre-Kahler def}
A \emph{pre-K\"{a}hler structure} on a complex manifold $M$ with complex structure operator $J$ is a real closed differential form $\omega$ on $M$ of type $(1,1)$. The form $\omega$ is its \emph{pre-K\"{a}hler form}, and its \emph{associated (possibly degenerate) (pseudo-)Riemannian metric} is $g(v,w):=\omega(v,Jw)$. Its \emph{associated pre-Hermitian structure} is $h:=g-i\omega$.
\end{definition}

Complex submanifolds in pseudo-K\"{a}hler manifolds not only serve as natural examples of pre-K\"{a}hler manifolds, but every pre-K\"{a}hler structure arises locally in this way.

\begin{proposition}\label{prop: submanifold}
For any $2n$-dimensional pre-K\"{a}hler manifold $(M,\omega,J)$ and $p\in M$ at which the rank of $\omega$ is $2r$, there exists a neighborhood $U$ of $p$ that can be embedded as a complex submanifold into a $2(2n-r)$-dimensional pseudo-K\"{a}hler manifold such that $\omega$ coincides with the pullback of the pseudo-K\"{a}hler form along the embedding.
\end{proposition}
\begin{proof}
Let $U^\prime\ni p$ be a neighborhood of $p$ identified with local complex coordinates $(z^1,\ldots, z^n)$ as $U^\prime\subset \mathbb{C}^n$ such that $p=(0,\ldots,0)$ and $\frac{\partial}{\partial z^j}$ are in the kernel of $\omega$ at $0$ for all $j>r$. Since $\omega$ is a rank $2r$ type $(1,1)$ form,  at $0$ it takes the form $\left.\omega\right|_0=\left.\operatorname{i}\sum_{j,k=1}^rH_{j\bar{k}}\exd z^j\wedge \overline{\exd z^k}\right|_0$ for some nondegenerate $r\times r$ Hermitian matrix $(H_{j\bar{k}})$. Hence the pre-symplectic $(1,1)$-form $\widetilde{\omega}=\omega+\sum_{j=1}^{n-r}\Re(\operatorname{i}\exd z^{r+j}\wedge \overline{\exd Z^j})$ on $U\times \mathbb{C}^{n-r}$ expressed using coordinates $(Z^1,\ldots, Z^{n-r})\in \mathbb{C}^{n-r}$ is symplectic in some neighborhood $U\times \mathbb{C}^{n-r}\subset U^\prime\times \mathbb{C}^{n-r}$ of $0$ because its value at $0$ is nondegenerate. Along the embedding $z\mapsto (z,0)$ of $U$ into $U\times \mathbb{C}^{n-r}$, the form $\widetilde{\omega}$ pulls back to $\omega$.
\end{proof}

Many properties of (pseudo-)K\"{a}hler structures promote to the pre-K\"{a}hler setting as they are not inherently related to K\"{a}hler forms' nondegeneracy. The form $\omega$ being type $(1,1)$ is equivalent to the compatibility condition $\omega(\cdot,\cdot)=\omega(J\cdot, J\cdot)$, and its being closed avails us of standard $\partial\bar\partial$ lemmas leading to local expressions of $\omega$ in terms of potentials, described further below. Of course we must now allow the associated metric \eqref{Kahler compatibility} to be degenerate, as there is the identity for kernels $\operatorname{ker} g=\operatorname{ker}\omega$, and thus many techniques from Riemannian geometry become unavailable.

A well-known class of closed (1,1)-forms are  the Ricci 2-form of (pseudo-)K\"ahler metrics. Indeed, one can show that  any pre-K\"ahler structure locally arises as the Ricci 2-form of some (pseudo-)K\"ahler metric. 
To show this, recall the local $\overline{\partial}$-Poincar\'{e} Lemma (and more directly, its $\partial\overline{\partial}$ Lemma corollary \cite{moroianu2007lectures}), stating that for a closed $(1,1)$-form $\omega$ on $M$ and any $p\in M$, there is a neighborhood $U\subset M$ of $p$ on which $\omega$ is $\partial\overline{\partial}$ exact, meaning there exists a function $\rho:U\to \mathbb{R}$ such that 
\begin{equation}\label{potential identity}
\partial\overline{\partial}\rho =\frac{\operatorname{i}}{2}\left.\omega\right|_{U}.
\end{equation}
When $(\omega,J)$ is K\"{a}hler, such $\rho$ are called \emph{potentials} of $\left.\omega\right|_{U}$, and we extend the same terminology to the pre-K\"{a}hler setting. 
  \begin{proposition}\label{lem:pre-kahler-Kahler-Ricci-cohom}
    There is a local one-to-one correspondence between pre-K\"ahler structures and equivalence classes of K\"ahler metrics with the same Ricci 2-form.
  \end{proposition}
  \begin{proof}
Let $\omega$ be the closed (1,1)-form defining a pre-K\"{a}hler structure on $M$. 
    Requiring that $\omega$ be the Ricci 2-form of a K\"ahler metric $g=\partial\ov\partial\varphi,$ for some potential function $\varphi,$ i.e. $g_{ij}=\partial_i\ov\partial_j\varphi,$ amounts to the local solvability of the equation
\[
-\ri\partial\ov\partial\log\det(\partial\ov\partial \varphi)=\omega,
\]
which is equivalent to $-\ri\partial\ov\partial\log\det(g)=2\ri\partial\ov\partial \rho$ and  implies $ \partial\ov\partial(\log\det(g)+2\rho)=0.$
    It follows that $\log\det(g)=-2\rho+ f,$ for a pluriharmonic function $f$ on $M.$ Thus, the desired K\"ahler potential $\varphi$ satisfies the Monge-Amp\`ere equation $\det(\partial\ov\partial\varphi)=e^{-2\rho+f}.$
    Assuming real analyticity and an appropriate noncharacteristic Cauchy data, or that $g$ is positive definite, such Monge–Amp\`ere equations are locally solvable.
  \end{proof}
Accordingly, pre-K\"{a}hler structures correspond to equivalence classes of K\"{a}hler metrics, where equivalence between $g$ and $\wt{g}$ is given by existence of a pluriharmonic $f$ satisfying
\[
\det \wt g=e^f\det g.
\]

Throughout this article we additionally assume that the rank of $\omega$ is constant, a property that holds in general almost everywhere. This assumption is even sometimes given within the definition of pre-symplectic, e.g. \cite{grabowska2023reductions} and \cite[Chapter 3, Remark 7.3]{libermann2012symplectic}, and it allows us to apply the following lemma.

\begin{lemma}\label{lemma: integrable kernel lemma}
The kernel of a constant rank closed $(1,1)$ form $\omega$ on a pre-K\"{a}hler manifold $(M,\omega,J)$ is integrable and locally generates a foliation by complex submanifolds. 
 \end{lemma}
 \begin{proof}
For two vector fields $X_1,X_2$ in the kernel of $\omega$, applying $X_j \im\omega = 0$ a few times with $\exd \omega = 0$ yields $0 = \exd \omega(X_1,X_2,Y) = -\omega([X_1,X_2],Y)$ for all vector fields $Y$, and hence the kernel is Frobenius integrable. It is invariant under $J$ because $\omega(\cdot,\cdot) = \omega(J\cdot,J\cdot)$.
\end{proof}

In sufficiently small neighborhoods of any point on $M$, the foliation's  leaf space has a canonical smooth manifold structure. Understanding when the pre-K\"{a}hler structure on $M$ descends to a pre-K\"{a}hler structure on this leaf space is essential in our study, which motivates the following definitions.

\begin{definition}[straightenable]\label{partially straightenable}
A pre-K\"{a}hler manifold $(M,\omega,J)$ is \emph{partially straightenable} around a point $p\in M$ if there is 
\begin{itemize}
\item a neighborhood $U\subset M$  of $p$,  
\item a pre-K\"{a}hler manifold $(U^\prime,\omega^\prime, J^\prime)$ with $\dim(U^\prime)<\dim(M)$, and
\item a complex manifold $M^\prime$
\end{itemize}
 such that $(U,\omega,J)$  is equivalent the complex manifold $U^\prime\times M^\prime$ with pre-K\"{a}hler form defined by trivially extending $\omega^\prime$ to $T(U^\prime\times M^\prime)$, that is,  setting $X\im\omega=0$ for all $X\in \{0\}\times \Gamma(TM^\prime)$.
The structure on $(M,\omega,J)$ is \emph{totally straightenable} around a point $p\in M$, if additionally $(U^\prime,\omega^\prime, J^\prime)$ is (pseudo-)K\"{a}hler.
The structure on $(M,\omega,J)$ is \emph{non-straightenable} if it is not partially straightenable.
\end{definition}
This terminology alludes to an analogous concept of biholomorphic straightening developed for CR geometry in \cite{freeman1977local}.
\begin{example}[Non-straightenable]\label{ex: light cone potential}
  Consider the functions
  \begin{subequations}\label{eq:rho-exa}
    \begin{gather} 
        \rho_a(z)=2(\Re(z^1)+1)^a(\Re(z^2)+1)^{1-a},\label{rho-Doubrov-example}\\
        \rho(z)=\frac{|z^1|^2+\Re((z^1)^2\overline{z^2})}{1-|z^2|^2},\label{rho-flat-example} 
          \end{gather}
         \end{subequations}
  on $M:=\{z=(z^1,z^2)\in\mathbb{C}^2\,|\,|z^2|<1,\,\Re(z^1)|<1\}$, and the pre-symplectic forms 
\begin{equation}\label{flat preKahler form example}
\begin{aligned}
\omega_a = \frac{\operatorname{i}}{2}\partial \overline{\partial}\rho_a
\quad\mbox{ and }\quad
 \omega=\frac{\operatorname{i}}{2}\partial \overline{\partial}\rho
\end{aligned}
\end{equation}
on $M$ generated from $\rho$ by applying the Dolbeault operators
\[
\partial=\sum \frac{\partial}{\partial z^j}\exd z^j
\quad\mbox{ and }\quad
\overline{\partial}=\sum \frac{\partial}{\partial \overline{z^j}}\exd \overline{z^j}.
\]
These structures are not K\"{a}hler, as one can compute
\[
\det
\left(
\begin{array}{cc}
\omega_a\left( \partial_{z^1},\partial_{\overline{z^1}}\right) & \omega_a\left( \partial_{z^1},\partial_{\overline{z^2}}\right)\\
\omega_a\left( \partial_{z^2},\partial_{\overline{z^1}}\right) & \omega_a\left( \partial_{z^2},\partial_{\overline{z^2}}\right)
\end{array}
\right)
=
\det
\left(
\begin{array}{cc}
\omega\left( \partial_{z^1},\partial_{\overline{z^1}}\right) & \omega\left( \partial_{z^1},\partial_{\overline{z^2}}\right)\\
\omega\left( \partial_{z^2},\partial_{\overline{z^1}}\right) & \omega\left( \partial_{z^2},\partial_{\overline{z^2}}\right)
\end{array}
\right)
=0.
\]
The ranks of $\omega_a$ and $\omega$ are 1 everywhere, and they define pre-K\"{a}ler structures as they are real pre-symplectic $(1,1)$ forms by construction. It turns out that all of these structures (except for the $a=0$ and $a=1$ cases) are (locally) non-straightenable everywhere, a fact easily established using properties developed in the sequel, specifically Proposition \ref{prop: holomorphically degenerate and straightenability}. 

This example has fundamental connections to CR geometry of the hypersurfaces $\{(w,z)\in\mathbb{C}\oplus\mathbb{C}^2 \,|\,\Re(w)=\rho_a(z)\}$, a relationship that is formalized in Theorem \ref{thm: correspondence theorem}. In particular, the hypersuface $\Re(w)=\rho(z)$ has been extensively studied in CR geometry, and it is interesting to note that while this hypersurface is locally equivalent to $\Re(w)=\rho_{1/2}(z)$ as CR manifolds, the pre-K\"{a}hler structures of $\rho$ and $\rho_{1/2}$ differ.  The formula for $\rho_a$  comes from the classification in \cite{DKR-affine} of affinely homogeneous surfaces in $\mathbb{R}^3$, as $\{\Re(w)=\rho_a\}\cap\mathbb{R}^3$  is a special class of such surfaces. We continue this example in Section \ref{sec:param-expr} (Examples \ref{exa:homog-surf-flat}, \ref{exa:homog-surf}, \ref{exa:special-homog-equiaaffine}).
\end{example}

We are going to encounter degenerate generalizations of several structures that commonly appear in K\"{a}hler geometry. For consistency and to clearly emphasize the parallels to established K\"{a}hler geometry, we refer to them all using a \emph{pre-} prefix. In particular we will encounter \emph{pre-contact}  and \emph{pre-Sasakian} structures.

\begin{definition}
Given a CR manifold $M$ with CR distribution $D\subset TM$ and almost complex structure operator $J:D\to D$, a \emph{(local) CR symmetry} is a (local) diffeomorphism whose differential preserves $D$ and commutes with $J$ (or, equivalently, whose differentials preserve the $\ri$-eigenspace bundle $T^{1,0}M=\{v\in\mathbb{C}D\,:\, Jv=\operatorname{i}v\}$ of $J$). An \emph{infinitesimal CR symmetry} is a vector field $X$ on an open subset of $M$ for which the sheaf of sections of $T^{1,0}M$ is invariant under $\operatorname{ad}_{X}$. When $M$ is hypersurface-type, we call $X$ \emph{transverse} if it is transverse to $D$ at every point.
\end{definition}
\begin{definition}[Pre-contact and pre-Sasakian]\label{def: Pre-contact and Pre-Sasakian}
A \emph{pre-contact structure} is a codimension 1 distribution on an odd-dimensional manifold. A \emph{pre-Sasakian structure} is a hypersurface-type CR structure together with a transverse  infinitesimal symmetry.
\end{definition}

\begin{remark}
The CR distribution on hypersurface-type CR structures is pre-contact in general, and moreover contact if and only if it is Levi nondegenerate.
A different definition of pre-Sasakian structure on $3$-manifolds is formulated in \cite{belgun2012metric}, so we stress that these are indeed not equivalent.
\end{remark}
We often restrict to considering only the generic case of pre-K\"{a}hler structures with constant rank pre-K\"{a}hler forms, all related pre-contact structures in this case will be of constant class in the sense of \cite[Chapter 5.3.2]{alekseevskij1991geometry} and are therefore locally described by Darboux's theorem for codimension 1 distributions. Similarly, the underlying CR structures of all pre-Sasakian structures that we encounter in this case will have constant rank Levi forms.

\subsection{Straightenability and $k$-nondegeneracy}\label{sec: Straightenability and k-nondegeneracy}
Straightenability is an important property to detect, as it allows descriptions of the local geometry to be reduced to lower dimensional non-straightenable cases, and, as we will show in \ref{sec: Symmetry reduction and pre-Sasakian structures}, it has an important role in describing which pre-K\"{a}hler structures have finite-dimensional symmetry algebras.  Characterizing non-straightenability gives rise to an integer-valued fundamental invariant \emph{the nondegeneracy order} of a pre-K\"ahler structure, which we describe in this section. The invariant and its definition are directly analogous to concepts developed for \emph{biholomorphic straightening} of CR structures in \cite{freeman1977local}.   Structures with nondegeneracy order $k$ will be termed $k$-nondegenerate, alluding to analogous invariants of CR structures.

In this paper, we develop only the notion of \emph{uniform} $k$-nondegeneracy (Definition \ref{def: nondegeneracy order}),
which is well-defined under regularity assumptions we will soon adopt, described by certain natural filtrations of the tangent bundles consisting of constant rank vector distributions. These regularity assumptions are similar to the setting of  \cite{freeman1977local}. There are more general concepts of $k$-nondegeneracy in CR geometry \cite[Section 11]{baouendi1999cr}, and similar formulations can be made for the pre-K\"{a}hler setting, which we outline in Remark \ref{k-nond. through the correspondence}. It is beyond our present aim, however, to discuss such generalizations in detail as it requires digression we prefer to avoid. The regularity assumptions we will soon adopt hold at almost every point on a general pre-K\"{a}hler manifold, and at points where they hold the two notions (\emph{uniform} $k$-nondegeneracy and the most general possible definition) coincide, which is an immediate consequence of \cite[Corollary 11.2.14]{baouendi1999cr}, Theorem \ref{thm: correspondence theorem}, and Proposition \ref{prop: leading terms of k nond def eqn lemma}.

The first of these \emph{regularity assumptions}  is that $\omega$ has constant rank, so its kernel is an integrable distribution in $TM$ by Lemma \ref{lemma: integrable kernel lemma}. Considering the foliation of the pre-K\"{a}hler manifold $M$ that the kernel generates, in a sufficiently small neighborhood of a point $p\in M,$ the leaf space of this foliation, $N,$  has a natural smooth manifold structure for which the quotient map from $M$ to $N$ is a smooth submersion. The germ of such $N$ near the leaf through $p$ is well-defined, and we refer to any such $N$ as the \emph{local leaf space at $p$} of the pre-K\"{a}hler structure (emphasizing uniqueness up to a natural local equivalence). The complex structure on $M$ does not necessarily descend to a complex structure on $N$, which is the essential obstruction to straightenability. At a point $p\in M$, we can consider $\operatorname{i}$ and $-\operatorname{i}$ eigenspaces $T^{1,0}_pM$ and $T^{0,1}_pM$  of the complex structure operator, and their image under the differential $q_*$ defines a splitting of $\mathbb{C}T_{q(p)}N$ into transverse subspaces related through conjugation by $\overline{q_*\left(T^{1,0}_pM\right)}=q_*\left(T^{0,1}_pM\right)$. For any smooth section $\sigma:N\to M$ of $q:M\to N$, the distributions $\Delta^{1,0}$ and $\Delta^{0,1}$ with fibers
\begin{equation}\label{complex structure eigenspaces}
\Delta^{1,0}_p=q_*\left(T^{1,0}_{\sigma(p)}M\right)
\quad\mbox{ and }\quad
\Delta^{0,1}_p=q_*\left(T^{0,1}_{\sigma(p)}M\right)
\quad\quad\forall\, p\in N,
\end{equation}
determine an almost complex structure on $N$, by taking these distributions to be the almost complex structure operator's $\operatorname{i}$ and $-\operatorname{i}$ eigenspaces, respectively. The distributions depend on $\sigma$ in general.

\begin{remark} For every section $\sigma:N\to M$ with $\sigma(N)\subset M$ a complex submanifold, the pull-backs $\sigma^*\omega$ and $\sigma^*J$ determine a (pseudo-)K\"{a}hler structure on $N$, and it is in this sense that \emph{pre-K\"{a}hler} is a fitting moniker for our general structures.
\end{remark}

Throughout the sequel, we label the (complexified) pre-symplectic kernel 
\begin{equation}\label{pre-symplectic kernel}
\mathcal{K}\oplus\overline{\mathcal{K}}:=\{v\in \mathbb{C}T_pM\,|\, v\im\omega_p=0,\,\forall\,p\in M\},
\end{equation}
reflecting the decomposition of this distribution into holomorphic $\mathcal{K}\subset T^{1,0}M$ and antiholomorphic $\overline{\mathcal{K}}\subset T^{0,1}M$ parts.

The almost complex structure on $N$ defined by \eqref{complex structure eigenspaces} is independent of $\sigma$ if and only if 
\begin{equation}\label{holomorphic and antiholomorphic bundles}
T^{1,0}M/q_*^{-1}(0)=\left(T^{1,0}M+\overline{\mathcal{K}}\right)/\left(\mathcal{K}\oplus\overline{\mathcal{K}}\right)
 \quad\mbox{ and }\quad
T^{0,1}M/q_*^{-1}(0)=\left(T^{0,1}M+\mathcal{K}\right)/\left(\mathcal{K}\oplus\overline{\mathcal{K}}\right)
\end{equation}
 are invariant under flows of vector fields in the kernel \eqref{pre-symplectic kernel}.   The infinitesimal expression of such invariance -- posed in terms of Lie brackets -- naturally leads us to consider the $\mathbb{C}$-linear map $v\mapsto \mathrm{ad}_v$ for each $p\in M$ from the fiber $\mathcal{K}_p$ to the space of antilinear endomorphisms on $T^{1,0}_pM/\mathcal{K}_p$ by
\begin{equation}\label{first KAO}
\mathrm{ad}_v(x+\mathcal{K}_p):=[V,\overline{X}]  \pmod{\mathcal{K}_p \oplus T^{0,1}_pM}
\quad\quad\forall\, v\in \mathcal{K}_p, x\in T^{1,0}_pM/\mathcal{K}_p,
\end{equation}
where $V\in \Gamma(\mathcal{K})$ and $X\in \Gamma(T^{1,0}M)$ are any vector fields satisfying $V(p)=v$ and $X(p)\equiv x \pmod{\mathcal{K}_p}$. It is straightforward to check that this definition does not depend on the choice of vector fields $X$ and $V$ extending $x$ and $v$.

\begin{lemma}\label{lemma: total straightenability lemma}
The induced almost complex structure on $N$ defined by \eqref{complex structure eigenspaces} is independent of the section $\sigma$ if and only if $M$ is locally totally straightenable, in which case $N$ caries a canonical (pseudo-)K\"{a}hler structure. Equivalently,  $M$ is locally totally straightenable if and only if $\mathrm{ad}_v=0$ for all $v\in \mathcal{K}$.
\end{lemma}
\begin{proof}
For the first statement, we work out only that local total straightenability implies section independence direction, as the converse is immediate. 

Set $r=\mathrm{rank}_{\mathbb{C}}\mathcal{K}$ and $2n=\dim_{\mathbb{R}}M$.
Since $\mathcal{K}\oplus \overline{\mathcal{K}}\cap TM$ is integrable and invariant under $J$, it generates a foliation by complex submanifolds of $M$.  More specifically, since both $\mathcal{K}$ and $\mathcal{K}\oplus \overline{\mathcal{K}}$ are integrable, we can apply \emph{the complex Frobenius theorem} (i.e., \cite[Theorem $1^\prime$]{nirenberg1958complex}) to conclude that for any point $p\in M$ there is a neighborhood $U\subset M$ of $p$ and diffeomorphism $\varphi:U\to U^\prime\times U^{\prime\prime}\subset\mathbb{C}^{r}\oplus \mathbb{R}^{2n-2r}$ such that leaves of the $\mathcal{K}\oplus \overline{\mathcal{K}}\cap TU$ foliation are the complex manifolds
\[
L_a:=\left\{p\in U\,\left|\,\varphi(p)=(z,a)\mbox{ for some } z\in U^\prime\subset\mathbb{C}^{r}\right.\right\}
\]
for each $(0,a)\in \{0\}\times U^{\prime\prime}= \varphi(U)\cap\left(\{0\}\oplus \mathbb{R}^{2n-2r}\right)$. In this chart, $V:=q(U)$ may be regarded as a subset in $\mathbb{R}^{2n-2r}$, and $V$ is naturally identified with images of sections $\sigma: V\to U$, which have the form $\sigma (x) = (f(x),x)$ for $f: V \to \mathbb{C}^{r}$. 

In the chart $(U,\varphi)$, invariance of \eqref{holomorphic and antiholomorphic bundles} under flows of vector fields in \eqref{pre-symplectic kernel} implies that the almost complex structure induced on the sections
\[
S_z:=\left\{p\in U\,\left|\,\varphi(p)=(z,a)\mbox{ for some } a\in U^{\prime\prime}\subset\mathbb{R}^{2n-2r}\right.\right\}
\]
are all related by translations, that is, $(z,a)\mapsto (z+t,a)$ defines a (local) biholomorphism between $S_z$ and $S_{z+t}$ for all (sufficiently small) $t\in \mathbb{C}$, identifying complex manifolds $U\cong L_0\times S_0$ in the obvious way.

Additionally, $\omega$ is always invariant under flows of vector fields in \eqref{pre-symplectic kernel} since its Lie derivatives with them vanish,
\[
\mathcal{L}_X\omega =\exd\circ \iota_X \omega+\iota_X\circ \exd\omega = 0
\quad\quad\forall\, X\in \Gamma\left(\mathcal{K}\oplus \overline{\mathcal{K}}\right)
\]
due to  $\exd\omega=0$. So $\omega$ descends to a canonical pre-symplectic form on $N$, that will moreover be symplectic since we quotiented out the original form's kernel. The pre-symplectic form on $L_0\times S_0$ given by trivially extending the form on $S_0$ induced by the natural identification $S_0\cong N$ coincides with $\omega$ on $U$ under the aforementioned identification $U\cong L_0\times S_0$.

The lemma's second if and only if statement now follows because $\mathrm{ad}_v=0$ for all $v\in \mathcal{K}_p$ is the infinitesimal characterization of  \eqref{holomorphic and antiholomorphic bundles}  being invariant under flows of vector fields in \eqref{pre-symplectic kernel}.
\end{proof}

If $\mathrm{ad}_v\neq 0$ for some $v$, the structure can still be partially straightenable, which is detectable by similar \emph{higher order} constructions. 
We are going to iteratively build levels of a filtration
\begin{equation}\label{freeman filtration on preKahler}
\mathcal{K}_{-1}=T^{1,0}M\supset\mathcal{K}=\mathcal{K}_0\supset\mathcal{K}_1\supset\cdots,
\end{equation}
and to proceed on each step we will need to assume that the level built in the preceding step is a regular vector distribution.  This is the required \emph{regularity assumption} mentioned at the beginning of this section.

To begin, label $\mathcal{K}_0:=\mathcal{K}$, and proceeding inductively, suppose we have defined nested distributions $\mathcal{K}_0\supset\mathcal{K}_1\supset\cdots\supset \mathcal{K}_{j-1}$. At a point $p\in M$, for a vector $v$ in the fiber $(\mathcal{K}_{j-1})_p$ define $\mathrm{ad}_v:{T}^{0,1}_pM\to(\mathcal{K}_{j-2})_p/(\mathcal{K}_{j-1})_p $ by
\begin{equation}\label{general KAO}
\mathrm{ad}_v(x+\mathcal{K}_p):=[V,\overline{X}]  \pmod{\mathcal{K}_p \oplus T^{0,1}_pM}
\quad\quad\forall\, v\in \mathcal{K}_p, x\in T^{1,0}_pM/\mathcal{K}_p,
\end{equation}
where $V\in \Gamma(\mathcal{K}_{j-1})$ and $X\in \Gamma(T^{1,0}M)$ are any vector fields satisfying $V(p)=v$ and $X(p)=x$. Verifying independence from the chosen sections $V$ and $X$ is similar to the exercise for \eqref{first KAO}. Define 
\[
\mathcal{K}_j=\{v\in \mathcal{K}_{j-1}\,\mathrm{ad}_v=0\}.
\]

We call this filtration the \emph{Freeman filtration} for its similarity to the filtration on CR structures introduced by Freeman in \cite{freeman1977local}.

\begin{definition}\label{def: nondegeneracy order} A pre-K\"{a}hler manifold $(M,\omega,J)$ \emph{admits a Freeman filtration} if each subset $\mathcal{K}_j\subset\mathbb{C}TM$ computed from the iterative construction above is a regular vector distribution.

The \emph{nondegeneracy order at $p\in M$} of a pre-K\"{a}hler structure $(M,\omega, J)$ admitting a Freeman filtration $\mathcal{K}_{-1}=T^{1,0}M\supset \mathcal{K}_0\supset\mathcal{K}_1\supset\cdots$ in a neighborhood of $p$ is $k$ if $\mathcal{K}_{k-2}\neq\mathcal{K}_{k-1}=0$ and $\infty$ if no such $k$ exists. If the nondegeneracy order at $p$ is $k<\infty$ then the pre-K\"{a}hler structure is \emph{uniformly $k$-nondegenerate} at $p$, or more generally we will call it \emph{uniformly finitely-nondegenerate} when specifying $k$ is unnecessary. If the nondegeneracy order is $\infty$ then the pre-K\"{a}hler structure is called \emph{holomorphically degenerate}  at $p$.
\end{definition}
The latter term, \emph{holomorphically degenerate}, again alludes to an analogous definition in CR geometry, and it characterizes the partial straightenability analogue of Lemma \ref{lemma: total straightenability lemma}.

\begin{proposition}\label{prop: holomorphically degenerate and straightenability}
A pre-K\"{a}hler manifold $(M,\omega,J)$ that admits a Freeman filtration is locally partially straightenable around a point $p\in M$ if and only if it is holomorphically degenerate at $p$.
\end{proposition}
\begin{proof}
Supposing $(M,\omega,J)$ is locally partially straightenable, let $U\subset M$ be a neighborhood with the product structure $U=\tilde{U}\times U^\prime$ described in Definition \ref{partially straightenable} with $\tilde{U}$ also pre-K\"{a}hler and $U^\prime$ a complex manifold, that is, $\{0\}\times TU^\prime$ is in the kernel of $\omega$. Hence, one has $T^{1,0}U^\prime\subset \mathcal{K}$, and, moreover,  $T^{1,0}U^\prime\subset \mathcal{K}_j$ for all $j,$ because vector fields in $\{0\}\times TU^\prime$ commute with vector fields in $T\tilde{U}\times \{0\}$.

Conversely, suppose  $(M,\omega,J)$ is holomorphically degenerate and let $\mu$ be the smallest integer for which $\mathcal{K}_{\mu-1}=\mathcal{K}_\mu\neq0$. The distributions $\mathcal{K}_\mu$ and $\mathcal{K}_\mu\oplus \overline{\mathcal{K}_\mu}$ are integrable, so the complex Frobenius theorem  \cite[Theorem $1^\prime$]{nirenberg1958complex} realizes sufficiently small neighborhoods $U\subset M$ as open sets in $\mathbb{C}^{\mathrm{rank}\mathcal{K}_\mu}\oplus\mathbb{R}^{2n-2\mathrm{rank}\mathcal{K}_\mu}$, where the $\mathbb{C}^{\mathrm{rank}\mathcal{K}_\mu}$ coordinates parameterize integral manifolds of $\mathcal{K}_\mu\oplus \overline{\mathcal{K}_\mu}\cap T_M$ with their natural complex structure.

A complex structure on the $\mathbb{R}^{2n-2\mathrm{rank}\mathcal{K}_\mu}$ factor also appears because
\[
[\mathcal{K}_\mu,T^{0,1}M]\equiv 0
\pmod{\mathcal{K}_\mu \oplus T^{0,1}M},
\]
due to $\mathcal{K}_{\mu-1}=\mathcal{K}_\mu$, following essentially the same arguments outlined in the proof of Lemma \ref{pre-symplectic kernel}. This gives sufficiently small neighborhoods $U\subset M$ a product structure $U=\tilde{U}\times U^\prime$ where leaves of the $\mathcal{K}_\mu\oplus \overline{\mathcal{K}_\mu}\cap T_M$ foliation are the submanifolds $\{\mathrm{pt}\}\times U^\prime$. Hence, $\omega$ vanishes on $\{\mathrm{0}\}\times TU^\prime$, and $U$ has the partially straightenable structure described in Definition \ref{partially straightenable}.
\end{proof}

These formulations of $k$-nondegeneracy can be alternatively expressed in terms of complex coframes consisting of elements in $\Omega^{1,0}(M)$ and $\Omega^{0,1}(M)$ which will be our notation for type (1,0) and (0,1) forms. Indeed, suppose $(M,\omega,J)$ admits a Freeman filtration $\mathcal{K}_{-1}=T^{1,0}M\supset\mathcal{K}=\mathcal{K}_0\supset\mathcal{K}_1\supset\cdots$ with dimensions $d_j=\dim_\mathbb{C}(\mathcal{K}_{j})-\dim_\mathbb{C}(\mathcal{K}_{j+1})$, and let $\mu$ be the smallest integer for which $d_\mu = 0$. Hence, it follows that
\[
n = \dim_{\mathbb{C}}(M) = \sum_{j=-1}^{\mu-1}d_j.
\]
Let $X_1,\ldots, X_{n}\in \mathcal{K}_{-1}$ be linearly independent vector fields adapted to the Freeman filtration in the sense that
\[
\mathcal{K}_j=\spn_{\mathbb{C}}\left\{X_{\ell}\,\left|\, \ell>\sum_{t=-1}^{j-1}d_j\right.\right\}
\quad\quad\forall j\geq0.
\]

Let $\theta^1,\ldots,\theta^n\in \Omega^{1,0}(M)$ be the $(0,1)$-forms dual to $X_1,\ldots, X_n$, and extend $\{X_j\}$ and $\{\theta^k\}$ to frames and coframes $(X_1,\ldots,X_{2n})$ and $(\theta^1,\ldots, \theta^2)$ of $\mathbb{C}TM$ and $\mathbb{C}T^*M$ by the rules
\[
X_{j+n}:=\overline{X_j}
\quad\mbox{ and }\quad
\theta^{j+n}:=\overline{\theta^j}
\quad\quad\forall\, j=1,\ldots, n.
\]
Based on the relationship between the structure functions $c^{i}_{jk}$ given by 
\begin{equation}\label{coframe struct eqns}
\exd\theta^i=\sum_{j,k} c^i_{jk} \theta^j\wedge\theta^k
\quad\mbox{ with }\quad
c^i_{jk}=-c^i_{kj}
\end{equation}
and $\Gamma^i_{jk}$ given by $[X_j,X_k]=\sum_{i}\Gamma^i_{jk}X_i$ with $\Gamma^i_{jk}=-\Gamma^i_{kj}$ arising from 
\[
c^{i}_{jk}=\exd\theta^i(X_j,X_k)=-\theta^i([X_j,X_k])=-\Gamma^i_{jk},
\]
the operators in \eqref{general KAO} are represented by special components from the $c^{i}_{jk}$ array.

\begin{proposition}
A pre-K\"{a}hler manifold $(M,\omega,J)$ is uniformly $k$-nondegenerate for some $k<\infty$ if and only if it admits a complex coframe $\theta^1,\ldots, \theta^{2n}$ with structure functions $c^{i}_{jk}$ given by \eqref{coframe struct eqns} satisfying the following properties:
\begin{itemize}
\item $\theta^1,\ldots,\theta^n\in \Omega^{1,0}(M)$ and $\theta^{n+j}=\overline{\theta^j}$ for all $j=1,\ldots,n$.
\item Setting $d_{-1}:=\tfrac{1}{2}\mathrm{rank}(\omega)$, the kernel of $\omega$ is the annihilator of $\{\theta^{j},\theta^{j+n}\,|\,j\leq d_{-1}\}$, that is,
\[
\mathrm{ker}\omega = \left\{\theta^{j},\theta^{j+n}\,\left|\, j\leq d_{-1} \right.\right\}^\perp.
\]
\item There exists a sequence of integers $d_0,\ldots, d_{k-2}$ such that $n=\sum_{j=-1}^{k-2}d_j$ with two properties, which we formulate with respect to a partition mapping 
\[
p(j):=\min \left\{s\,\left|\, j < \sum_{t=1}^{s}d_{t-2}\right.\right\} 
\quad\quad\forall \,j=1,\ldots, n
\]
partitioning the indices $1,\ldots, n$ into the $k$ level sets $p^{-1}(1)$,\ldots, $p^{-1}(k)$ of $p$:
\begin{enumerate}
\item For all $(i,j,k)$ with $i,j\leq n<k$ and $p(i)+1<p(j)$, we require $c^i_{jk}=0$.
\item Label $d^1:=0$, and for each $1<\eta\leq k$, label $d^\eta:=\sum_{t=1}^{\eta-1} d_{t-2}$. For such $\eta$ and each integer $d^\eta<j\leq d^\eta+d_{\eta-2}$, consider the $d_{\eta-3}\times d_{-1}$ matrix $T_j$ whose $(\alpha,\beta)$ component is 
\[
(T_j)_{\alpha,\beta}=c^{\alpha+d^{\eta-1}}_{j\beta}
\quad\quad\forall\, 1\leq \alpha\leq d_{\eta-3},\,1\leq \beta\leq d_{-1}.
\]
For every $1<\eta\leq k$, the matrices $T_{d^\eta+1},\ldots, T_{d^\eta+d_{\eta-2}}$ are all linearly independent.
\end{enumerate}
\end{itemize}
\end{proposition}

\subsection{Pre-K\"{a}hler--Sasakian correspondence}\label{sec: Symmetry reduction and pre-Sasakian structures} 

The  constructions in \ref{sec: Straightenability and k-nondegeneracy} arise from translating developments in CR geometry to the pre-K\"{a}hler setting using a relationship between pre-K\"{a}hler structures and pre-Sasakian structures. In this section we establish a one-to-one correspondence between these structures, and note some immediate implications for pre-K\"{a}hler symmetries, which follow  from known results in CR geometry.

To build the first direction of this correspondence, suppose we have a pre-K\"{a}hler manifold $(M,\omega,J)$, and let $\rho$ be a locally defined potential of $\omega$,  as in \eqref{potential identity}.
Of course $\rho$ is only defined modulo real-valued functions in the kernel of $\partial\overline{\partial}$, which is exactly the family of real parts of holomorphic functions on $U$, i.e. pluri-harmonic functions\footnote{See  \cite[Section 3]{poincare1898proprietes} as an interesting source where this characterization is worked out.}. That is, any other potential $\rho^\prime$ of $\omega$ on $U$ has the form
\begin{equation}\label{potentials family}
\rho^\prime = \rho+\Re(f)
\end{equation}
for some holomorphic function $f$ on $U$. 

Introducing a new complex variable $w=u+\operatorname{i}v$, for each potential $\rho$ the function $u-\rho$ on $\mathbb{C}\times U$ cuts out a real hypersurface
\begin{equation}\label{associated real hypersurface}
M_\rho:=\{(w,z)\in \mathbb{C}\times U\,|\,\Re(w)=\rho(z)\}
\end{equation}
carrying the infinitesimal symmetry  $\frac{\partial}{\partial v}=\mathrm{Re}(2\operatorname{i}\frac{\partial}{\partial w})$.
Let us denote this symmetry as
\begin{equation}\label{associated real hypersurface symmetry}
X_\rho:=\left.\frac{\partial}{\partial v}\right|_{M_\rho}\in \Gamma(TM_\rho).
\end{equation}
Since $X_\rho$ is indeed transverse to the natural CR distribution on $M_\rho$, the pair $(M_\rho,X_\rho)$ are examples of pre-Sasakian manifolds. If $(M,\omega,J)$ is K\"{a}hler then $M_\rho$ is a strongly pseudo-convex CR hypersurface, and $(M_\rho, X_\rho)$ defines a Sasakian structure in a standard way.

For any other potential $\rho^\prime$, writing it in the form \eqref{potentials family} yields
\[
M_{\rho^\prime}=\{(w,z)\in \mathbb{C}\times U\,|\,\Re(w+f(z))=\rho(z)\},
\]
which shows that $M_{\rho^\prime}$ is the image of $M_\rho$ under the biholomorpic transformation $(w,z)\mapsto (w-f(z),z)$, defining a CR equivalence between the hypersurfaces.  Since $\frac{\partial}{\partial v}$ is also invariant under this transformation, it carries $X_\rho$ to $X_{\rho^\prime}$. Let us summarize these observations in a lemma.
\begin{lemma}\label{prekahler to presasakian}
For a pre-K\"{a}hler structure $(M,\omega,J)$, a point $p\in M$, and complex variable $w=u+\operatorname{i}v$, let $\rho$ be a potential of $\omega$ in a neighborhood $U$ of $p$, and let $M_\rho$ be the associated real hypersurface given by \eqref{associated real hypersurface} endowed with the symmetry $X_\rho$ from \eqref{associated real hypersurface symmetry}. The germ at $(\rho(p),p)\in \mathbb{C}\times U$ of the pre-Sasakian structure $(M_\rho,X_\rho)$ does not depend on $\rho$, and is therefore determined by the germ at $p$ of $(M,\omega,J)$.
\end{lemma}

Now for the converse direction, let $(S,X)$ be a pre-Sasakian structure defined by an abstract hypersurface-type $(2n+1)$-dimensional CR manifold $S$ equipped with an infinitesimal symmetry $X\in \Gamma(TS)$ such that $X$ is everywhere transverse to the CR distribution. For a point $p\in S$, take a sufficiently small neighborhood $V\subset S$ such that the leaf space $M$ of integral curves of $X$ in  $V$ has a smooth structure.  Let $\tilde \theta$ be the $1$-form on $V\subset S$ annihilating its CR distribution $\tilde\theta^\perp$ satisfying $X \im \tilde\theta=1$, which is unique. Since $X$ is a symmetry of  $\tilde\theta^\perp$, there is some $c\in C^\infty(S)$ such that $\mathcal{L}_X\tilde\theta=c\tilde\theta$, which implies $\mathcal{L}_X\tilde\theta=0$ because
\[
c=c\tilde\theta(X)=\mathcal{L}_X\tilde\theta(X)=[\exd \tilde\theta(X)+X\im  \exd\tilde\theta](X)=[X \im \exd\tilde\theta](X)=\exd\tilde\theta(X,X)=0.
\]

Therefore,  $\tilde\theta$ is invariant under flows of $X$ and descends to a $1$-form $\theta$ on $M$. By assumption, the CR manifold's almost complex structure operator is also preserved by $X$, and therefore descends to an almost complex structure $J$ on $M$, equipping $M$ with a complex manifold structure. The pre-symplectic form $\omega:=\exd\theta$ on $M$ is closed because it is exact. It is type $(1,1)$ with respect to $J$ because its pullback, $\exd \tilde\theta,$ is type $(1,1)$ with respect to the complex structure inducing $J$.

Thus, from a pre-Sasakian structure we obtain a canonically associated pre-K\"{a}hler structure $(M,\omega,J)$. An important question that remains is whether this construction is an inverse of the pre-K\"{a}hler  to pre-Sasakian construction in Lemma \ref{prekahler to presasakian}. To answer that, let us consider embeddings of the abstract CR structure on $S$ into $\mathbb{C}^{n+1}$ given by \cite[Theorem II.1, Section 1]{baouendi1985cr}, where it is shown that any point $p\in S$ is contained in a neighborhood $V\subset S$ that can be embedded into $\iota(V)\subset\mathbb{C}\oplus \mathbb{C}^n$ as the graph
\begin{equation}\label{BRT embedding}
\iota(V) = \{(w,z)\,|\, \Re(w)=\rho(z),\, z\in U\}
\end{equation}
for some $U\subset\mathbb{C}^n$ and some real-valued function $\rho:U\subset \mathbb{C}^n\to \mathbb{R}$, such that 
\[
\iota_*X=\Re\left(-\operatorname{i}\frac{\partial}{\partial w}\right)
\quad\mbox{ and }\quad
\iota(p)=0.
\] 
Taking $\tilde\theta = \exd v+\operatorname{i}(\partial-\overline{\partial})\rho$ indeed annihilates the CR distribution on \eqref{BRT embedding} and satisfies $\tilde\theta(\tfrac{\partial}{\partial v})=1$. The leaf space $M$ is parameterized by the $z$ coordinates, with respect to which we get $\theta=\operatorname{i}(\partial-\overline{\partial})\rho$.
Since
\[
d\theta=(\partial+\overline{\partial})\theta = -2\operatorname{i}\partial\overline{\partial}\rho,
\]
taking $\omega=d\theta$ satisfies \eqref{potential identity}, from which it is clear that the two constructions revert each other. So we have established the following.

\begin{theorem}\label{thm: correspondence theorem}
There is a one-to-one correspondence between pre-K\"{a}hler and pre-Sasakian structures given by the constructions in Lemma \ref{prekahler to presasakian}.
\end{theorem}

The correspondence in \eqref{thm: correspondence theorem} lifts to a  natural correspondence between symmetries of pre-K\"{a}hler structures and their associated pre-Sasakian structures. In turn, symmetries of pre-Sasakian structures of course embed into the symmetries of their underlying CR structures. The embedded subalgebra in the CR symmetry algebra can even be realized by holomorphic vector fields, that is, holomorphic  sections of $T^{(1,0)}\mathbb{C}^{n+1}$.
\begin{corollary}\label{equivalent embeddings corollary}
For a pre-Sasakian structure $(S,X)$, given local embeddings into $\mathbb{C}^{n+1}$ of the form \eqref{BRT embedding}, there exists a (local) biholomorphism of $\mathbb{C}^{n+1}$ carrying one embedded CR hypersurface to the other.
For a pre-K\"{a}hler structure $(M,\omega,J)$ with potential $\rho:U\subset M\to\mathbb{R}$ defined in a neighborhood $U$ of a point $p\in M$, the infinitesimal symmetry algebra of  $(M,\omega,J)$ at $p$ naturally embeds as a subalgebra into the holomorphic infinitesimal symmetry algebra at $(\rho(p),p)$ of the associated CR hypersurface $S_\rho=\{(w,z)\in \mathbb{C}\times U\,|\, \Re(w)=\rho(z)\}$, i.e., the subalgebra of germs at $(\rho(p),p)$ of vector fields on $\mathbb{C}\times U$ whose flows generate local biholomorphisms around $(\rho(p),p)$  leaving $S_\rho$ invariant.
\end{corollary}
\begin{proof}
Consider two such embeddings realized as graphs of functions $\rho,\rho^\prime:M\subset\mathbb{C}^n\to \mathbb{R}$ centering the same point in $S$ at the origin, namely with $\rho(0)=\rho^\prime(0)=0$. By Theorem \ref{thm: correspondence theorem}, the associated pre-K\"{a}hler structures given by $\omega=-2\operatorname{i}\partial\overline{\partial}\rho$ and $\omega^\prime=-2\operatorname{i}\partial\overline{\partial}\rho^\prime$  on $M\subset\mathbb{C}^n$ are locally equivalent at $0$, so there is a (local) biholomorphism $\varphi:M\to M$ with $\varphi(0)=0$ and $\varphi^*\omega^{\prime}=\omega$. Since $\varphi$ is holomorphic, $\varphi^*$ commutes with $\partial$ and $\overline{\partial}$, and hence
\[
-2\operatorname{i}\partial\overline{\partial}\varphi^*\rho^{\prime}=-2\operatorname{i} \varphi^*   \partial\overline{\partial}\rho^{\prime}=\varphi^* \omega^{\prime}=\omega.
\]
Since $\varphi^*\rho^{\prime}$ is a potential of $\omega$, it has the form \eqref{potentials family} for some holomorphic function $f$ on $M$, that is,
\[
\rho^{\prime}\left(\varphi^{-1}(z)\right)=\varphi^*\rho^{\prime}(z)=\rho(z)+\Re\left(f(z)\right).
\]
Therefore, the local biholomorphism $(w,z)\mapsto (w-f(z),\varphi(z))$ transforms the second embedding $ \{(w,z)\,|\, \Re(w)=\rho^\prime(z),\, z\in U\}$ into the first
\[
 \{(w,z)\,|\, \Re(w)=\rho(z),\, z\in U\}= \{(w,z)\,|\, \Re(w)=\rho^{\prime}\left(\varphi^{-1}(z)\right)-\Re\left(f(z)\right),\, z\in U\}.
\]

For the second statement, a (local) pre-K\"{a}hler symmetry is given by (local) biholomorphism $\varphi:U\subset M\to \varphi(U)\subset M$ preserving $\omega$, which transforms $\rho$ to another potential of $\omega$, and hence $\rho(z)=\rho\left(\varphi^{-1}(z)\right)-\Re(f(z))$ for some holomorphic function $f$ on $U\subset M$. The corresponding transformation $(w,z)\mapsto (w-f(z),\varphi(z))$ (locally) preserves $S$. This embedding of (local) Lie groups determines the embedding of their Lie algebras.
\end{proof}
\begin{remark}
For readers familiar with techniques in locally extending general analytic CR symmetries on real analytic submanifolds in $\mathbb{C}^N$ to local symmetries of the ambient complex space (e.g. \cite[Proposition 12.4.22]{baouendi1999cr}), it is notable that this last corollary's extension construction for lifted pre-K\"{a}hler symmetries is much simpler and even applies without restricting to the analytic category.
\end{remark}

The first statement in Corollary \ref{equivalent embeddings corollary} resolves a subtle question on inequivalent embeddings of the form \eqref{BRT embedding}. That is, while \cite[Theorem II.1, Section 1]{baouendi1985cr} gives existence of local embeddings of the form \eqref{BRT embedding}, they are far from unique. To our knowledge, the result of Corollary \ref{equivalent embeddings corollary} was previously known only in special cases, such as where the CR structure on $S$ is real analytic and Levi nondegenerate, in which case there are (local) coordinate transformations bringing both embeddings to the Chern--Moser normal form \cite{chern1974real}, after which the equipped (pseudo-)Sasakian symmetries can be aligned with an appropriate symmetry group action. 

The correspondence in Theorem \ref{thm: correspondence theorem} and Corollary \ref{equivalent embeddings corollary} allows the following important application of known results on holomorphic infinitesimal symmetries of CR hypersurfaces.
\begin{theorem}\label{finite-dimensionality theorem}
A real analytic pre-K\"{a}hler structure admitting a Freeman filtration has finite-dimensional infinitesimal symmetry algebras if and only if it is $k$-nondegenerate.
 \end{theorem}
 \begin{proof}
By \cite[Theorem 1.7]{stanton1996infinitesimal}, a real analytic CR hypersurface has an infinite-dimensional algebra of holomorphic infinitesimal symmetries if and only if it is holomorphically degenerate. Comparing definitions, it is easily seen that a pre-K\"{a}hler structure admitting a Freeman filtration is $k$-nondegenerate with $k<\infty$ if and only if its associated CR hypersurface is $k$-nondegenerate (and thereby not holomorphically degenerate). So $k$-nondegeneracy implies finite-dimensionality by Corollary \ref{equivalent embeddings corollary}.

Conversley, if $(M,\omega,J)$ is not $k$-nondegenerate at $p$ for some $k<\infty$ then, by Proposition \ref{prop: holomorphically degenerate and straightenability}, there is a neighborhood $U\subset M$ of $p$, a neighborhood $U^\prime\subset \mathbb{C}$, and another pre-K\"{a}hler manifold $(M^\prime,\omega^\prime,J^\prime)$ such that $(U,\omega,J)$ is equivalent to $M^\prime\times U^\prime$ equipped with a pre-symplectic form given by trivially extending $\omega^\prime$. Every (locally) holomorphic function $f:M^\prime\to U^\prime$ determines a different (local) symmetry $(p,z)\mapsto (p,z+f(p))$ of $M^\prime\times U^\prime$, so the infinitesimal symmetry algebra at a point $p\in U\subset M$ is infinite-dimensional.\end{proof}

 \begin{remark}\label{CR FF VS PK FF}
In \cite{freeman1977local}, a filtration is introduced on CR manifolds analogous to \eqref{freeman filtration on preKahler}. On a CR manifold $(S,D,J)$,  one takes $\tilde{\mathcal{K}}_{-1}:=\mathbb{C}TD\cap T^{1,0}N$, defines $\tilde{\mathcal{K}}_0$ to be the kernel of the Levi form in $\tilde{\mathcal{K}}_{-1}$, and then defines the rest of the sequence $\tilde{\mathcal{K}}_{-1}\supset\tilde{\mathcal{K}}_{0}\supset\tilde{\mathcal{K}}_{1}\supset\cdots$ in terms of $\tilde{\mathcal{K}}_{-1}$ and $\tilde{\mathcal{K}}_0$ using exactly the same formulas we used to define $\mathcal{K}_{-1}\supset\mathcal{K}_{0}\supset\mathcal{K}_{1}\supset\cdots$ in terms of $\mathcal{K}_{-1}$ and $\mathcal{K}_0$. Freeman's filtration is invariant under CR symmetries and in particular invariant under a pre-Sasakian structure's distinguished CR symmetry. Thus, Freeman's filtration on a pre-Sasakian structure descends to a filtration on the pre-K\"{a}hler structure associated to it via  the correspondence in Theorem \ref{thm: correspondence theorem}. The latter filtration exactly coincides with  filtration  \eqref{freeman filtration on preKahler} on the pre-K\"{a}hler structure. 
 \end{remark}
 
 \begin{remark}\label{k-nond. through the correspondence}
 A more general notion of $k$-nondegeneracy for CR structures is given in \cite{baouendi1999cr}, which is well-posed even without assuming the existence of the regular Freeman filtration. In light of Theorem \ref{thm: correspondence theorem}, one can naturally define an analogous property for pre-K\"{a}hler structures in terms of the CR structure underlying the pre-Sasakian structure associated to each point.
 \end{remark}
 
 \subsection{Pre-symplectification of pre-Sasakian structures}\label{sec: presymplectification section}
So far we have observed a correspondence between $(2n+1)$-dimensional pre-Sasakian structures and $2n$-dimensional pre-K\"{a}hler structures. For actual Sasakian and K\"{a}hler structures there is also a well-known natural construction of $(2n+2)$-dimensional K\"{a}hler structures from $(2n+1)$-dimensional Sasakian structures. The Sasakian structures' underlying CR distributions are contact and the higher-dimensional manifolds on which these K\"{a}hler structures are constructed are defined by  the \emph{symplectification of contact manifolds} described in  \cite[Appendix 4.E]{arnold2013mathematical}. This \emph{symplectification} naturally generalizes to a \emph{pre-symplectification}, e.g. see \cite{grabowska2023reductions}, producing $(2n+2)$-dimensional pre-symplectic manifolds from pre-contact manifolds. Let us review it.

Let $S$ be a $(2n+1)$-dimensional manifold with codimension $1$ distribution $D\subset TS$. Given a $1$-form ${\theta}$ that annihilates $D$ locally on some neighborhood $U\subset S$, all other such $1$-forms on $U$ have the form $f{\theta}$ for some nowhere zero function $f: U\subset S\to \mathbb{R}$. Let $\pi:\hat M_U\to U$ be the fiber bundle over $U$ whose fiber over a point $p\in U$ consists of all nonzero multiples of $\left.{\theta}\right|_p$. Hence, the $1$-forms annihilating $D$  on $U$ are sections of $\hat M_U$. There is a canonical 1-form $\hat{\theta}_U$ on $\hat M_U$ given by 
\[
\hat{\theta}_U(X):=t\theta\big(\pi_*(X)\big)
\quad\quad\forall\, X\in T_{t\theta}\hat M_U,\,0\neq t\in \mathbb{R},
\]
as this definition does not depend on the original choice of $\theta$. Given another neighborhood $U^\prime\subset S$ with a $1$-form annihilating $D$, both $\left.\hat M_U\right|_{U\cap U^\prime}$ and $\left.\hat M_{U^\prime}\right|_{U\cap U^\prime}$ are naturally identified  with $\hat M_{U\cap U^\prime}$ by construction, and  $\left.\hat\theta_U\right|_{\hat M_{U\cap U^\prime}}$ and $\left.\hat\theta_{U^\prime}\right|_{\hat M_{U\cap U^\prime}}$ are naturally identified as well in the obvious way. Using such identifications to patch together various $\hat M_U$ defined for different $U$, we get a canonical bundle $\pi:\hat M\to S$ equipped with a canonical $1$-form $\hat{\theta}$. This bundle is referred to as a \emph{pre-symplectic cover} in \cite{grabowska2023reductions}, and it retracts to a $2$-sheeted covering of $S$. 
The differential $\exd\hat{\theta}$ is closed and is therefore a pre-symplectic form on $\hat M$. We call $(\hat M,\exd\hat{\theta})$ the \emph{pre-symplectification} of $(S,D)$.

Since the pre-symplectification of pre-contact structures is well-posed in general, when the procedure is applied to a pre-Sasakian manifold we are left to consider if the pre-symplectification admits a canonical compatible almost complex structure.
\begin{remark}
The  underlying pre-contact distribution $D$ of a pre-Sasakian structure $(S,D,J,X)$ is co-oriented by $X$, which distinguishes a subbundle $\hat M_U^+=\{\phi\in \hat M_U\,|\, \phi(X)>0\}$. This subbundle deformation retracts onto a $1$-sheeted cover of $S$.
\end{remark}

Now, suppose $(S,D,J,X)$ is a pre-Sasakian manifold with  $J:D\to D$ the almost complex structure defining the underlying CR structure and $X$ the distinguished infinitesimal symmetry. And still let $(\hat M,\exd\hat{\theta})$ be the pre-symplectification of $(S,D)$. For $t>0$, let ${\theta}_t$ be the $1$-form on $S$ annihilating $D$ and satisfying ${\theta}_t(X)=t$. Such $1$-forms are easily found, even in explicit coordinates using  \cite[Theorem II.1, Section 1]{baouendi1985cr} to describe  $(S,D,J,X)$ in local coordinates of the form \eqref{BRT embedding}, and clearly $\theta_t$ is unique as scaling $\theta_t$ would break ${\theta}_t(X)=t$. In particular, if $(S,D,J,X)$  is locally realized as a hypersurface in $\mathbb{C}^{n+1}$ given by 
\begin{equation}\label{BRT embedding a}
\{(u+iv,z)\in \mathbb{C}\oplus\mathbb{C}^n\,|\,u=\rho(z)\}
\end{equation}
for some real-valued $\rho$ with $X=\frac{\partial}{\partial v}$, then one has
\begin{equation}\label{canonical pre-symplectification 1form}
{\theta}_t=\left.\operatorname{i}t(\partial-\bar\partial)(-u+\rho)\right|_{u=\rho(z)}=\left.\left(t\,dv+\operatorname{i}t(\partial-\bar\partial)\rho\right)\right|_{u=\rho(z)}
\quad\quad\forall\, t>0.
\end{equation}
Taking $(v,z,t)$ as local coordinates for $\hat M$, the canonical $1$-form on $\hat M$ is
\[
\hat{\theta}=t\,\pi^*(dv)+\operatorname{i}t\pi^*\left((\partial-\bar\partial)\rho\right),
\]
and hence the pre-symplectic form on $\hat M$ is
\begin{equation}\label{canonical pre-symplectification form}
\exd\hat{\theta} = dt\wedge\pi^* \left(dv +\operatorname{i}(\partial-\bar\partial)\rho\right)-2\operatorname{i}t\pi^*\left((\partial\bar\partial)\rho\right)=dt\wedge \pi^*\theta_1+t\pi^*\left(\exd\theta_1\right),
\end{equation}
where $\pi^*$ denotes the pullback by $\pi$.

Each $\theta_t$ defines a section of $\hat M$ and $\hat M$ is foliated by such sections. Therefore, we have a canonical connection on the bundle $\pi:\hat M\to S$ given by taking the tangent spaces of this foliation's leaves to be the connection's horizontal spaces in $T\hat M$. Let
\[
T\hat M=V\hat M\oplus H\hat M
\]
denote this splitting of $T\hat M$ into vertical and horizontal tangent spaces. Every vector field $W$ on $S$ lifts  to a horizontal  vector field $\hat{W}\in \Gamma(H\hat M)$ defined by 
\[
\pi_*\hat{W}_{x}=W_{\pi(x)}
\quad\mbox{ and }\quad
\hat{W}_{x}\in H_x\hat M
\quad\quad\forall x\in \hat M.
\]
Similarly, the distribution $D$ also lifts to a horizontal distribution $\hat D\subset H\hat M$ and $J$ lifts to an almost complex structure operator on $\hat D$ by defining $\hat{J}(\hat{V})=\widehat{J(V)}$ for all $V\in \Gamma(D)$. We can extend $\hat J$ to an almost complex structure on $\hat M$ by prescribing how the extension is defined on the lift of $X$. For this, let $\hat{Y}\in \Gamma(V\hat M)$ be the vertical vector field such that $\exd\hat\theta(\hat{Y},\hat{X})=1$ and define
\[
\hat{J}(\hat{X})=-\hat{Y}.
\]
In coordinates \eqref{canonical pre-symplectification form}, $\hat{Y}$ takes the simple form $\hat{Y}=\tfrac{\partial}{\partial t}$.

\begin{proposition}\label{prop: presymplectification prop}
Let $(S,D,J,X)$ be a pre-Sasakian manifold, and let $(\hat M,\exd\hat{\theta})$, $\hat{J}$, $\hat{Y}$ be defined as above. The pre-symplectic form $\exd\hat{\theta}$ is type $(1,1)$  with respect to $\hat{J}$. That is, $(\hat M,\exd\hat{\theta})$ is a pre-K\"{a}hler structure. If $(S,D,J,X)$ is uniformly $k$-nondegenerate then so is $(\hat M,\exd\hat{\theta})$.
\end{proposition}
\begin{proof}
Take a basis 
\begin{equation}\label{basis for presymplectification prop}
(\hat{e_1},\ldots, \hat{e_{2n}},\hat{X}_x,\hat{Y}_x)
\end{equation}
comprised of the lifts of some basis $(e_1,\ldots, e_{2n})$ of $D_{\pi(x)}\subset S$ to horizontal vectors $(\hat{e_1},\ldots, \hat{e_{2n}})$ in $H_x\hat M$ and the vector fields $\hat{X}$ and $\hat{Y}$, as described above. It is a straightforward computation to verify $\exd\hat{\theta}(v,w)=\exd\hat{\theta}(\hat Jv,\hat Jw)$ for all $v$ and $w$ in this basis. Uniform $k$-nondegeneracy promotes because the Freeman filtration on $(S,D,J,X)$  (Remark \ref{CR FF VS PK FF}) lifts to the Freeman filtration on $(\hat M,\exd\hat{\theta})$.
\end{proof}

\subsection{Low order jets of finitely-nondegenerate potentials}\label{sec: Low order jets of finitely-nondegenerate potentials}

Relating a pre-K\"{a}hler structure's  potentials to CR structures underlying its associated Sasakian manifolds, as discussed in \ref{sec: Symmetry reduction and pre-Sasakian structures}, one can apply  \cite[Corollary 11.2.14]{baouendi1999cr} to conclude the following.
\begin{proposition}\label{prop: low-order finite nondeg consequenes}
Let $\rho:U\subset\mathbb{C}^n\to \mathbb{R}$ be a local potential of a pre-K\"{a}hler structure expressed in some local coordinates $(z^1,\ldots, z^n)$ and let $\rho_z$ denote the length $n$ vector of partial derivatives $\frac{\partial \rho}{\partial z^j}$. If the structure is uniformly  $k$-nondegenerate then
\begin{equation}\label{general knond}
\mathbb{C}^n=\spn\left\{\left.\frac{\partial}{\partial \overline{z^{\eta_{1}}}}\cdots \frac{\partial}{\partial \overline{z^{\eta_{j}}}}\rho_z(0)\,\right|\, j\leq k,\,1\leq \eta_\ell\leq n\right\}.
\end{equation}
\end{proposition}
In fact, the proposition does not require \emph{uniform} $k$-nondegeneracy, but states that \eqref{general knond} holds if and only if the CR structure on $\{(w,z)\,|\,\R(w)=\rho\}$ satisfies a weaker pointwise definition of $k$-nondegeneracy at $0$ given in \cite[Chapter 11.1]{baouendi1999cr}. The stronger \emph{uniform} $k$-nondegeneracy explored in the present article has strong additional implications for the $(k+1)$-jet of $\rho$, which we will derive now.

Fix an arbitrary point $p\in M$ and assume $(M,\omega, J)$ is a uniformly $k$-nondegenerate $2n$-dimensional pre-K\"{a}hler structure with Freeman filtration 
\[
\mathcal{K}_{-1}=T^{1,0}M\supset\mathcal{K}=\mathcal{K}_0\supset\cdots\subset \mathcal{K}_{k-2}\supset \mathcal{K}_{k-1}=0.
\] 
Let $\rho:U\subset M\to \mathbb{R}$ be a local potential of $\omega$ in a neighborhood of $p\in U\subset M$ with $\rho(p)=0$, and let $(S,D,J,X)$ denote the pre-Sasakian structure associated $p\in M$ having Freeman filtration  
\[
\tilde{\mathcal{K}}_{-1}=\mathbb{C}D\cap T^{1,0}S\supset\tilde{\mathcal{K}}_0\supset\cdots\subset\tilde{\mathcal{K}}_{k-2}\supset \tilde{\mathcal{K}}_{k-1}=0,
\]
as described in Remark \ref{CR FF VS PK FF}. Label filtration level dimensions by
\begin{equation}\label{FF levels dimensions}
r_{j}:=\mathrm{rank}(\tilde{\mathcal{K}}_{j})-\mathrm{rank}(\tilde{\mathcal{K}}_{j+1})
\quad\mbox{ and }\quad
R_j:=\sum_{i=-1}^jr_i
\quad\quad\forall\, j=-1,\ldots,k-2.
\end{equation}
The correspondence of Theorem \ref{thm: correspondence theorem} identifies $(S,D,J,X)$ with the real hypersurface
\[
(S,D,J)=\{(w,z)\in\mathbb{C}\oplus\mathbb{C}^n\,|\, \Re(w)=\rho(z),\, z\in U\}
\]
equipped with the distinguished symmetry $X=\frac{\partial}{\partial v}$, where $w=u+\operatorname{i}v$.

If $\left.\exd \rho\right|_0\neq 0$, then a change of coordinates of the form $(w,z)\mapsto (w+L(z),z)$ with $L:\mathbb{C}^n\to \mathbb{C}^n$ linear will  preserve $\frac{\partial}{\partial v}$ and can force $\left.\exd \rho\right|_0= 0$. By changing coordinates further via a linear transformation applied to the $z$ coordinates if necessary, we can assume
\[
\left.\exd \rho\right|_0= 0
\quad\mbox{ and }\quad
\left.\tilde{\mathcal{K}}_j\right|_0=\spn\left\{\left.\left.\frac{\partial}{\partial z^{\ell}}\right|_{0}\right|\,\ell>R_j-r_{j}\right\}
\quad\quad\forall\, j.
\]
As a remark for later, since Freeman's filtration on $(S,X,J)$ is a CR invariant, it is invariant under flows of $X$, and therefore
\begin{equation}\label{coordinates adapted to FF}
\left.{\mathcal{K}}_j\right|_0=\spn\left\{\left.\left.\frac{\partial}{\partial z^{\ell}}\right|_{0}\right|\,\ell>R_j-r_{j}\right\}
\quad\quad\forall\, j
\end{equation}
as well.

For each $-1\leq j<k-1$ let $(X^j_1,\ldots, X^{j}_{r_j})$ be a set of vector fields in $T^{1,0}S$ spanning $\tilde{\mathcal{K}}_j/\tilde{\mathcal{K}}_{j+1}$ in a neighborhood of $0$ satisfying
\begin{equation}\label{adapted FF basis}
\left.X_\ell^j\right|_0 = \left.\frac{\partial}{\partial z^{R_j-r_j+\ell}}\right|_{0}.
\end{equation}
To describe such $X_\ell^j$ in greater detail, note that the vector fields
\[
V_{j}:=\frac{\partial}{\partial z^j}+\frac{\partial \rho}{\partial z^j} \frac{\partial}{\partial w}
\quad\quad\forall j=1,\ldots, n
\]
span $\tilde{\mathcal{K}}_{-1}=\mathbb{C}D\cap T^{1,0}S$ at every point on $S$ (see, e.g. \cite[Chapter 1.6]{baouendi1999cr}). Therefore, one has
\begin{equation}\label{adapted FF basis in detail}
X^j_\ell=V_{R_j-r_j+\ell}+\sum_{i=R_j-r_j+1}^n\alpha_i^j(z) V_i
\end{equation}
for some $\alpha_i^j\in C^\infty(\mathbb{C}^{n})$ with $\alpha_i^j(0)=0$. The reason the sums in \eqref{adapted FF basis in detail} start from $i=R_j-r_j+1$ is that $\left.V_i\right|_0\not\in \tilde{\mathcal{K}}_j$ for small $i$. And the reason $\alpha_i^j$ need not depend on $w$ is that $\frac{\partial}{\partial v}$ is a symmetry, so we might as well work with a basis invariant under translations by $v$.  For simplicity, let us moreover assume 
\begin{equation}\label{adapted FF basis in detail alt}
X^{-1}_j=V_j \quad\quad\forall\, j=1,\ldots, r_{-1}.
\end{equation}

\begin{proposition}\label{prop: leading terms of k nond def eqn lemma}
For any $X_\ell^j$ as constructed above, there exist $j+2$ possibly repeating indices $\eta_1,\ldots, \eta_{j+2} \in \{1,\ldots, r_{-1}\}$ such that 
\begin{equation}\label{leading terms for k-nond def eqn}
\begin{aligned}
\rho_{\overline{\eta_1},\ldots,\overline{\eta_{j+2}},R_j-r_j+\ell}(0):=&\frac{\partial}{\partial \overline{z^{\eta_{1}}}}\cdots \frac{\partial}{\partial \overline{z^{\eta_{j+2}}}}\frac{\partial}{\partial z^{R_j-r_j+\ell}}\rho(0)\\
=&\left.\left[\cdots\left[X_\ell^j,\overline{X^{-1}_{\eta_1}}\right],\cdots,\overline{X^{-1}_{\eta_{j+2}}}\right](w)\right|_{0}\neq0.
\end{aligned}
\end{equation}
For any other multi-index $\nu=(\nu_1,\ldots, \nu_s)$ with $\nu_1,\ldots, \nu_s \in \{1,\ldots, r_{-1}\}$ and length $|\nu|=s<j+2$
\begin{equation}\label{leading terms for k-nond def eqn alt}
\rho_{\overline{\nu_1},\ldots,\overline{\nu_s},R_j-r_j+\ell}(0)=\left.\left[\cdots\left[X_\ell^j,\overline{X^{-1}_{\nu_1}}\right],\cdots,\overline{X^{-1}_{\nu_s}}\right](w)\right|_{0}=0.
\end{equation}
\end{proposition}
\begin{proof}
Since $X_\ell^j\in \tilde{\mathcal{K}}_j$, by the definition of $\tilde{\mathcal{K}}_j$ there exist $j+1$ possibly repeating vector fields $Y_1,\ldots, Y_{j+2}$ from the $\tilde{\mathcal{K}}_{-1}$ basis \eqref{adapted FF basis} such that $[\cdots[X_\ell,\overline{Y_1}],\ldots,\overline{Y_{j+2}}]$ is not contained in $\mathbb{C}D$. The mapping $(X_1,\ldots, X_{j+2})\mapsto [\cdots[X_\ell,\overline{X_1}],\ldots,\overline{X_{j+2}}]_0\pmod{\mathbb{C}D_0}$ from $\bigotimes^{j+2} \Gamma(\mathbb{C}D\cap T^{1,0}S) \to \mathbb{C}T_0S/\mathbb{C}D_0$ is symmetric, so by possibly reordering $Y_1,\ldots, Y_{j+2}$ we can assume without loss of generality that all of the $Y_s$ belonging to $\Gamma(\tilde{\mathcal{K}}_0)$ are enumerated last. 

Again by the definition of $\tilde{\mathcal{K}}_j$, one has $[\cdots[X_\ell,\overline{Y_1}],\ldots,\overline{Y_{s}}]\in \Gamma(\mathbb{C}D)$ for any $s<j+2$. This implies that the $\frac{\partial}{\partial w}$ part of Lie bracket  in \eqref{leading terms for k-nond def eqn alt} is $0$ at $0$ since $\mathbb{C}D$ is spanned by $\{\frac{\partial}{\partial z^i},\frac{\partial}{\partial \overline{z^i}}\}$ there. Computing this $\frac{\partial}{\partial w}$ part directly using  \eqref{adapted FF basis in detail}  and \eqref{adapted FF basis in detail alt},  we find that its coefficient is 
\[
0=\left.\rho_{\overline{\nu_1},\ldots,\overline{\nu_s},R_j-r_j+\ell}+\sum_{i=R_j-r_j+1}^n\frac{\partial}{\partial \overline{z^{\eta_{1}}}}\cdots \frac{\partial}{\partial \overline{z^{\eta_{j+2}}}}\left(\alpha_i^j\frac{\partial \rho}{\partial z^{i}}\right)\right|_0.
\]
Accordingly, either $\rho_{\overline{\nu_1},\ldots,\overline{\nu_s},R_j-r_j+\ell}(0)=0$, or there exists another index $i^\prime>R_j-r_j$ with a shorter subset of multi-indices $\nu^\prime_1,\ldots, \nu^\prime_{s^\prime}\in\{\nu_1,\ldots, \nu_s\}$ with $s^\prime<0$ such that $\rho_{\overline{\nu^\prime_1},\ldots,\overline{\nu^\prime_{s^\prime}},i^\prime}(0)=0$. The latter eventually leads to a contradiction however, because if we repeat this argument recursively, proceeding next with the field $X_{\ell^\prime}^{j^\prime}$ having $i^\prime=R_{j^\prime}-r_{j^\prime}+\ell^\prime$, then the implication on each recursion is that there exists yet another index $i^{\prime\prime}$ with corresponding shorter list of muli-indices $\nu^{\prime\prime}_1,\ldots, \nu^{\prime\prime}_{s^{\prime\prime}}\in\{\nu_1,\ldots, \nu_s\}$ where $s^{\prime\prime}<s^\prime$ and eventually one exhausts all indices so that the implication of \emph{there existing another index} cannot hold. Therefore \eqref{adapted FF basis in detail alt} holds for all $s<j+2$.

Regarding  \eqref{leading terms for k-nond def eqn}, since $[\tilde{\mathcal{K}}_0,\mathbb{C}D]\subset \mathbb{C}D$ and $[\cdots[X_\ell,\overline{Y_1}],\ldots,\overline{Y_{j+1}}]\in \Gamma(\mathbb{C}D)$, if $Y_{j+2}\in\Gamma(\tilde{\mathcal{K}}_0)$ then $[\cdots[X_\ell,\overline{Y_1}],\ldots,\overline{Y_{j+2}}]$ is contained in $\mathbb{C}D$, a contradiction. Therefore, $Y_j\in \{X^{-1}_i\,|\, i=1,\ldots, r_{-1}\}$ must hold. By direct computation, noting all $\alpha^j_i$ vanish at $0$ and using \eqref{adapted FF basis in detail}, \eqref{adapted FF basis in detail alt}, and now \eqref{leading terms for k-nond def eqn alt},  we find that the $\frac{\partial}{\partial w}$ part of the Lie bracket  in \eqref{leading terms for k-nond def eqn} has exactly the same coefficient on the left side of \eqref{leading terms for k-nond def eqn}. The $\frac{\partial}{\partial \overline{w}}$ part also has the same coefficient, and it must be nonzero because the bracket is outside of $\mathbb{C}D$, whilst $\mathbb{C}D$ is spanned by $\langle\frac{\partial}{\partial z^i},\frac{\partial}{\partial \overline{z^i}}\rangle$ at $0$.
\end{proof}

\begin{corollary}\label{corol: leading terms of k nond def eqn}
Let $\rho:U\subset\mathbb{C}^n\to \mathbb{R}$ be a local pre-K\"{a}hler potential of a uniformly $k$-nondegenerate pre-K\"{a}hler structure expressed in coordinates adapted to its Freeman filtration at $0$ as in \eqref{coordinates adapted to FF}, with $R_j$ and $r_j$ as in \eqref{FF levels dimensions}. For $-1\leq j<k-1$, consider the $(j+2)$-dimensional arrays $A^j_1,\ldots,A^j_{r_j}$ of size $(r_{-1})^{j+2}$ whose $(\eta_1,\ldots,\eta_{j+2})$ entries are
\[
\left(A^j_\ell\right)_{\eta_1,\ldots,\eta_{j+2}}:=\rho_{\overline{\eta_1},\ldots,\overline{\eta_{j+2}},R_j-r_j+\ell}(0)
\quad\quad\forall\,1\leq \ell\leq r_j.
\]
For all  $-1\leq j<k-1$, $A^j_1,\ldots,A^j_{r_j}$ are linearly independent.
\end{corollary}
\begin{proof}
Suppose $A^j_1,\ldots,A^j_{r_j}$ are not linearly independent and let $0\neq (c_1,\ldots,c_{r_j})\in\mathbb{R}^{r_j}$ satisfy $\sum_\ell c_\ell A^j_\ell=0$. Apply an invertible linear map $T:\mathbb{C}^{n}\to \mathbb{C}^{n}$ satisfying
\[
T(z^{R_j-r_j+1})=\sum_{\ell=1}^{r_j}c_\ell z^{R_j-r_j+\ell}
\quad\mbox{ and }\quad
T(z^\ell)\in \spn\{z^{R_j-r_j+1},\ldots,z^{R_j}\}
\]
and $T(z^i)= z^i$ for all $i\not\in\{R_j-r_j,\ldots,R_j\}$ to transform the  $z$ coordinates. In these new coordinates, we have
\[
\rho_{\overline{\eta_1},\ldots,\overline{\eta_{j+2}},R_j-r_j+1}(0)=\sum_{\ell=1}^{r_j}c_\ell\left(A^j_\ell\right)_{\eta_1,\ldots,\eta_{j+2}}=0
\]
for all $\eta_1,\ldots, \eta_{j+2} \in \{1,\ldots, r_{-1}\}$, which contradicts Lemma \ref{prop: leading terms of k nond def eqn lemma}.
\end{proof}

\section{2-nondegenerate pre-K\"ahler structures on complex surfaces}  
\label{sec:4D-prekahler}
In this section we  give a solution of the equivalence problem for 2-nondegenerate pre-K\"ahler structures on complex surfaces in the form of a Cartan geometry, express their basic invariants in terms of a potential, and give an interpretation of the vanishing of each of the basic invariants. We refer the reader to the accompanying file \cite{PythonCode} for a notebook where the main part of the computations are carried out using  \cite{dgcv}.     
\subsection{The structure bundle}\label{sec:structure-bundle}
In order to solve the equivalence problem and find the structure equations for  2-nondegenerate pre-K\"ahler complex surfaces, we follow Cartan's method of equivalence by sequentially finding coframings adapted to increasingly restrictive but naturally defined conditions for any pre-K\"ahler complex surface. Our treatment here is self-contained, however familiarity with Cartan's method of equivalence would be helpful, for which we refer to \cite{Gardner}. 

\subsubsection{Pre-K\"ahler structure and 1-adaptation}
 Let $M$ be a 4-dimensional real manifold with a coframe $(\alpha^1,\alpha^2,\beta^1,\beta^2)$ on $T^*M$ with respect to which the pre-symplectic form $\omega$ is expressed  as
 \begin{equation}\label{eq:omega-4D-prekahler}
   \omega=\alpha^1\w\alpha^2,
    \end{equation}
    where the planes that annihilate $\omega$ are tangent planes to the 2-dimensional leaves of the integrable Pfaffian system  $\cI_{\alpha}=\{\alpha^1,\alpha^2\}.$ The 1-forms $(\alpha^1,\alpha^2)$ satisfying \eqref{eq:omega-4D-prekahler} are defined up to the action of $\mathrm{SL}(2,\RR),$ i.e. they are defined up to transformations $\alpha^i\to A^i_j\alpha^j$ where $[A^i_j]\in\mathrm{SL}(2,\RR)$.

    For a pre-K\"ahler structure  the leaf space of $\cI_\alpha,$ denoted as $N,$ is equipped with a symmetric bilinear form $g=g_{ij}\alpha^i\alpha^j.$ Since $\alpha^i$'s, satisfying \eqref{eq:omega-4D-prekahler}, are defined up to an action of $\mathrm{SL}(2,\RR),$ one can find a coframe $(\alpha^1,\alpha^2)$  with respect to which the following holds
\begin{equation}\label{eq:omega-g-4D-prekahler}
  g=\lambda\left((\alpha^1)^2+(\alpha^2)^2\right),\quad \omega=\alpha^1\w\alpha^2,
  \end{equation}
  for some nowhere vanishing function $\lambda.$ Coframes $(\alpha^1,\alpha^2)$ with respect to which $g$ and $\omega$ are written as \eqref{eq:omega-g-4D-prekahler} are defined up to an action of $\mathrm{SO}(2,\RR).$

Using  the compatibility of  $g$ and $\omega,$ expressed as $g(\cdot ,\cdot )=\omega(\cdot ,J \cdot),$ where $J\colon T M\to TM$ is the  integrable almost complex structure on $M,$ it follows that  $\lambda=1$ and that  $\alpha^1+\ri\alpha^2$ is an $\ri$-eigenvector for the dual action of $J$  on $T^*M$.  Now one can choose $\beta^1$ and $\beta^2$ so that the coframe $(\alpha^1,\alpha^2,\beta^1,\beta^2)$  satisfies \eqref{eq:omega-g-4D-prekahler} and the 1-forms  $(\alpha^1+\ri\alpha^2,\beta^1+\ri\beta^2)$ are  $\ri$-eigenspace for the almost complex structure on $T^*M.$ The integrability of the almost complex structure implies that the complex Pfaffian system $\{\alpha^1+\ri\alpha^2,\beta^1+\ri\beta^2\}$  is  integrable and its annihilator defines the holomorphic distribution  $\scH\subset \CC\otimes TM.$

By the discussion above, one obtains
\begin{equation}\label{eq:omega-g-1-adapted-4D}
  \omega=\tfrac{\ri}{2}\theta^1\w\cthetao,\quad g=\theta^1\cthetao,
  \end{equation}
where
\[\theta^1:=\alpha^1+\ri\alpha^2,\quad \theta^2:=\beta^1+\ri\beta^2\]
are the holomorphic 1-forms on $M.$

The \emph{1-adapted} coframe $(\theta^1,\theta^2,\cthetao,\cthetat)$ with respect to which  $\omega$ and $g$ can be expressed as \eqref{eq:omega-g-1-adapted-4D}  are defined up to transformations $\theta^i\to A^i_j\theta^j$, where $A\in G_1$ and
\begin{equation}\label{eq:G1-matrix}
  G_1:=\left\{ A\in\mathrm{GL}(2,\CC)\ \vline \ A=
    \begin{pmatrix}
      e^{a\ri} & 0\\
      b_1 + \ri b_2 & c_1 +\ri c_2
    \end{pmatrix}
  \right\}.
  \end{equation}

\subsubsection{2-nondegeneracy and 2-adaptation}
The degenerate directions of $\omega$ are equipped with a splitting  $K=\scK\oplus\cscK$  where $\scK=\langle\tfrac{\partial}{\partial\theta^2}\rangle.$ As a result of  Proposition \ref{prop: low-order finite nondeg consequenes}, 2-nondegeneracy can be expressed as
\begin{equation}\label{eq:2nondeg}
  \begin{aligned}
    \exd\theta^1\equiv \widetilde\lambda\ \cthetao\w\theta^2\mod\{\theta^1\},\\
      \end{aligned}
  \end{equation}
  for a nowhere vanishing complex-valued function $\widetilde\lambda.$

  Via coframe transformations $\theta^i\to A^i_j\theta^j$ where $[A^i_j]\in G_1,$ one arrives at the transformation law for $\widetilde\lambda,$ given by   \[\widetilde\lambda\to e^{-2a\ri}(c_1+\ri c_2)\widetilde\lambda.\]
  As a result, one can define the \emph{2-adapted} coframes as those with respect to which
  \begin{equation}\label{eq:lambda-normaized-1}
    \widetilde\lambda=1.
  \end{equation}
  Consequently, such coframes are defined up to the action of $G_2\subset G_1$ for which $c_1+\ri c_2=e^{2a\ri}$ in  \eqref{eq:G1-matrix}, i.e.   
  \begin{equation}\label{eq:G2-2-adapted}
    G_2=\left\{ A\in\mathrm{GL}(2,\CC)\ \vline \ A=
    \begin{pmatrix}
      e^{a\ri} & 0\\
      b_1+\ri b_2 & e^{2a\ri}
    \end{pmatrix}
  \right\}.
\end{equation}

\subsubsection{Some differential relations}

Using the integrability of $\{\theta^1,\theta^2\}$ together with  2-nondegeneracy \eqref{eq:2nondeg} conditioned to the normalization \eqref{eq:lambda-normaized-1}, it follows that 
\begin{equation}\label{eq:dtheta12-G2}
  \begin{aligned}  \exd\theta^1=&\cthetao\w\theta^2+A^1_{12}\theta^1\w\theta^2+A^1_{1\co}\theta^1\w\cthetao+A^1_{2\co}\theta^2\w\cthetao+A^1_{2\ct}\theta^2\w\cthetat,\\
    \exd\theta^2=&A^2_{12}\theta^1\w\theta^2+A^2_{i\overline j}\theta^i\w\cthetaj,
      \end{aligned}
\end{equation}
for some complex-valued functions $A^i_{12}$ and $A^i_{j\ov k}$ on $M.$ 

Inspecting   $\exd\omega=0,$ it is straightforward to obtain $A^1_{2\co}=-\overline{A^1_{12}}$ and $A^1_{2\ct}=0.$ Thus, one arrives at
\begin{equation}\label{eq:dtheta1-G2}
  \exd\theta^1= \cthetao\w\theta^2 +A^1_{12}\theta^1\w\theta^2+  A^1_{1\co}\theta^1\w\cthetao-\ov{A^1_{12}}\theta^1\w\cthetat.
\end{equation}
Furthermore, inspecting $\exd^2\theta^1=0$ mod $\{\theta^1\},$ it is straightforward to obtain $A^2_{2\ct}=-2\ov{A^1_{12}}.$ As a result, one has
\begin{equation}\label{eq:dtheta12-2-G2}
  \begin{aligned}
    \exd\theta^1=&\cthetao\w\theta^2 +A^1_{12}\theta^1\w\theta^2+  A^1_{1\co}\theta^1\w\cthetao-\ov{A^1_{12}}\theta^1\w\cthetat,\\    \exd\theta^2=&A^2_{12}\theta^1\w\theta^2+A^2_{1\co}\theta^1\w\cthetao+A^2_{1\ct}\theta^1\w\cthetat+A^2_{2\co}\theta^2\w\cthetao-2\ov{A^1_{12}}\theta^2\w\cthetat.
      \end{aligned}
    \end{equation}
    The differential relation we are seeking is between $A^2_{1\ct}$ and derivatives of $A^1_{1\co}$ and $A^2_{1\ct},$ which will be used subsequently. To relate these quantities, it is a matter of  computation to show that in the expansion of $\exd^2\theta^1=0$ the following has to hold
    \[(-A^1_{1\co} \ov{A^1_{12}}+\ov{A^1_{12}} \ov{A^2_{12}}+A^2_{1\ct}+A^1_{1\co;\ct}+\ov{A^1_{12;\co}})\theta^1\w\cthetao\w\cthetat=0,\]
 wherein we are using the notation for coframe differentiation introduced in \ref{sec:conventions}.    Similarly, from the expansion of $\exd^2\theta^2=0$ one obtains
    \[(A^2_{2\co} \ov{A^1_{12}}-2 \ov{A^1_{12}} \ov{A^2_{12}}+A^2_{1\ct}-A^2_{2\co;\ct}-2 \ov{A^1_{12;\co}})\cthetao\w\theta^2\w\cthetat=0.\]
    
    As a result of the two relations above one arrives at
    \begin{equation}\label{eq:B111-B221}
      A^2_{2\co ;\ct}-2A^1_{1\co;\ct}= 3A^2_{1\ct}+\ov{A^1_{12}}(A^2_{2\co} - 2A^1_{1\co}).
    \end{equation}
\subsubsection{A distinguished holomorphic splitting and 3-adaptation}
Starting with a 2-adapted coframing satisfying \eqref{eq:dtheta12-2-G2}, as a result of the differential relation \eqref{eq:B111-B221}, it is an elementary computation to show that via the admissible $G_2$-transformation
\begin{equation}\label{eq:distinguished-splitting-G2-trans}
  \theta^2\to \theta^2+\tfrac 13(A^2_{2\ov 1}-2A^1_{1\ov 1})\theta^1,
  \end{equation}
one has
\begin{equation}\label{eq:dtheta12-3-G3}
  \begin{aligned}
    \exd\theta^1=&\cthetao\w\theta^2 +B^1_{12}\theta^1\w\theta^2+  B^1_{1\ov 1}\theta^1\w\cthetao -\ov {B^1_{12}}\theta^1\w\cthetat,\\    \exd\theta^2=&B^2_{12}\theta^1\w\theta^2+B^2_{1\ov 1}\theta^1\w\cthetao+2B^1_{1\ov 1}\theta^2\w\cthetao-2\ov{B^1_{12}}\theta^2\w\cthetat,
      \end{aligned}
    \end{equation}
    where $B^1_{12}=A^1_{12},$ $B^1_{1\ov 1}=\tfrac 13(-2 \ov {A^1_{12}} \ov{A^1_{1\ov 1}}+\ov{A^1_{12}} \ov{A^1_{21}}+A^1_{1\ov 1}+ A^1_{2\ov 1})$ while $B^2_{12}$ and  $B^2_{1\ov 1}$ involve the first jet of $A^1_{12}$ and $A^2_{2\ov 1}$ and $A^1_{1\ov 1}.$ Obtaining the exact expressions of these quantities is straightforward, although tedious, and will not be important for us.  

The main point of the preceeding computation is that  $\exd\theta^2$ in \eqref{eq:dtheta12-3-G3} does not have a term involving $\theta^1\w\cthetat.$ This will be the defining property of \emph{3-adapted} coframes, i.e.    among 2-adapted coframing satisfying \eqref{eq:dtheta12-2-G2}, the 3-adapted coframing are defined by the property that with respect to them one has $A^2_{1\ov 2}=0.$

Our computations above shows, given any 2-adapted coframe, how one can make it 3-adapted, i.e. by an action of $G_2$ as in \eqref{eq:G2-2-adapted} in which $b_1+\ri b_2=\tfrac 13(A^2_{2\ov 1}-2A^1_{1\ov 1}).$ As a result, 3-adaptation of a coframe reduces the group of admissible transformation to the subgroup $G_3\subset G_2$ defined as 
  \begin{equation}\label{eq:G3-3-adapted}
    G_3=\left\{ A\in\mathrm{GL}(2,\CC)\ \vline \ A=
    \begin{pmatrix}
      e^{a\ri} & 0\\
      0 & e^{2a\ri}
    \end{pmatrix}
  \right\}\cong \mathrm{U}(1).
\end{equation}
Hence, 3-adapted coframes $(\theta^1,\theta^2)$ are defined up to a $\mathrm{U}(1)$ action and, thus, define a  splitting
\[\scH=\langle\tfrac{\partial}{\partial\theta^1}\rangle\oplus \langle\tfrac{\partial}{\partial\theta^2}\rangle\]
for the holomorphic distribution $\scH\subset \CC TM.$
 
As was mentioned in \ref{sec:outline-main-results}, the resulting holomorphic splitting is \emph{distinguished} due to the fact that $[\partial_{\theta^{2}},\partial_{\ov{\theta^1}}]\equiv  \partial_{\theta^1} $ modulo $\langle\partial_{\ov{\theta^1 }},\partial_{\theta^{2}}\rangle.$
\begin{remark}
  Alternatively, one could find the $G_2$-action on the torsion entry $B^2_{12}$ and translate it to zero. Doing so  requires some familiarity with Cartan's method of equivalence, which we will not elaborate  on here.  
\end{remark}
\subsubsection{Lifted coframe on a $\mathrm{U}(1)$-bundle}

The set of 3-adapted coframes defines  a principal $\mathrm{U}(1)$-bundle $\varsigma:\cG\rightarrow M$ with the right action  given by
\[R_{g}( \vartheta_p)=g^{-1} \vartheta_p,\]
where $\vartheta_p\in\varsigma^{-1}(p)$ is a coframe at $p\in M,$    $g\in  G_{3},$  and the action on the right hand side is the ordinary matrix multiplication on  $(\theta^1,\theta^2)^\top$ and its conjugate on $(\cthetao,\cthetat)^\top$

Restricting to an open set $U\subset M$, on $\cG|_U\cong  U \times  G_3,$ one can define a canonical  set of semi-basic  1-forms  by lifting $\theta^i$'s
 and  $\cthetai$'s. In the case of $\theta^i$'s, their lift, denoted by $(\what{\theta}{}^1,\what{\theta}{}^2)^\top,$ is  \[\what{\vartheta}(p,g):=(\what{\theta}{}^1,\what{\theta}{}^2)^\top= g^{-1} \what{\vartheta}{}_p,\]
where $\what{\vartheta}{}_p$ is a choice of 3-adapted holomorphic coframe  $(\theta^1,\theta^2)^\top$  at $p\in M.$ 
 Their exterior derivatives are given by
\begin{equation}
  \label{eq:str-eqns-strart}
  \exd\what\vartheta=\exd g^{-1} \w \vartheta+g^{-1} \exd \vartheta=-g^{-1}\exd g \w g^{-1} \vartheta+g^{-1}\exd \vartheta=-\Omega\w \what\vartheta+ T,
\end{equation}
where $\Omega(p,g)= g^{-1}\exd g$ is the $\mathfrak g_3$-valued 1-form of the Lie group \eqref{eq:G3-3-adapted}. The 2-forms  $T(p,g)$ is the lift of the torsion to $\cG.$

Because from now on we will be mostly working with  lifted coframe on $\cG$, we will drop  $\ \what{}\ $ that was used to distinguish between lifted 1-forms on $\cG$ and those on $M.$  The distinction, if not clear from the context, will be mentioned explicitly. 

The 1-forms in $\Omega$ along each fiber of $\cG\to M$ can be interpreted as the Maurer--Cartan forms of $G_3$ and, therefore, one has  
\begin{equation}\label{eq:Omega-MC-forms-G2}
  \Omega=  \begin{pmatrix}
    \ri \psi & 0\\
    0 & 2\ri\psi
  \end{pmatrix},
\end{equation}
for some real-valued 1-form $\psi\in\Omega^1(\cG).$
Expressing   $\Omega$  as \eqref{eq:Omega-MC-forms-G2} and using \eqref{eq:dtheta12-3-G3}, it follows that 
\begin{equation}\label{eq:streqns-G3-lifted}
  \begin{aligned}
    \exd\theta^1=&-\ri\psi\w\theta^1+\cthetao\w\theta^2 - (C^1_{12}\theta^2+  C^1_{1\ov 1}\cthetao -\ov {C^1_{12}}\cthetat)\w\theta^1,\\    \exd\theta^2=&-2\ri\psi\w\theta^2+C^2_{12}\theta^1\w\theta^2+C^2_{1\ov 1}\theta^1\w\cthetao-2(C^1_{1\ov 1}\cthetao-\ov{C^1_{12}}\cthetat)\w\theta^2,
      \end{aligned}
    \end{equation}
    where the complex-valued functions $C^i_{ij}$ and $C^i_{i\ov j}$ are defined on $\cG$ and are the lift of the functions  $B^i_{ij}$ and $B^i_{i\ov j}$ in \eqref{eq:dtheta12-3-G3} via the induced $\mathrm{U}(1)$ action. More explicitly, one can  easily compute 
    \[C^1_{12}=B^1_{12}e^{2a\ri },\quad C^1_{1\ov 1}= B^1_{1\ov 1}e^{-a\ri},\quad C^2_{12}= B^2_{12}e^{a\ri},\quad C^2_{1\ov 1}=B^2_{1\ov 1}e^{-2a\ri}.\]

 \subsubsection{Absorption of torsion and absolute parallelism}

  The real 1-form $\psi$ on $\cG$ is ambiguous up to a linear combination of the semi-basic 1-forms, i.e. it can be replaced by 
  \begin{equation} \label{eq:xis-for-psi}
      \begin{aligned}
    \psi&\to \psi+x_1\theta^1+\ov{x_1}\cthetao+x_2\theta^2+\ov{x_2}\cthetat,\\
    \end{aligned}
   \end{equation}
   for some complex-valued functions $x_i$. A canonical coframe on $\cG$ can be determined by prescribing an absorption of torsion. A natural choice of torsion absorption would be to absorb the non-constant torsion coefficient of $\exd\theta^1$ in \eqref{eq:streqns-G3-lifted} i.e.
\[x_1=\ri \ov{C^1_{1\ov 1}},\quad x_2=-\ri C^1_{12}.\]
Using the above choice of $x_1$ and $x_2$ in the replacement \eqref{eq:xis-for-psi}, equations \eqref{eq:streqns-G3-lifted} take the form
\begin{equation}\label{eq:dtheta12-on-cG}
  \begin{aligned}
    \exd\theta^1&=-\ri\psi\w\theta^1+\cthetao\w\theta^2,\\
      \exd\theta^2&=-2\ri\psi\w\theta^2 +T_1\theta^1\w\ov{\theta^1}+T_2\theta^1\w\theta^2,\\    
      \end{aligned}
  \end{equation}
  where
  \begin{equation}\label{eq:C-T}
  T_1=C^2_{1\co},\quad T_2=C^2_{12}+2\ov{C^1_{1\co}}.
  \end{equation}

 As a result, the 1-forms $(\theta^1,\theta^2,\cthetao,\cthetat,\psi)$ define an $\{e\}$-structure, also known as an absolute parallelism, on $\cG.$

 \subsubsection{The structure equations and local generality}
Using equations \eqref{eq:dtheta12-on-cG}, it is elementary to use $\exd^2\theta^i=0$ to find
  \begin{equation}\label{eq:dphi-streqs}
        \begin{aligned}
      \exd\psi &=\ri \theta^2\w\ov{\theta^2}+\ri \ov{T_2}\theta^1\w\ov{\theta^2}-\ri T_2\ov{\theta^1}\w\theta^2+\tfrac{\ri}{2}(T_{2;\bar 1}-T_{1;2})\theta^1\w\ov{\theta^1},
    \end{aligned}
  \end{equation}
  together with Bianchi identities
  \begin{subequations}\label{eq:T1-2-translaw}
      \begin{gather}
      \exd T_1=T_{1;1}\theta^1+T_{1;\bar 1}\ov{\theta^1}+T_{1;2}\theta^2-2\ri T_1\psi,\label{eq:T1-trans}\\
      \exd T_2=T_{2;1}\theta^1+T_{2;\bar 1}\ov{\theta^1}+T_{2;2}\theta^2+2\ov{T_2}\ov{\theta^2}+\ri T_2\psi,\label{eq:T2-trans}
    \end{gather}
      \end{subequations}
  satisfying the relation
  \begin{equation}\label{eq:one-Bianchi-T2-T1}
    T_{2;\bar 1}-T_{1;2}=\ov{T_{2}}{}_{; 1}-\ov{T_{1}}{}_{;\bar 2},
    \end{equation}
which guarantees that \eqref{eq:dphi-streqs} is real-valued.

Applying Cartan--K\"ahler theory to the structure equations above, it follows that real analytic 2-nondegenerate pre-K\"ahler 4-manifolds locally depend on 2 functions of 3 variables.
\begin{remark}
We point out that in all known examples, e.g. see \cite[Sections 5.1.5, 5.2.5, 5.3.6]{MS-cone}, if the local generality of a geometric structure is given by $s$ functions if $k$ variables, then the local generality of those structures equipped with a distinguished choice of infinitesimal symmetry would be $s$ functions of $k-1$ variables. In other words, the presence of an infinitesimal symmetry can be interpreted as a local \emph{redundancy} in the local description of the structure.   As a result, our count above suggests that the local generality of real analytic 2-nondegenerate CR structures in dimension five is given by 2 functions of 4 variables.
\end{remark}

  \subsection{Cartan geometric description}\label{sec:cart-geom-description}
In order to give a Cartan geometric solution of the equivalence problem for 2-nondegenerate pre-K\"ahler structure on complex surfaces, recall the definition of a Cartan geometry.
  \begin{definition}\label{def:cart-geom-definitio}
  Let $G$ be a Lie group and $P\subset G$ a Lie subgroup  with Lie algebras $\fg$ and $\fp\subset\fg,$ respectively.  A Cartan geometry of type $(G,P)$ on a manifold $M,$ denoted as $(\cG\to M,\varphi),$ is a right principal $P$-bundle $\cG\to M$ equipped with a Cartan connection $\varphi\in\Omega^1(\cG,\fg),$ i.e. a $\fg$-valued 1-form on $\cG$ satisfying
  \begin{enumerate}
  \item $\varphi$ is $P$-equivariant, i.e. $R_g^*\varphi=\mathrm{Ad}_{g^{-1}}\varphi$ for all $g\in P,$ where $R_g$ denotes the right action by $g$.
  \item  $\varphi_z\colon T_z\cG\to \fg$ is  a linear isomorphism for all $z\in \cG.$
  \item $\varphi$ maps fundamental vector fields to their generators, i.e. $\varphi(\zeta_X)=X$ for any $X\in\fp$ where $\zeta_X(z):=\frac{\exd}{\exd t}\,\vline_{\,t=0}R_{exp(tX)}(z).$
  \end{enumerate} 
 The 2-form $\Phi\in\Omega^2(\cG,\fg)$ defined as
    \[\Phi(u,v)=\exd\varphi(u,v)+[\varphi(X),\varphi(Y)]\quad \text{for\ \ }X,Y\in \Gamma(T\cG),\]
is called the Cartan curvature and is $P$-equivariant and semi-basic with respect to the fibration $\cG\to M.$
\end{definition} 
  Using the preceding  discussion,  one has the following solution for the equivalence problem of 2-nondegenerate pre-K\"ahler structures on complex surfaces.
  \begin{theorem}\label{thm:cart-geom-descr}
     Any  2-nondegenerate pre-K\"ahler structure $(g,\omega)$ on a complex surface $M$ canonically defines a Cartan geometry $(\cG\to M,\varphi)$ of type $(\RR^2\rtimes\mathrm{SL}(2,\RR) ,\mathrm{U}(1))$ where the Cartan connection and its curvature are
     \begin{equation}\label{eq:Cartan-conn-curv}
\varphi=
  \begin{pmatrix}
    0 & 0 & 0\\
    \ov{\theta^1} & -\ri\psi & \ov{\theta^2}\\
    {\theta^1}&{\theta^2} & \ri\psi
  \end{pmatrix},\quad 
  \exd\varphi+\varphi\w\varphi=
  \begin{pmatrix}
    0 & 0 & 0\\
    0 & -\ri\Psi & \ov{\Theta^2}\\
    0&{\Theta^2} & \ri\Psi
  \end{pmatrix},
\end{equation}
in which 
\begin{equation}\label{eq:cartan-curv-prekahler}
      \begin{aligned}
        \Theta^2=T_1\theta^1\w\ov{\theta^1}+T_2\theta^1\w\theta^2,\quad
        \Psi =\ri \ov{T_2}\theta^1\w\ov{\theta^2}-\ri T_2\ov{\theta^1}\w\theta^2+T_3\theta^1\w\ov{\theta^1},
    \end{aligned}
\end{equation}
  for some functions $T_1,T_2\in C^\infty(\cG,\CC)$ and $T_3\in C^{\infty}(\cG,\ri \RR)$.  Conversely, any such Cartan geometry defines a unique pre-K\"ahler structure on $M.$ The basic invariants for such Cartan geometries are  $\bT_1,\bT_2\in C^\infty(M,\RR)$  where
  \begin{equation}\label{eq:fund-inv-prekahler4}
\bT_1:=T_1\ov{T_1},\quad   \bT_2:=T_2\ov{T_2},  
  \end{equation}
whose vanishing  characterizes locally flat pre-K\"ahler structures, i.e. $(g,\omega)$ is locally equivalent to the homogeneous space $G/\mathrm{U}(1)$ where $G=\RR^2\rtimes \mathrm{SL}(2,\RR).$
\end{theorem}
\begin{proof} 
As a result of our discussion in  \ref{sec:structure-bundle}, a principal $\mathrm{U}(1)$-bundle $\cG\to M$ is canonically defined for any pre-K\"ahler structure.   By construction, it is straightforward to check that  $\varphi$ satisfies the properties in Definition \ref{def:cart-geom-definitio}.  The structure equations  \eqref{eq:dtheta12-on-cG} and \eqref{eq:dphi-streqs}  on  $\cG$ coincide with the Cartan curvature \eqref{eq:Cartan-conn-curv} and \eqref{eq:cartan-curv-prekahler}. Note that by \eqref{eq:dphi-streqs} one has
  \begin{equation}\label{eq:T3-Bianchi}
    T_3=\tfrac{\ri}{2}(T_{2;\bar 1}-T_{1;2})=-\ov{T_3}
  \end{equation}
  in \eqref{eq:cartan-curv-prekahler}.  To see that the Cartan connection $\varphi$ takes value in $\RR^2\rtimes\mathfrak{sl}(2,\RR),$ we have used the isomorphism $\mathfrak{sl}(2,\RR)\cong\mathfrak{su}(1,1).$

  Conversely, taking any section $s\colon M\to \cG,$ then $(g,\omega),$ where $g=s^*\theta^1 s^*\ov{\theta^1}$ and $\omega=s^*\theta^1\w s^*\ov{\theta^1},$ defines a pre-K\"ahler structure on $M.$

  Lastly, if $\bT_1=\bT_2=0$ holds then, by \eqref{eq:T3-Bianchi}, one has $T_3=0,$ and, consequently, $\exd\varphi+\varphi\w\varphi=0,$ i.e. $\varphi$ is the Maurer--Cartan form for the Lie algebra of $\RR^2\rtimes\mathrm{SL}(2,\RR),$ which by the first part of the theorem is equipped with a canonical pre-K\"ahler structure. 
\end{proof} 

\subsection{Parametric expressions}\label{sec:param-expr}

Staring with a (local) potential function $\rho=\rho(z^1,z^2,\ov{z^1},\ov{z^2}),$ we express the structure functions  $T_1$ and $T_2$ in terms of $\rho.$ The expressions can have singularities in general. However, any such singularities can always be removed by shifting the singular point to the origin and then transforming the $\rho$ representation with a change of coordinates of the form
\begin{equation}\label{eqn: removing coordinate expression singularities}
(z^1,z^2) \mapsto (z^1+\epsilon z^1z^2,z^2)
\end{equation}
for some $\epsilon\in \mathbb{C}$. In this new coordinate system around the origin, the following computation can be carried out  without producing singularities (see Remark \ref{rem: removing singularities}).

  Let $H=[\rho_{i\ov j}],$ where $\rho_{i\ov j}=\rho_{z^i\ov{z^j}}$ and   $\ov{\rho_{i\ov j}}=\rho_{j\ov i}$, denote the $(1,1)$ Hessian matrix of $\rho.$ Being pre-K\"ahler implies that $H$ has rank 1 and, therefore, $\det(H)=0$. Since $H$ is not identically zero, one can assume  $\rho_{1\ov{1}}\neq 0,$ as a result of which one has
\[H=\def\arraystretch{1.5}
    \begin{pmatrix}
    \rho_{1\ov 1} & \rho_{1\ov 2}\\
    \rho_{2\ov 1} &\tfrac{\rho_{1\ov 2}{\rho}_{2\ov 1}}{\rho_{1\ov 1}}
  \end{pmatrix}.
\]
In order to find a choice of adapted 1-form $\theta^1,$ one notes that the symmetric bilinear form $g\in\Gamma(\mathrm{Sym}^2 T^*M)$ in \eqref{eq:omega-g-1-adapted-4D} can be written as
\[g=\rho_{i\ov j}\exd z^i\exd \ov{z^j}=\theta^1\cthetao.\]
Setting
\[R= \def\arraystretch{1.5}
  \begin{pmatrix}
    0 & 0\\
    \ri \sqrt{\rho_{1\ov 1}} & \ri\tfrac{ \rho_{2\ov 1}}{\sqrt{\rho_{1\ov 1}}}
  \end{pmatrix},
\]
 one has $H=R^\top\ov{R},$ where $R^\top$ and $\ov{R}$ denote the transpose and  complex-conjugate of $R,$ respectively.
Thus, a choice of $\theta^1$ adapted via the matrix $R$  is given by
\begin{equation}\label{eq:theta1-diagonalized}
  \theta^1=\ri \sqrt{\rho_{1\ov 1}}\exd z^1+\ri\tfrac{ \rho_{2\ov 1}}{\sqrt{\rho_{1\ov 1}}}\exd z^2.
\end{equation}

 Completing $\theta_1$ to a $3$-adapted coframe $(\theta_1,\theta_2)$ determines a section  $s\colon M\to \mathcal{G}$ of the principal $\mathrm{U}(1)$-bundle $\cG\to M$ in Theorem \ref{thm:cart-geom-descr}. For the remainder of the section we will work with the pull-back 1-forms  $(s^*\theta^1,s^*\theta^2,s^*\psi)$ in Theorem \ref{thm:cart-geom-descr} and, by abuse of notation, will drop the pull-back symbol $s^*$ in our expressions. As result, the pull-back of the  $1$-forms $\theta^2$ and $\psi$ can be expressed as
\begin{equation}\label{eq:theta2-psi}
  \theta^2=A_i\exd z^i
  \quad\mbox{ and }\quad
   \psi=B_i\exd z^i+\ov{B_i}\exd\ov{z^i}
  \end{equation}
  for some functions $A_1,A_2,B_1,B_2$ on $M$ that depend  on $(z^1,z^2,\ov{z^1},\ov{z^2})$.

  With such $\theta^1$,  $\theta^2$,  and $\psi$,  the identity $(\exd\theta^1-\ov{\theta^1}\w\theta^2)\w\theta^1=0$ from \eqref{eq:Cartan-conn-curv} determines $A_2$ to be
  \[A_2=A_1\tfrac{\rho_{2\ov 1}}{\rho_{1\ov 1}}-\left(\tfrac{ \rho_{2\ov 1}} {\rho_{1\ov 1}}\right)_{\ov{z^1}}.\]
  Subsequently, $\exd\theta^1+\ri\psi\w\theta^1-\ov{\theta^1}\w\theta^2=0$  determines $B_1$ and $B_2$ as
   \[B_1=-\ri \ov{A_1}-\tfrac\ri 2 \tfrac{\rho_{11\ov 1}}{\rho_{1\ov 1}},\quad B_2=\ri\tfrac{\rho_{2\ov 1}}{\rho_{1\ov 1}} \ov{A_1}-\tfrac \ri 2\tfrac{\rho_{12\ov 1}}{\rho_{1\ov 1}}. \]
Lastly, the identity  $(\exd\theta^2+2\ri\psi\w\theta^2)\w\theta^1=0$ from \eqref{eq:Cartan-conn-curv} is used to find $A_1$ to be
\[
  \begin{aligned}
        A_1=&\tfrac 13\left(\ln C\right)_{\ov{z^1}},\\
    \end{aligned}
  \]
  where
  \[C=\tfrac{1}{\rho_{1\ov 1}}\left(\tfrac{\rho_{2\ov 1}}{\rho_{1\ov 1}}\right)_{\ov{z^1}}=\tfrac{1}{\rho_{1\ov 1}^{3}}\left(\rho_{2\ov 1\ov 1}\rho_{1\ov 1}-\rho_{1\ov 1\ov 1}\rho_{2\ov 1}\right).\]
  In terms of the 3-adapted coframe above, the structure functions $T_1 $ and $T_2$ can be  computed in terms of the fifth jet of $\rho$ to be
  \[
    \begin{aligned}
       T_1&=-\tfrac{1}{3\rho_{1\ov 1}}\left(\ln C\right)_{\ov{z^1 z^1}}+\tfrac{2}{9\rho_{1\ov 1}}\left(\ln C\right)_{\ov{z^1}}\left((\ln C)_{\ov{z^1}}+\tfrac{3 \rho_{1\ov 1\ov 1}}{2\rho_{1\ov 1}}\right),\\      
    T_2&=  \tfrac{\ri\rho_{2\ov 1}}{3 C\rho_{1\ov 1}^{5/2}}\left(\ln C\right)_{\ov{z^1} z^1}
    -\tfrac{\ri}{3C\rho_{1\ov 1}^{3/2}}
    \left(\ln C\right)_{\ov{z^1} z^2}-\tfrac{\ri}{\rho_{1\ov 1}^{1/2}}\left(\ln C\right)_{z^1}
    -\tfrac{2\ri}{3\rho_{1\ov 1}^{1/2}}\left(\ln \ov{C}\right)_{z^1}-\tfrac{2\ri \rho_{11\ov 1}}{\rho_{1\ov 1}^{3/2}}.
    \end{aligned}
  \]

\begin{remark}\label{rem: removing singularities}
The formulas for $A_i$, $B_i$, and $T_i$ have singularities at the origin if and only if $\rho_{1\bar 2}\rho_{1\bar 1 \bar 1}=\rho_{1 \bar 1\bar 2}\rho_{1\bar 1}$. Letting $\rho^\prime$ denote the transformation of $\rho$ achieved by \eqref{eqn: removing coordinate expression singularities}, we have
\[
\rho^\prime_{1\bar 2}(0)\rho^\prime_{1\bar 1 \bar 1}(0)=\rho_{1\bar 2}(0)\rho_{1\bar 1 \bar 1}(0) 
\quad\mbox{ and }\quad
\rho^\prime_{1 \bar 1\bar 2}(0)\rho^\prime_{1\bar 1}(0)=\big(\rho_{1 \bar 1\bar 2}(0)+\epsilon \rho_{1\bar 1}(0)\big)\rho_{1\bar 1}(0),
\]
so this section's formulas computed with respect to the coordinates achieved by \eqref{eqn: removing coordinate expression singularities} have no singularities at the origin for all but one value of $\epsilon\in\mathbb{C}$ because $\rho_{1\bar{1}}(0)\neq 0$ by assumption.
\end{remark}

\begin{example}\label{exa:homog-surf-flat}[Flat model]
The 3-adapted coframe for the pre-K\"ahler complex surface with potential \eqref{rho-flat-example} is given by
\begin{equation}\label{eq:coframe-exa-flat-model}
      \begin{gathered}
      \theta^1=\frac{-\ri\exd z^1}{\sqrt{1-|z^2|^2}}-\frac{\ri(z^1\ov{z^2}+\ov{z^1})\exd z^2}{(1-|z^2|^2)^{3/2}},\quad \theta^2=\frac{\exd z^2}{|z^2|^2-1},\quad \psi=\frac{\ri\left(\ov{z^2}\exd z^2-z^2\exd\ov{z^2}\right)}{2|z^2|^2-2},
    \end{gathered}
\end{equation}
  using which, one obtains
  \begin{equation}\label{eq:T12-exa-flat-model}
    T_1=T_2=0.
      \end{equation}
\end{example}

\begin{example}[Exmaple \ref{ex: light cone potential}: continued]\label{exa:homog-surf}
Here we continue with pre-K\"ahler complex surfaces defined by potential $\rho_a$ in \eqref{rho-Doubrov-example}.
Following the discussion above, one finds  that  a 3-adapted coframe  is given by
\begin{equation}\label{eq:coframe-exa-homog-surf}
      \begin{gathered}
      \theta^1=\tfrac{\ri}{2} \sqrt{a(a-1)\rho_a}\left(\tfrac{1}{x^1+1}\exd z^1-\tfrac{1}{x^2+1}\exd z^2\right),\quad \theta^2=-\tfrac{a-2}{6(x^1+1)}\exd z^1+\tfrac{a+1}{6(x^2+1)}\exd z^2,\\
      \psi=-\tfrac{\ri(a-2)}{12(x^1+1)}\exd(z^1-\ov{z^1})+\tfrac{\ri(a+1)}{12(x^2+1)}\exd(z^2-\ov{z^2})
    \end{gathered}
\end{equation}
  and, furthermore,
  \begin{equation}\label{eq:T12-exa-homog-surf}
    \begin{aligned}
      T_1&=-\frac{(a+1)(a-2)}{9 a(a-1)}\rho_a^{-1},\qquad T_2&=0,\qquad T_3&=\frac{\ri(1+a)(a-2)}{9a(a-1)}\rho_a^{-1}.
          \end{aligned}
      \end{equation}
      A notable special case in this family is $a=2,-1$, which is equivalent to the pre-K\"{a}hler structure considered in \cite[Section 2]{mok2023elliptic}, and is exactly the case where $T_1=T_2=T_3=0$, i.e. locally equivalent to the flat model.
\end{example}

\subsection{Twistor bundle of symplectic connections on surfaces}
\label{sec:twist-bundle-characterization}
We start with the definition of a symplectic connection and its twistor bundle, which was first introduced in \cite{Twistors2}.
\begin{definition}\label{def:twist-bundle-equi}
Given a $2n$-dimensional symplectic manifold $(N,\sigma)$, a symplectic connection is given by a torsion-free linear connection $\nabla$ on $N$  such that $\nabla\sigma=0.$ The twistor bundle associated to $(\nabla,\sigma)$ is  denoted by $\tau\colon \cT^{(\nabla,\sigma)}\to N$ whose fiber $\cT^{(\nabla,\sigma)}_x:=\tau^{-1}(x)$   is the space of compatible almost complex structures on $T_xN,$ e.g. $J\colon T_xN\to T_xN, J^2=-\mathrm{Id}$ such that $g(\cdot,\cdot):=\sigma(\cdot,J\cdot)$ is a bilinear form of signature $(2n-2q,2q).$
\end{definition}
In this article we will be concerned with  symplectic connections on surfaces. Thus, in Definition \ref{def:twist-bundle-equi} one only needs to consider positive definite bilinear form $g.$   Taking a coframe $(\omega^1,\omega^2)$ on a surface $N,$ a torsion-free linear connection can be defined by its Christoffel symbols $\gamma^i_{jk}$ where
\begin{equation}\label{eq:nabla-connection-forms}
  \nabla\omega^i=-\omega^i_j\otimes \omega^j,\quad \omega^i_j=\gamma^i_{jk}\omega^k.
  \end{equation}
Choosing 1-forms $\omega^i$  so that $\sigma=\omega^1\w\omega^2,$ one obtains
\[\nabla\sigma=-(\omega^1_1+\omega^2_2)\otimes\sigma.\]
As a result, $\nabla\sigma=0$  implies  $\omega^1_1+\omega^2_2=0.$ Furthermore, from  \eqref{eq:nabla-connection-forms} it follows that
\begin{equation}\label{eq:d-omegai}
  \exd\omega^i=-\omega^i_j\w\omega^j,\quad \omega^1_1+\omega^2_2=0.
\end{equation}
By Definition \ref{def:twist-bundle-equi}, the group of admissible coframe transformations of a symplectic connection preserves $\sigma$ and, therefore, is $\mathrm{SL}(2,\RR).$ Hence, similarly to our discussion leading to \eqref{eq:str-eqns-strart}, by lifting the coframe to the principal $\mathrm{SL}(2,\RR)$-bundle $\ccA\to N$ of adapted coframes, one obtains a Cartan geometric description of 2-dimensional symplectic connections. In the statement below, by abuse of notation, we will not distinguish between  $\omega^i_j$'s in \eqref{eq:nabla-connection-forms} and their lift to $\ccA.$ 
\begin{proposition}\label{prop:cartan-conn-symp-conn}
  Symplectic connections  $\nabla$ on a symplectic surfaces $(N,\sigma)$ are in one-to-one correspondence with  Cartan geometries  $(\nu\colon \ccA\to N,\phi)$ of type $(\RR^2\rtimes\mathrm{SL}(2,\RR),\mathrm{SL}(2,\RR))$ where $\phi$  satisfies
  \begin{equation}\label{eq:phi-equiaffine-cartan-conn}
\phi=
  \begin{pmatrix}
    0 & 0 & 0\\
    \omega^1 & \omega^1_1 & \omega^1_2\\
    \omega^2 & \omega^2_1 & -\omega^1_1
  \end{pmatrix},\qquad 
\Phi:=\exd\phi+\phi\w\phi=
  \begin{pmatrix}
    0 & 0 & 0\\
    0 & R_{1}^1\omega^1\w\omega^2 & R^1_2\omega^1\w\omega^2\\
    0 & R_1^2\omega^1\w\omega^2 & -R_1^1\omega^1\w\omega^2
  \end{pmatrix}    
\end{equation}
for some real-valued functions $R_1^2,R_1^1,R_2^1$ on $\ccA,$ and $\nu^*\sigma=\omega^1\w\omega^2.$  The fundamental invariant of $(\ccA\to N,\phi)$ is given by the $\mathrm{SL}(2,\RR)$-invariant symmetric bilinear form
\begin{equation}\label{eq:R-binary-quadric}
  \bR:=s^*\left(R_2^1(\omega^2)^2+2R_1^1\omega^1\omega^2-R_1^2(\omega^1)^2\right)\in \Gamma(\mathrm{Sym}^2(T^*N)),
\end{equation}
for any section $s\colon N\to \ccA.$ The vanishing of $\bR$  implies that $(\nabla,\sigma)$ is locally flat, i.e. locally equivalent to the canonical symplectic connection on $(\RR^2\rtimes\mathrm{SL}(2,\RR))\slash\mathrm{SL}(2,\RR).$
\end{proposition}
Now we can state our pre-K\"ahler characterization for the twistor bundle of 2-dimensional symplectic connections. 

\begin{theorem}\label{thm:twistor-bundle-characterization}
  There is a one-to-one correspondence between symplectic connections on surfaces and pre-K\"ahler structures on complex surfaces satisfying $\bT_2=0.$
\end{theorem}
\begin{proof}
  For a 2-dimensional symplectic connection, $\nabla,$ on $(N,\sigma)$ at every point $x\in N$ the structure group, $\mathrm{SL}(2,\RR),$ acts transitively on compatible almost complex structures. 
  Using any section $s\colon N\to\mathcal{A}$, one obtains coframing $(s^*\omega^1,s^*\omega^2)$  on $N$. Via the action of the structure group on such coframing, the  almost complex structure represented in this coframe basis as
\[
J_0=
  \begin{pmatrix}
    0 & -1\\
    1 & 0
  \end{pmatrix}
\]
acts on each (co-)tangent space. The stabilizer of $J_0$ is $\mathrm{SO}(2,\RR)\subset \mathrm{SL}(2,\RR).$

As a result, by Definition \ref{def:twist-bundle-equi}, at $x\in N$ one has
\[\cT_x^{(\nabla,\sigma)}\cong \mathrm{SL}(2,\RR)\slash\mathrm{SO}(2,\RR)\cong \DD^{2},\]
where $\DD^2$ is the real 2-disk. Thus, the twistor bundle can be expressed as an associated bundle to the principal $\mathrm{SL}(2,\RR)$-bundle $\ccA\to N,$ in the following way
\[\cT=\ccA\slash\mathrm{SO}(2,\RR):=\ccA\times_{\mathrm{SL}(2,\RR)}\left(\mathrm{SL}(2,\RR)\slash\mathrm{SO}(2,\RR)\right).\]
Using the Cartan connection \eqref{eq:phi-equiaffine-cartan-conn}, the 1-forms $(\omega^1,\omega^2,\omega^1_2+\omega^2_1,\omega^1_1)$ are semi-basic with respect to the fibration $\ccA\to \cT^{(\nabla,\sigma)}.$ Furthermore, by structure equations \eqref{eq:phi-equiaffine-cartan-conn}, the complex-valued Pfaffian system $\{\omega^1+\ri\omega^2,\omega^1_1+\tfrac{\ri}{2}(\omega^1_2+\omega^2_1)\}$ is integrable. Hence, $\cT^{(\nabla,\sigma)}$ is a complex surface whose holomorphic 1-forms are given by $(\theta^1,\theta^2)$ where
\begin{equation}\label{eq:coframe-change-2}
  \theta^1=\omega^1+\ri\omega^2,\quad \theta^2=\omega^1_1+\tfrac{\ri}{2}(\omega^1_2+\omega^2_1).
  \end{equation}
With respect to this coframe,  $\sigma$ is expressed as $\sigma=\tfrac{\ri}{2}\theta^1\w\ov{\theta^1},$ the symmetric bilinear form $g=\theta^1\ov{\theta^1}$ is well-defined on $\cT^{(\nabla,\sigma)},$ and $\exd\theta^1\equiv \ov{\theta^1}\w\theta^2$ mod $\{\theta^1\}.$ Thus, $\cT^{(\nabla,\sigma)}$ is equipped with a 2-nondegenerate pre-K\"ahler structure and, by inspection, satisfies \eqref{eq:Cartan-conn-curv} and \eqref{eq:cartan-curv-prekahler} where
\begin{equation}\label{eq:psi-T1}
  \psi=\half(\omega^2_1-\omega^1_2),\quad T_1=-\tfrac{1}{4}(R_1^2+R^1_2-2\ri R_1^1),\quad T_2=0,\quad T_3=\tfrac{\ri}{4}(R^1_2-R^2_1).
  \end{equation}
Now, starting from a 2-nondegenerate pre-K\"ahler  complex surface $N$ satisfying $T_2=0,$ it is apparent from the discussion above how to  define a symplectic connection on the leaf space of the kernel of the pre-symplectic form $\omega.$

\end{proof}

 Recall from Theorem \ref{thm:cart-geom-descr} that, assuming $\bT_2=0,$ the condition $\bT_1=0$ implies $T_3=0$ and, therefore, the pre-K\"ahler structure is flat. At the level of the symplectic connection, one  obtains from  \eqref{eq:psi-T1}, that  $\bT_1=0$ implies $R^1_1=0$ and $R^1_2=-R^2_1$, and the vanishing $T_3=0$ implies $R^1_2=R^2_1=0,$ i.e. the symplectic connection is flat.

\begin{remark}
There is a class of 5-dimensional 2-nondegenerate CR structures arising as the twistor bundle of contact projective  3-manifolds (see Definition \ref{def:contact-proj-str} and the proof of Proposition \ref{prop:Bochner-flat-pre-Kahler}).  It would be interesting to find a characterization of such CR structures in the spirit of Theorem \ref{thm:twistor-bundle-characterization}. 
\end{remark}

\subsubsection{Remark on equiaffine surfaces and tubification} 
\label{sec:remark-equi-surf-tubification}
We would like to highlight a surprising link between Levi nondegenerate CR hypersurfaces and $2$-nondegenerate CR hypersurfaces in $\mathbb{C}^3$ determined by intermediary relationships to symplectic connections on surfaces.

The \emph{tubification} of an embedded equiaffine surface $N\subset \AAA^3$ results in a CR hypersurface  $N^\CC\subset\CC^3$, and the (non)degeneracy of equiaffine first fundamental form of $N$ determines Levi (non)degeneracy  of $N^\CC.$ 
It is known that an embedded equiaffine surface $N\subset \AAA^3$ is equipped with a natural symplectic connection. In particular,  the $\mathrm{SL}(2,\RR)$-invariant symmetric bilinear form  \eqref{eq:R-binary-quadric} is referred to as the \emph{equiaffine Weingarten form} for an  embedded nondegenerate equiaffine surface \cite[Equation (2.58)]{affine} and is determined by its Blaschke metric and Pick form. Pre-contactification of the twistor bundle of this symplectic connection defines a $2$-nondegenerate CR hypersurface in $\mathbb{C}^3$.

These observations result in a map from the space of tubular Levi nondegenerate CR hypersurfaces in $\mathbb{C}^3$ with distinguished \emph{tube translation symmetries} to $2$-nondegenerate CR hypersurfaces in $\mathbb{C}^3$ with a distinguished transverse symmetry. Starting with an embedded nondegenerate equiaffine surface $N\subset \AAA^3,$ the tube $N^\CC\subset \CC^3$ is, by construction, an $\mathbb{R}^3$-principle bundle over $N$ whose principle action generates a distinguished 3-dimensional abelian CR symmetry algebra of \emph{tube translations}. However, using the induced symplectic connection on  $N,$ its twistor bundle defines a 2-nondegenerate pre-K\"ahler structure, which, via pre-contactification, results in a  $2$-nondegenerate CR hypersurface $S\subset \mathbb{C}^3$ with a distinguished transverse infinitesimal symmetry. The resulting 2-nondegenerate structure on $S$ is independent of the equiaffine embedding, unlike the nondegenerate CR hypersurface $N^\CC.$  For example,  embedded equiaffine surfaces with nondegenerate first equiaffine fundamental form whose equiaffine Weingarten form vanishes are referred to as \emph{improper affine spheres} e.g. see \cite[Section 3.1.1]{affine}. As a result, the pre-K\"ahler structure defined on the twistor bundle of their corresponding symplectic connection is flat, despite the fact that the tubification of improper affine spheres results in inequivalent Levi nondegenerate CR 5-manifolds. That is, the structure on $S$ is always flat while there are many non-flat possibilities for $N^\mathbb{C}$ defined by improper affine spheres.

This relationship is unexpected and will be explored more in a forthcoming work.

\subsubsection{Remark on $\bT_1=0$}\label{sec:remark-bt_1=0}
There is a characterization of 2-nondegenerate pre-K\"ahler complex surfaces satisfying $\bT_1=0$ which we do not intend to fully elaborate on here since it is rather technical and will be pursued elsewhere. In a nutshell, in the real analytic category, complexifying the manifold and using $(\eta^1,\eta^2,\eta^3,\eta^4)$ to denote the complexification of $(\theta^1,\theta^2,\ov{\theta^1},\ov{\theta^2}),$ the complixification of  structure equations \eqref{eq:dtheta12-on-cG} for any pre-K\"ahler complex surfaces shows that the  leaf space  of $\{\eta^1,\eta^2,\eta^3\},$ denoted by $Q,$ is a complex 3-fold with a contact distribution   $\cD:=\ker\{\eta^1\}$ endowed with a splitting $\cD=\ell_1\oplus\ell_2$ where $\ell_1=\langle\partial_{\eta^2}\rangle$ and $\ell_2=\langle\partial_{\eta^3}\rangle.$ As a result, $Q$ is endowed with a 3-dimensional \emph{complex pseudo-product structure}, also known as \emph{complex para-CR structure}. Furthermore,  $Q$ can be locally identified as $\PP T R$ where $R$ is the 2-dimensional local leaf space of $\{\eta^1,\eta^2\}$ which is said to be equipped with a \emph{complex path geometry}  and is locally  associated with the point equivalence class of a scalar complex second order ODE. Using the structure equations \eqref{eq:dtheta12-on-cG}, one obtains that the  bilinear form $(\eta^3)^2$ is well-defined up to a scale on the vertical bundle of $Q\to R.$ Such geometric structures are referred to as \emph{orthopath geometries} in \cite[Definition 3.4]{MS-cone} which in this case turn out to be \emph{variational orthopath geometries} \cite[Section 3.3]{MS-cone}.

If $\bT_1=0$ holds, it follows from  \eqref{eq:dtheta12-on-cG} that $\{\eta^2\}$ is integrable and, therefore, $Q$ is fibered over a complex curve which  locally corresponds to the \emph{fiber equivalence class} of  scalar second order ODEs as defined in, e.g. \cite{HK-ODE}.  Careful inspection shows  that   such complex second order ODEs are highly constrained and are given explicitly as
  \begin{equation}\label{eq:ODE-T1-0}
    y''=a_2(x,y)(y')^2+a_1(x,y)y'+a_0(x,y),\quad\text{where}\quad  \tfrac{\partial}{\partial y}a_1(x,y)-2 \tfrac{\partial}{\partial x}a_2(x,y)=0,
      \end{equation}
 for some functions $a_2,a_1,a_0$ on the complex surface $R.$ As a result, such a complex path geometry defines a complex projective structure on the complex surfaces $R$ whose projective holonomy algebra takes value in a maximal parabolic subalgebra of $\mathfrak{sl}(3,\CC).$

 Unlike Theorem \ref{thm:twistor-bundle-characterization}, although the complexification of each pre-K\"ahler complex surface uniquely determines a complex orthopath geometry, this correspondence is not one-to-one.  For instance,  the local generality of  complex scalar ODEs \eqref{eq:ODE-T1-0} under fiber equivalence relation $x\to \tilde x=\Upsilon(x)$ and $y\to \tilde y=\chi(x,y)$ is clearly seen to be 1 function of 2 variables, i.e.  the function $\chi(x,y)$ in the pseudo-group of fiber transformations can be used to translate  $a_2(x,y)$  to zero which reduces the generality of the ODEs \eqref{eq:ODE-T1-0} to the function $a_0(x,y)$. Moreover, using the gauge freedom $\Upsilon(x),$ one can put $a_1=0$  and, up to fiber preserving transformations, complex scalar ODEs   $y''=a_0(x,y)$ are  associated to pre-K\"ahler condition $\bT_1=0.$ Such functional generality is strictly smaller than the local generality of pre-K\"ahler structures with $\bT_1=0$ whose generality, using Cartan--K\"ahler analysis, are found to be  3 functions of 2 variables, assuming real analyticity.

 \section{Symmetry reductions of homogeneous 2-nondegenerate CR 5-manifolds}\label{sec:symm-reduct-homog}
 In this section we start by describing pre-Sasakian structures arising from   flat 2-nondegenerate CR 5-manifolds using a Cartan geometric approach. The resulting pre-K\"ahler structures satisfy $\bT_2=0,$ and we characterize the symplectic connections they correspond to, showing also that they satisfy a certain criticality condition for symplectic connections. We then study pre-Sasakian structures defined by homogeneous 2-nondegenerate CR 5-manifolds, which will be used to show that the flat model is the only homogeneous 2-nondegenerate pre-K\"ahler complex surface, up to local equivalence.

\subsection{The flat model}
\label{sec:symm-reduct-flat}

 Our objective here is to study a condition for 2-nondegenerate pre-K\"ahler {\color{violet}geometry} that is analogous to the vanishing of the Bochner tensor \cite{Bochner} in K\"ahler geometry, restricting ourselves to the case of complex surfaces. As was originally observed in  \cite{webster1978pseudo},  Bochner-flat (pseudo-)K\"ahler metrics are locally in one-to-one correspondence with the flat CR structure of appropriate signature together with a choice of infinitesimal symmetry, which, in Webster's terminology, coincides with pseudo-Hermitian structures with vanishing torsion and vanishing 4th order Chern--Moser curvature tensor. In other words, Bochner-flat K\"ahler metrics can be viewed as  Sasakian structures whose corresponding CR structures are flat.

From this viewpoint, in dimension 4 there is a natural  pre-K\"ahler analogue of Bochner-flatness defined as those 2-nondegenerate pre-K\"ahler structures that arise via a symmetry reduction of a flat 2-nondegenerate CR 5-manifold. This is due to the fact that there is a well-posed notion of \emph{flatness} for 5-dimensional $2$-nondegenerate CR structures of hypersurface type, distinguished by vanishing of the curvature invariant defined in \cite{isaev2013reduction}. Such flat structures are locally unique \cite[Corollary 5.1]{isaev2013reduction}, and we recall their local description in the following proposition.

\begin{proposition}\label{prop:flat-model-CR-2nondeg}
  Any flat 5-dimensional 2-nondegenerate CR structure $\widetilde M$ is locally equivalent to $\mathrm{Sp}(4,\RR)\slash H$ where $H=H_0\ltimes H_+\subset P_1,$  $H_0=\mathrm{CO}(2,\RR),$  $H_+\cong \mathrm{Heis}(3)$ is the 3-dimensional Heisenberg subgroup, and $P_1$ is the contact parabolic subgroup in $\mathrm{Sp}(4,\RR)$. A well-known (local) coordinate representation of this hypersurface is $\Im(w)=\rho$, where $\rho$ is given by \eqref{rho-flat-example}. 
\end{proposition}

\begin{remark}\label{rem: general k-nondegenerate flat}
An analogous notion of \emph{flat model} is not yet developed for general $k$-nondegenerate CR structures in higher dimensions. Even in the most studied case of $2$-nondegenerate hypersurfaces, the picture is rather complex, although, for particular values of certain invariants, termed \emph{CR symbols}, there is a well-defined notion of \emph{flat models}, e.g. see \cite[Section 6]{SZ23}. Demonstrating their abundance, there are $9$ homogeneous models in dimension $7$, $20$ are known in dimension $9$, and $40$ are known in dimension $11$, \cite{sykes2025}. Each such model gives rise to distinct homogeneous $2$-nondegenerate pre-K\"{a}hler structures via the correspondence in Theorem \ref{thm: correspondence theorem}.
\end{remark}

Let us write  the right invariant Maurer--Cartan forms on $\mathrm{Sp}(4,\RR)$ as
\begin{equation}\label{eq:flat-5D-2nondeg-CR}
      \def\arraystretch{1.5}
\eta=  \begin{pmatrix}
    -\half\psi_2 & \tfrac 14{\alpha_1} & \tfrac 14 \ov{\alpha_1} & \half \alpha_0\\
    {\ov{\widetilde\theta^1}} & -\ri\psi_1 & \ov{\check\theta^2} & \half \ri\ov{\alpha_1}\\
    \widetilde\theta^1 & \check\theta^2 & \ri\psi_1 & -\half \ri{\alpha_1}\\
    2\widetilde\theta^0& -\half \ri\widetilde\theta^1 & \half \ri\ov{\widetilde\theta^1} & \half\psi_2
  \end{pmatrix},
\end{equation}
where $(\widetilde\theta^0,\widetilde\theta^1,\check\theta^2,\ov{\widetilde\theta^1},\ov{\check\theta^2})$ are semi-basic with respect to the fibration $\mathrm{Sp}(4,\RR)\to \widetilde M=\mathrm{Sp}(4,\RR)\slash H.$

Identifying the algebra of infinitesimal symmetries with $\mathfrak{sp}(4,\RR)$, let $V\in \mathfrak{sp}(4,\RR)$ be an infinitesimal symmetry for the CR structure, which implies
\begin{equation}\label{eq:infin-symmetry}
  \cL_V\eta=0.
\end{equation}
Furthermore, $V$ being transverse means  $V\im\wt\theta^0\neq 0$ on $\mathrm{Sp}(4,\RR)$ which, after taking a section $s\colon\widetilde M\to \mathrm{Sp}(4,\RR),$ implies that $V$ is transverse to the corank 1 distribution $\ker\{s^*\wt\theta^0\}$ on $\wt M.$ This transversality occurs almost everywhere, which is sufficient for this section's local analysis.

In the coordinate system with respect to which $\rho$ is given by \eqref{rho-flat-example}, there is a section  $s\colon \widetilde M\to \mathrm{Sp}(4,\RR)$ such that
\begin{equation}\label{eq:local-coframe-2nondeg-CR-5D}
  \begin{gathered}
    s^*\wt\theta^0=\,\exd \Re(w)+\Re\left(\frac{\ri(z^1\ov{z^2}+\ov{z^1})}{4(|z^2|^2-1)}\exd z^1-\frac{\ri(z^1\ov{z^2}+\ov{z^1})^2}{8(|z^2|^2-1)^2}\exd z^2\right),\\
    	s^*\wt\theta^1=\, -\frac{\ri}{\sqrt{1 - |z^{2}|^2}} \exd z^{1}-\frac{\ri(z^{1} \overline{z^{2}} + \overline{z^{1}})}{\left(1 - |z^{2}|^2\right)^{\frac{3}{2}}} \exd z^{2},
   	\quad\mbox{ and }\quad
	s^*\check\theta^2=\frac{1}{|z^{2}|^2-1} \exd z^{2}.
  \end{gathered}
\end{equation}

By lifting the coframe above via the action of $H,$ we obtain a coordinate system adapted to the fibration  $\mathrm{Sp}(4,\RR)\to\wt M.$ 
Let us represent $H$ as
\begin{equation}\label{eq:strgroup-H}
  H=\begin{pmatrix}
      \tfrac 1{r_2} &  0 & 0 & 0\\
      0 & e^{-\ri r_1} & 0 & 0 \\
      0 & 0 & e^{\ri r_1} & 0\\
      0 & 0 & 0 & r_2
    \end{pmatrix}\begin{pmatrix}
      1 & \half {s_1}  & \half \ov{s_1} & r_0\\
      0 & 1 & 0 & \ri \ov{s_1} \\
      0 & 0 & 1 & -\ri{s_1}\\
      0 & 0 & 0 & 1
    \end{pmatrix},
  \end{equation}
  where $ r_0,r_1,r_2\in\RR$ and $s_1\in\CC.$
Recall that on  $\mathrm{Sp}(4,\RR)$ the transformation of $\eta$ by the action of $\bh\in H$ along the fiber at $p\in\mathrm{Sp}(4,\RR)$ is given by
\begin{equation}\label{eq:gauge-change}
  \eta(p)\to \eta(\bh^{-1}p)=\bh^{-1}\eta\bh+\bh^{-1}\exd\bh.
\end{equation}
As a result of \eqref{eq:gauge-change}, the lifted coframe $\bh^{-1}\eta\bh,$ together with $\bh^{-1}\exd\bh,$ result in  a coframing on  $\mathrm{Sp}(4,\RR)$  in terms of  local coordinates $(z_0,z_1,z_2,\ov{z_1},\ov{z_2},r_0,r_1,r_2,s_1,\ov{s_1}).$ Using \eqref{eq:gauge-change}, the action of $H$ on  $\wt\theta^0$ is found to be
\[\wt\theta^0\to \tfrac{1}{r_2^2}\wt\theta^0.\]
Since $V$ is transverse to $\ker\{\wt\theta^0\},$ there is a unique choice of $r_2$ for which 
  \begin{equation}  \label{eq:disting-theta0}
    V\im \widetilde\theta^0=1,
    \end{equation}
    The function $r_2\in C^\infty(M,\RR)$ defines a section $t_1\colon  \wt\cG_1\to \mathrm{Sp}(4,\RR)$ where  $\wt\cG_1:=\mathrm{Sp}(4,\RR)\slash \mathrm{\RR}^*$ is a principal $(\mathrm{SO}(2,\RR)\ltimes H_+)$-bundle over $\wt M.$ Since $r_2$ is a function on $\wt M,$ its differential on $\wt\cG_1$ is semi-basic and therefore, via pull-back, one arrives at
    \[
      \begin{aligned}
            t_1^*\psi_2=& X_{20}\widetilde\theta^0+X_{21}\wt\theta^1+X_{22}\check\theta^2+\ov{X_{21}}\ov{\widetilde\theta^1}+\ov{X_{22}}\ov{\check\theta^2},
\end{aligned}
\]
for some functions $X_{ij}$ on $\wt\cG_1$ wherein, by abuse of notation, we have dropped the pull-back by $t_1$ on the right hand side.

Similarly, using the  pull-back of the transformation \eqref{eq:gauge-change} to $\wt\cG_1,$ one obtains the action of $\mathrm{SO}(2,\RR)\ltimes H_+$ on $X_{20}$ and $\wt\theta^1$ to be
\[
  \begin{aligned}
    X_{20}&\to  X_{20}+\Re(\ri X_{22} s_1^2 e^{2\ri r_1}-2\ri e^{\ri r_1} X_{21}{s_1})+4r_0,\\
    \wt\theta^1&\to e^{-\ri r_1}\wt\theta^1+2\ri s_1\wt\theta^0.
    \end{aligned}
\]
Noting that $X_{20}$ is $\mathbb{R}$-valued, by choosing $(r_0,s_1)$ appropriately,  a principal $\mathrm{SO}(2,\RR)$-bundle $\cG_2\to \wt M$ can be defined as
\begin{equation}\label{eq:G_2-def}
  \wt\cG_2=\{p\in \wt\cG_1\ | \ X_{20}(p)=0, V\im \wt\theta^1(p)=0\},
  \end{equation}
  which can be viewed as a section $t_2\colon\wt\cG_2\to\wt\cG_1.$ Defining $t=t_1\circ t_2\colon\wt\cG_2\to \mathrm{Sp}(4,\RR),$ one obtains
  \begin{equation}\label{eq:reductions-G2}
      \begin{aligned} t^*\alpha_0=&X_{00}\widetilde\theta^0+X_{01}\wt\theta^1+X_{02}\check\theta^2+\ov{X_{01}}\ov{\widetilde\theta^1}+\ov{X_{02}}\ov{\check\theta^2}+ X_{0}\psi_1,\\ t^*\alpha_1=&X_{10}\widetilde\theta^0+X_{11}\wt\theta^1+X_{12}\check\theta^2+{Y_{11}}\ov{\widetilde\theta^1}+{Y_{02}}\ov{\check\theta^2}+ X_{1}\psi_1,\\
 t^*\psi_2=&X_{21}\wt\theta^1+X_{22}\check\theta^2+\ov{X_{21}}\ov{\widetilde\theta^1}+\ov{X_{22}}\ov{\check\theta^2}
\end{aligned}
  \end{equation}
  for some functions $X_i,X_{ij},Y_{ij}$ on $\wt\cG_2$.

Inspecting \eqref{eq:infin-symmetry} for the distinguished 1-form $\wt\theta^0$ satisfying \eqref{eq:disting-theta0}, it follows that
\begin{equation}\label{eq:Lv-theta0}
  0=\cL_V\wt\theta^0=(\exd\iota_V+\iota_V\exd)\wt\theta^0=\iota_V\exd\wt\theta^0=\iota_V(-\psi_2\w\wt\theta^0+\half\ri\wt\theta^1\w\ov{\wt\theta^1})
\equiv \psi_2\mod\{\wt\theta^0\}.
\end{equation}
Consequently,  using \eqref{eq:reductions-G2},   one obtains
\begin{equation}\label{eq:psi2-zero}
  t^*\psi_2=0.
\end{equation}
Moreover, using the Maurer--Cartan equations, it follows that
\[
  \begin{aligned}
    0=&\exd(t^*\psi_2)=t^*(\exd\psi_2)=t^*(2\alpha_0\w\wt\theta^0+\half \alpha_1\w\ov{\wt\theta^1}+\half\ov{\alpha_1}\w\wt\theta^1).
  \end{aligned}
\]
Solving the relations arising from the insertion of the first two expressions of \eqref{eq:reductions-G2} in the equation above, the expressions in \eqref{eq:reductions-G2} simplify to
  \begin{equation}\label{eq:reductions-G2-dpsi2-0}
      \begin{aligned} t^*\alpha_0=&X_{00}\widetilde\theta^0+X_{01}\wt\theta^1+\ov{X_{01}}\ov{\widetilde\theta^1},\\ t^*\alpha_1=&4\ov{X_{01}}\widetilde\theta^0+X_{11}\wt\theta^1+{Y_{11}}\ov{\widetilde\theta^1},
\end{aligned}
\end{equation}
where $X_{00}, Y_{11}\in C^\infty(\wt\cG_2,\RR)$ and $X_{01},X_{11}\in C^{\infty}(\wt\cG_2,\CC).$

  Using \eqref{eq:disting-theta0} and \eqref{eq:G_2-def},  the infinitesimal symmetry  $V\in\Gamma T\wt\cG_2$ with respect to the frame that is dual to   $(\wt\theta^0,\wt\theta^1,\ov{\wt\theta^1},\check\theta^2,\ov{\check\theta^2},\psi_1)$ can be expressed as
  \begin{equation}  V=\tfrac{\partial}{\partial\wt\theta^0}+Z_2\tfrac{\partial}{\partial\check\theta^2}+\ov{Z_2}\tfrac{\partial}{\partial\ov{\check\theta^2}}+Z_1\tfrac{\partial}{\partial{\psi}_1}
  \end{equation}
  for  $Z_2\in C^\infty(\wt\cG_2,\CC)$ and $Z_1\in C^\infty(\wt\cG_2,\RR).$ In order to rectify $V,$ we transform the coframe on $\wt\cG_2$ as follows:
  \begin{equation}\label{eq:coframe-trans-rectify}
    \check\theta^2\to \wt\theta^2:=\check\theta^2-Z_2\wt\theta^0,\quad \psi_1\to\wt\psi_1:=\psi_1-Z_1\wt\theta^0.
  \end{equation}
  The 1-forms $(\wt\theta^0,\wt\theta^1,\ov{\wt\theta^1},\wt\theta^2,\ov{\wt\theta^2},\wt\psi_1)$ still define a coframe on $\wt\cG_2$ with respect to which one now has
  \begin{equation}\label{eq:rectified-v}
    V=\tfrac{\partial}{\partial\wt\theta^0}.
  \end{equation}
  Inspecting \eqref{eq:infin-symmetry} on $\wt\cG_2$ for the 1-form entries $\wt\theta^0$ and $\wt\theta^1$ of $\eta,$  as given in \eqref{eq:flat-5D-2nondeg-CR},  following computations similar to \eqref{eq:Lv-theta0} and using $V\im t^*\wt\theta^1=0$,  one obtains
  \[X_{11}=-Z_1,\quad Y_{11}=\ri Z_2.\]
  Computing  \eqref{eq:infin-symmetry} on $\wt\cG_2$ for all other entries of $\eta$ will determine the first jet of $Z_1$ and $Z_2$ to be
  \begin{equation}\label{eq:1st-jet-Z1-Z2}
    \begin{aligned}
      \exd Z_1=& \ri X_{01}\wt\theta^1-\ri\ov{X_{01}}\ov{\wt\theta^1}+\ri\ov{Z_2}\wt\theta^2-\ri Z_2\ov{\wt\theta^2}\\
      \exd Z_2=&-2\ov{X_{01}}\wt\theta^1+2\ri Z_1\wt\theta^2-2\ri Z_2\wt\psi_1.
    \end{aligned}
  \end{equation}
  Lastly, by the pull-back of the Maurer--Cartan equations to $\wt\cG_2$,   the first jets of $X_{00}\in C^\infty(\wt\cG_2,\RR)$ and $X_{01}\in C^\infty(\wt\cG_2,\CC)$ in \eqref{eq:reductions-G2-dpsi2-0} are found using 
  \[\exd\alpha_{1}=(\half\psi_2-\ri\psi_1)\w\alpha_1-\ov{\alpha_1}\w\theta^2+\ri\alpha_0\w\theta^1,\quad \exd\alpha_0=\psi_2\w\alpha_0-\tfrac{\ri}{2}\alpha_1\w\ov{\alpha_1},\]
as result of which, one obtains
\begin{equation}\label{eq:1st-jet-X00-X01}
  \begin{aligned}
    \exd X_{01}=&\tfrac 14\ri(  Z_1^2-  Z_2 \ov{Z_2}+  X_{00}) \ov{\wt\theta^1} + \ov{X_{01}} \ov{\wt\theta^2} +  \ri X_{01} \wt\psi_1,\\
    \exd X_{00}=&2( \ri X_{01} Z_1+\ov{X_{01}Z_2}) \wt\theta^1+2( Z_2X_{01} - \ri Z_1 \ov{X_{01}}) \ov{\wt\theta^1}.
  \end{aligned}
\end{equation}
The differential relations \eqref{eq:1st-jet-Z1-Z2} and \eqref{eq:1st-jet-X00-X01} give  a closed Pfaffian system on $\wt\cG_2$. Furthermore, using the definition of coframe derivatives in \ref{sec:conventions}, the equations \eqref{eq:1st-jet-Z1-Z2} and \eqref{eq:1st-jet-X00-X01} imply
\begin{equation}\label{eq:cof-deriv-Z1-Z2}
  Z_1=-\half \ri Z_{2;2},\quad X_{01}=-\half \ov{Z_{2; 1}},\quad X_{00}=-2\ri Z_{2;11}-Z_1^2+Z_2\ov{Z_2}.
\end{equation}

\subsection{Critical and special symplectic connections}  
\label{sec:crit-spec-sympl}

To describe pre-K\"ahler structures arising via  symmetry reductions of flat 2-nondegenerate CR 5-manifolds, let $\cG$ be the leaf space of the integral curves of the infinitesimal symmetry $V,$ defined via the quotient map $q\colon \wt\cG_2\to \cG.$ Using \eqref{eq:infin-symmetry} and the rectification \eqref{eq:rectified-v}, it follows that there is a coframe $(\theta^0,\theta^a,\ov{\theta^a},\psi)$ on $\cG$ such that
\[q^*\theta^0=\wt\theta^0,\quad q^*\theta^1=\wt\theta^1,\quad q^*\theta^2=\wt\theta^2,\quad q^*\psi=\wt\psi_1.\]
By our preceding discussions, such coframing is defined on a principal $\mathrm{U}(1)$-bundle $\cG\to M,$ where $M$ is the leaf space of $\{\theta^1,\theta^2,\ov{\theta^1},\ov{\theta^2}\},$ and therefore, by Theorem \ref{thm:cart-geom-descr},  canonically defines a 2-nondegenerate  pre-K\"ahler complex surface.

Using the structure equations on $\wt\cG_2,$ it is a matter of straightforward computation to show that the structure functions of  such pre-K\"ahler structures, as given in \eqref{eq:cartan-curv-prekahler}, are related to functions $Z_1,Z_2,X_{00},X_{01}$ in the following way
\begin{equation}\label{eq:invariants}
  T_1=-\ri Z_2,\quad T_2=0,\quad T_3=\half  Z_{2;2},\quad T_{1;1}=-\ri Z_{2;1},\quad T_{1;11}=-\ri {Z_{2;11}},
\end{equation}
where we have used the relations \eqref{eq:cof-deriv-Z1-Z2}. Subsequently, \eqref{eq:1st-jet-Z1-Z2} and \eqref{eq:1st-jet-X00-X01} give the closed Pfaffian system defined by the differentials of the  structure functions of such pre-K\"ahler structures, which is 
\begin{equation}\label{eq:Bianchies-closed-system}
    \begin{aligned}
   & \exd T_1+2\ri T_1\psi+2\ri T_3\theta^2-T_{1;1}\theta^1=0\\
   & \exd T_3+\half \ri \ov{T_{1;1}}\theta^1-\ri T_1\ov{\theta^2}+\half \ri T_{1;1}\ov{\theta^1}-\ri\ov{T_1}\theta^2=0\\
   & \exd T_{1;1}+\ri T_{1;1}\psi+\ov{T_{11}}\theta^2-T_{1;11}\theta^1=0\\
 &  \exd T_{1;11}-(\ri \ov{T_{1;1}}T_3-T_{1;1}\ov{T_1})\theta^1-(\ri T_3T_{1;1}-\ov{T_{1;1}}T_1)\ov{\theta^1}=0.
  \end{aligned}
\end{equation}

By Theorem \ref{thm:twistor-bundle-characterization}, since $T_2=0$, this class of pre-K\"ahler complex surfaces correspond to symplectic connections on the leaf space of the pre-symplectic kernel. To characterize this class of symplectic connections, we first recall some facts about contact projective structures from \cite{Fox}.
\begin{definition}\label{def:contact-proj-str}
Given a contact manifold $\wt N,$ a \emph{contact geodesic} of a linear connection $\nabla$ on $\wt N$ is a  geodesic curve with the property that its tangent vector is in the contact distribution everywhere. A \emph{contact projective structure}, $[\nabla],$ on $\wt N$ is an equivalence class of linear connections  on $\wt N$ whose contact geodesics coincide as unparameterized curves and are defined  along every contact direction of each point of $\wt N.$
\end{definition}
Contact projective structures are examples of parabolic geometries whose flat model is given by the homogeneous space $\mathrm{Sp}(2n,\RR)\slash P_1,$ where $\mathrm{Sp}(2n,\RR)$ is the  group of linear automorphisms of a real $2n$-dimensional symplectic vector space $\VV$ and  $P_1$ is the parabolic subgroup that stabilizes a 1-dimensional subspace of $\VV.$ 

Using the notion of special symplectic connections, as defined in \cite{CS-special}, we give the following definition.
 \begin{definition}\label{def:critical-equiaffine-surface}
A 2-dimensional symplectic connection is called \emph{special} if it is locally a symmetry reduction of the flat contact projective 3-manifold, $\mathrm{Sp}(4,\RR)\slash P_1,$ by an  infinitesimal symmetry that is transverse to the contact distribution. 
\end{definition}
Now we can give our characterization of pre-K\"ahler complex surfaces that are symmetry reductions of a flat 2-nondegenerate CR 5-manifold.
\begin{proposition}\label{prop:Bochner-flat-pre-Kahler}
There is a one-to-one correspondence between the family of 2-nondegenerate pre-K\"ahler complex surfaces locally defined via a symmetry reduction of a flat 2-nondegenerate CR 5-manifold and the 2-parameter family of 2-dimensional special symplectic connections, which, among symplectic connections, are characterized by the vanishing of the $\mathrm{SL}(2,\RR)$-invariant symmetric trilinear form
  \begin{equation}\label{eq:binary-cubic}
    \bC:=R^1_{2;2}(\omega^2)^3+(R^1_{2;1}+2R^1_{1;2})(\omega^2)^2\omega^1+(2R^1_{1;1}-R^2_{1;2})\omega^2(\omega^1)^2- R^2_{1;1}(\omega^1)^3\in\Gamma(\mathrm{Sym}^3(T^*N)).
  \end{equation}
\end{proposition} 
\begin{proof}
  By our discussion above, the flat model for contact projective 3-manifolds is  $\widetilde N=\mathrm{Sp}(4,\RR)\slash P_1$ and, by Proposition \ref{prop:flat-model-CR-2nondeg}, the flat model for 2-nondegenerate CR 5-manifolds is $\widetilde M=\mathrm{Sp}(4,\RR)\slash H$ where $H\cong H_0\ltimes H_+\subset P_1\cong P_0\ltimes P_+,$   $P_0=\mathrm{GL}(2,\RR),H_0=\mathrm{CO}(2,\RR)$ and $H_+=P_+=\mathrm{Heis}(3).$

  As a result, assuming that the contact manifold is co-oriented, which is always satisfied locally, one arrives at the  fibration $\widetilde M\to \widetilde N$ with fibers  $P_1\slash H\cong \mathrm{GL}^+(2,\RR)\slash \mathrm{CO}(2,\RR)\cong \DD^2.$ In other words, following the same construction as in the proof of Theorem \ref{thm:twistor-bundle-characterization}, $\widetilde M$ can be viewed as the \emph{twistor bundle} of the contact projective structure on $\widetilde N,$ i.e. the bundle of all almost complex structures on the contact distribution of $\widetilde \cC\subset T\widetilde N$ that are compatible with the conformal symplectic 2-form on $\widetilde \cC.$  Following the same steps, one obtains that  the space of such almost complex structures at each tangent space gives rise to the 2-disk bundle $\widetilde M\to \widetilde N.$

  By construction, there is a canonical bijection between infinitesimal symmetries of the flat model of 3-dimensional  contact projective structures  and 2-nondegenerate CR structures on its twistor bundle. It is shown in \cite[Corollary C]{CS-special} that affine connections obtained via symmetry reductions of the flat contact projective structure in dimension $(2n+1)$ define  an $n$-parameter family of \textrm{special symplectic connections} (in the sense of Remark \ref{rem: coarse equivalence}). In our case, the symmetry reductions result in a 1-parameter family of symplectic connections on a  surface $N$, each of which preserve a volume form up to homothety, denoted by $[\sigma]$. A choice of volume form, $\sigma\in[\sigma],$ corresponds to a homothety factor, which determines a symplectic surface $(N,\sigma)$. By Theorem \ref{thm:twistor-bundle-characterization}, the twistor bundle of  such symplectic connections on a corresponding symplectic surface $(N,\sigma)$  uniquely define a 2-nondegenerate pre-K\"ahler complex surface, giving rise to a  2-parameter family \emph{geometrically distinct} of pre-K\"ahler structures (see Remark \ref{rem: coarse equivalence}). 

  For the last part of the theorem, firstly one can show that the binary cubic $\bC$ is invariantly defined either by  finding the induced $\mathrm{SL}(2,\RR)$-action on  $R^i_{j;k}$'s or, infinitesimally, using the Bianchi identities
  \begin{equation}\label{eq:1st-Bianchies}
    \begin{aligned}
      \exd R^1_{2}&=R^1_{2;i}\omega^i+2R^1_1\omega^1_2-2R^1_2\omega^1_1\\
      \exd R^1_1&=R^1_{1;i}\omega^i+R^1_2\omega^2_1-R^2_1\omega^1_2\\
      \exd R^2_1&=R^2_{1;i}\omega^i-2R^1_1\omega^2_1+2R^2_1\omega^1_1,
    \end{aligned}
  \end{equation}
  one can show that the Lie derivative of $\bC$ along the vertical tangent vectors  of  the principal bundle $\ccA\to N$ in Proposition \ref{prop:cartan-conn-symp-conn}, is zero.  The vanishing of $\bC$ implies that the first jet of $R^i_j$'s can be expressed as $R^1_{1;1}$ and $R^1_{1;2}.$
  Subsequently, using the structure equations \eqref{eq:phi-equiaffine-cartan-conn} and the identities  $\exd^2R^i_{j}=0,$ one obtains the higher Bianchi identities
  \begin{equation}\label{eq:2nd-Bianchies}
    \exd R^1_{1;1}=R^1_{1;12}\omega^2+R^1_{1;1}\omega^1_1-R^1_{1;2}\omega^2_1,\quad \exd R^1_{1;2}=R^1_{1;12}\omega^1-R^1_{1;2}\omega^1_1-R^1_{1;1}\omega^1_2.
  \end{equation}
  As a result, the second jet of $R^i_j$'s depend on one function, denoted by $R^1_{1;12}.$ Similarly, one obtains
  \begin{equation}\label{eq:3rd-Bianchi}
    \exd R^1_{1;12}=(R^2_1R^1_{1;2}-R^1_1R^1_{1;1})\omega^1-(R^1_{1}R^1_{1;2}+R^1_{2}R^1_{1;1})\omega^2.
      \end{equation}
      Hence, such symplectic connections define a closed Pfaffian system. Now it is an elemetary task to use the change of coframe \eqref{eq:coframe-change-2} and \eqref{eq:psi-T1} to show that the identities \eqref{eq:1st-Bianchies}, \eqref{eq:2nd-Bianchies} and \eqref{eq:3rd-Bianchi} coincide with \eqref{eq:Bianchies-closed-system} where
      \[T_1=-\tfrac{1}{4}(R_1^2+R^1_2-2\ri R_1^1),\quad T_2=0,\quad T_3=\tfrac{\ri}{4}(R^1_2-R^2_1),\quad T_{1;1}=\half (R^1_{1;2}+\ri R^1_{1;2}),\quad T_{1;11}=\half R^1_{1;12}.\]

\end{proof}
\begin{remark}\label{rmk:sym-red-other-2nondeg-CR}
  Although the first part of Proposition \ref{prop:Bochner-flat-pre-Kahler} reduces to a special case (via Theorem \ref{thm:twistor-bundle-characterization}) of what is studied in \cite{CS-special}, the Cartan geometric viewpoint we adopted to carry out the symmetry reduction in \ref{sec:symm-reduct-flat} has the advantage of being used for  a larger class of geometric structures that are neither  flat nor defined on a  contact manifold c.f. \cite{CS-contactification}. 
For example, it can be applied to the broad class of flat models of $2$-nondegenerate CR structures in higher dimensions discussed in Remark \ref{rem: general k-nondegenerate flat}. 
\end{remark}
\begin{example}[Exmaple \ref{ex: light cone potential}: continued]\label{exa:special-homog-equiaaffine}
 Recall that  pre-K\"ahler surfaces with potential $\rho_a$  in \eqref{rho-Doubrov-example},
  have 3-adapted coframes  \eqref{eq:coframe-exa-homog-surf}, with respect to which 
 $T_1$ and $T_2$ are given as \eqref{eq:T12-exa-homog-surf}.
We would like to find the values of $a$ for which such a pre-K\"ahler complex surface corresponds to a special symplectic connection. To do so, one can use  Bianchi identities \eqref{eq:Bianchies-closed-system} to determine the value of  $a$ for which they are satisfied. Alternatively, one can use the coframe transformation \eqref{eq:coframe-change-2} and \eqref{eq:psi-T1} to find adapted coframe and curvature of the corresponding symplectic connection and, subsequently, determine the value of $a$ for which the binary cubic $\bC$ in \eqref{eq:binary-cubic} is zero. Following the latter, recall that by Theorem \ref{thm:twistor-bundle-characterization}, when $\bT_2=0$ the principal $\mathrm{SL}(2,\RR)$-bundle $\ccA$ in Proposition \ref{prop:cartan-conn-symp-conn} can be identified with the principal $\mathrm{U}(1)$-bundle $\cG$ from Theorem \ref{thm:cart-geom-descr}. Thus, the section $s\colon M\to \cG$ used for the parametric computation in \ref{sec:param-expr} by the initial choice of $\theta^1,$ can be viewed as a local embedding   $s\colon M\to\ccA$. Applying the  coframe change \eqref{eq:coframe-change-2} and \eqref{eq:psi-T1}   to the 3-adapted coframe \eqref{eq:coframe-exa-homog-surf}, one obtains
\[s^*\omega^1=\half(\theta^1+\ov{\theta^1}),\quad s^*\omega^2=\tfrac{-\ri}{2}(\theta^1-\ov{\theta^1}),\quad s^*\omega^1_1=\half(\theta^2+\ov{\theta^2}),\quad s^*\omega^1_2=\ri(\ov{\theta^2}-\theta^2),\quad s^*\omega^2_1=0.\]
It is a matter of  straightforward computation to find that the binary quadric \eqref{eq:R-binary-quadric} and binary cubic \eqref{eq:binary-cubic} take the form
\begin{equation}\label{eq:R-C-rho_a}
  s^*\bR=\frac{4(a+1)(a-2)}{9a(a-1)\rho_a}\left(\omega^2\right)^2,\quad   s^*\bC=-\frac{16(a+1)(a-2)(2a-1)}{27\left(a(a-1)\rho_a\right)^{3/2}}\left(\omega^2\right)^3,\quad 
\end{equation}
As a result, the only possible values of $a$ for which $\bC=0$ are $-1,2$ and $\half.$  Since $\bR=0$ for $a=-1,2,$ these values define a flat pre-K\"ahler structure. When $a=\half,$  one has a non-flat special symplectic connection for which $T_3=\tfrac{1}{2\sqrt{(x^1+1)(x^2+1)}},$ and $T_{1;1}=T_{1;11}=0.$ 
\end{example}

\begin{remark}\label{rmk:Q-quartic-examples}
  It would be interesting to characterize  symplectic connections defined by the potential functions $\rho_a,$ for all $\half< a<2,$  in terms of the vanishing of some higher order  $\mathrm{SL}(2,\RR)$-invariants. For instance, consider the binary quartic $\bQ\in\Gamma(\mathrm{Sym}^4(T^*N))$ defined as
    \begin{multline}\label{eq:binary-quartic} \bQ=R^1_{2;22}(\omega^2)^4+(2R^1_{2;21}+2R^1_{1;22})(\omega^2)^3\omega^1\\+(4R^1_{1;21}-R^2_{1;22}+R^1_{2;11})(\omega^2)^2(\omega^1)^2
      +(2R^1_{1;11}-2R^2_{1;21})\omega^2(\omega^1)^3- R^2_{1;11}(\omega^1)^4,
    \end{multline}
    which is an absolute invariant. Carrying out the computations of  Example \ref{exa:special-homog-equiaaffine} to determine second coframe derivatives of $R^1_2$, one obtains
    \[s^*\bQ=\frac{32(a+1)(a-2)(2a-1)^2}{27\left(a(a-1)\rho_a\right)^{2}}\left(\omega^2\right)^4,\]
Subsequently, it follows that this class of symplectic connections satisfy the invariant condition
\[\bQ-\tfrac{6(2a-1)^2}{(a+1)(a-2)}(\bR)^2=0.\]
    \end{remark}

Now we describe yet another similarity between  this class of pre-K\"ahler complex surfaces and their K\"ahler counterpart. Although  the Bochner tensor of K\"ahler metrics on complex curves is always zero, in \cite{Bryant} a definition is given for Bochner-flatness that coincides with being \emph{extremal}, e.g. see \cite[Remark 1.2]{Rui}. In  \cite{Calabi}  a local description of  extremal K\"ahler metrics is given as those for which     the  Hamiltonian vector field defined by the scalar curvature is an  infinitesimal symmetry. Furthermore, it is shown in  \cite{CS-special} that such K\"ahler metrics on complex curves  correspond to symmetry reductions of flat CR 3-manifolds.

A condition analogous to being extremal has been studied for symplectic connections for which we refer the reader to \cite{Cahen1,Fox-affine} for an overview. To define the condition, we follow  \cite{Fox-affine} and restrict ourselves to the 2-dimensional case, although it holds in any dimension. Let $\mathsf{R}^i_{jkl}$ denote the entries of the curvature tensor for $\nabla.$ In terms of the connection forms $\omega^i_j$ in  \eqref{eq:d-omegai}, they are defined as 
\[ \exd\omega^i_j=-\omega^i_k\w\omega^k_j+\half \mathsf{R}^i_{jkl}\omega^k\w\omega^l.\]
The entries of the Ricci tensor of $\nabla$ are  $\mathsf{Ric}_{ij}:=\mathsf{R}^k_{ikj}.$ In terms of the coefficients of the  Cartan curvature \eqref{eq:phi-equiaffine-cartan-conn} one has
\begin{equation}\label{eq:Ricci-entries}
  \mathsf{Ric}_{11}=R^1_2,\quad \mathsf{Ric}_{22}=R^2_1,\quad \mathsf{Ric}_{12}=-\mathsf{Ric}_{21}=R^1_1.
  \end{equation}
Note that being symplectic implies that the Ricci tensor is symmetric.

Using the symplectic form to raise and lower indices,  define
\begin{equation}\label{eq:K-nabla}
  K_\nabla=\nabla^{i}\nabla^j\mathsf{Ric}_{ij}-\half \mathsf{Ric}^{ij}\mathsf{Ric}_{ij}+\tfrac 14\mathsf{R}^{ijkl}\mathsf{R}_{ijkl}.
\end{equation}
By inspection, one can check that in dimension two $-\half \mathsf{Ric}^{ij}\mathsf{Ric}_{ij}+\tfrac 14\mathsf{R}^{ijkl}\mathsf{R}_{ijkl}=0.$   Thus, \eqref{eq:K-nabla} simplifies to $K_\nabla=\nabla^{i}\nabla^j\mathsf{Ric}_{ij}.$  In terms of the structure functions \eqref{eq:phi-equiaffine-cartan-conn}, writing the symplectic form as $\omega=\half \ve_{ij}\omega^i\w\omega^j=\omega^1\w\omega^2,$ one can express $K_\nabla$ in terms of coframe derivatives, defined in \ref{sec:conventions}. Since  in the case of  linear connections one has $f_{;i}=\nabla_if,$ it easily follows that
\begin{equation}\label{eq:K-nabla-coframe-derivative}
  K_\nabla=R^1_{2;11}-2R^1_{1;21}-R^2_{1;22}.
\end{equation}

\begin{remark}
  Denoting the discriminant of the quadric  \eqref{eq:R-binary-quadric} as $\mathsf{Disc}(\bR)=(R^1_1)^2+R^1_2R^2_1$,  \eqref{eq:Ricci-entries} gives  \[\mathsf{Ric}^{ij}\mathsf{Ric}_{ij}=-2\mathsf{Disc}(\bR).\]
  Furthermore, writing $\nabla_i\mathsf{Ric}_{jk}=\mathsf{Ric}_{jk;i},$ the cubic \eqref{eq:binary-cubic} and quartic \eqref{eq:binary-quartic} can be expressed as
\[ \begin{aligned}
     \bC&=\mathsf{Ric}_{11;2}(\omega^2)^3 +(\mathsf{Ric}_{11;1}+2\mathsf{Ric}_{12;2})(\omega^2)^2\omega^1 -(2\mathsf{Ric}_{21;1}+\mathsf{Ric}_{22;2})\omega^2(\omega^1)^2- \mathsf{Ric}_{22;1}(\omega^1)^3.\\
     \bQ&=\mathsf{Ric}_{11;22}(\omega^2)^4+2(\mathsf{Ric}_{11;12}+\mathsf{Ric}_{12;22})(\omega^2)^3\omega^1+(4\mathsf{Ric}_{12;12}-\mathsf{Ric}_{22;22}+ \mathsf{Ric}_{11;11})(\omega^2)^2(\omega^1)^2\\
     &
      \ \ -2(\mathsf{Ric}_{22;21}+\mathsf{Ric}_{21;11})\omega^2(\omega^1)^3- \mathsf{Ric}_{22;11}(\omega^1)^4,
  \end{aligned}
\]
Note that using \eqref{eq:coframe-change-2} and \eqref{eq:psi-T1} one can express the invariants above in terms of  the corresponding pre-K\"ahler structure e.g. $\mathsf{Disc}(\bR)=4(\bT_1+(T_3)^2).$
\end{remark}

 \begin{definition}\label{def:special-equiaffine-surface}
Given a symplectic manifold $(N,\sigma),$ a symplectic connection, $\nabla,$ is called critical  if the Hamiltonian vector field defined by $K_\nabla$ is an infinitesimal symmetry of $\nabla.$ Moreover, when $K_\nabla$ is constant, the symplectic connection is called moment flat.
\end{definition}

Similar to the case of extremal K\"ahler metrics, it can be shown \cite{Cahen1,Fox-affine} that critical symplectic connections on a symplectic manifold $(N,\sigma)$ correspond to critical points  of the functional $\int_N K_\nabla^2\tfrac{\sigma^n}{n!}.$

\begin{proposition}\label{prop:crit-spec-sympl-1}
  Special symplectic connections on surfaces  are critical. 
\end{proposition}

\begin{proof}
  By Proposition \ref{prop:Bochner-flat-pre-Kahler},  for 2-dimensional special symplectic connections the cubic $\bC$ in \eqref{eq:binary-cubic} is zero and the structure functions satisfy \eqref{eq:1st-Bianchies}, \eqref{eq:2nd-Bianchies}, and \eqref{eq:3rd-Bianchi}. Directly inspecting these equation, using \eqref{eq:Ricci-entries},  one obtains \eqref{eq:K-nabla-coframe-derivative} simplifies to $K_\nabla=-6R_{1;12}^1$ and
  \begin{equation}
    \label{eq:Knabla-Ricci}
    \exd K_\nabla=-6\exd (R^1_{1;12})=-6(R^2_1R^1_{1;2}-R^1_1R^1_{1;1})\omega^1+6(R^1_{1}R^1_{1;2}+R^1_{2}R^1_{1;1})\omega^2.
      \end{equation}
  As a result, taking any section $s\colon M\to\ccA,$ the corresponding Hamiltonian vector field on $N$ is given by
  \[H_{K_\nabla}=6(R^1_{1}R^1_{1;2}+R^1_{2}R^1_{1;1})\tfrac{\partial}{\partial s^*\omega^1}+6(R^2_1R^1_{1;2}-R^1_1R^1_{1;1})\tfrac{\partial}{\partial s^*\omega^2},\]
  where $(\tfrac{\partial}{\partial s^*\omega^1},\tfrac{\partial}{\partial s^*\omega^2})$ is the frame on $M$ dual to the coframe $(s^*\omega^1,s^*\omega^2).$ 
  Using the connection $\nabla,$ the horizontal lift of $H_{K_\nabla}$ to $\ccA$  is found to be
  \begin{multline}
    \what H_{K_\nabla}=6(R^1_{1}R^1_{1;2}+R^1_{2}R^1_{1;1})\tfrac{\partial}{\partial \omega^1}+6(R^2_1R^1_{1;2}-R^1_1R^1_{1;1})\tfrac{\partial}{\partial \omega^2}\\
    +6(R^1_1R^1_{1;12}-R^1_{1;1}R^1_{1;2})\tfrac{\partial}{\partial\omega^1_1} 
    +6(R^1_{1;12}R^1_2+(R^1_{1;2})^2)\tfrac{\partial}{\partial\omega^1_2}+ 6(R^1_{1;12}R^1_2-(R^1_{1;2})^2)\tfrac{\partial}{\partial\omega_1^2}.
  \end{multline}
  The fact that $\what H_{K_\nabla}$ is an infinitesimal symmetry of $\nabla$   follows by showing that  the  Lie derivatives of $\omega^i$'s and $\omega^i_j$'s along $\what H_{K_\nabla}$ is zero, which is  a straightforward computation using structure equations \eqref{eq:phi-equiaffine-cartan-conn}, and  Bianchi identities \eqref{eq:1st-Bianchies}, \eqref{eq:2nd-Bianchies}, and \eqref{eq:3rd-Bianchi}.

\end{proof}

\begin{remark}
   It can be shown that moment flat special  symplectic connections are either locally flat or  homogeneous with 3-dimensional symmetry algebra isomorphic to  $\mathfrak{so}(p+1,q)$ or $\mathfrak{so}(p,q)\ltimes\RR^2$ where $(p,q)=(2,0)$ or $(1,1).$ Furthermore,  the corresponding pre-K\"haler structure in the case of non-flat homogeneous symplectic connections have cohomogeneity one.
\end{remark}

\subsection{Pre-Sasakian structures on homogeneous CR manifolds}
\label{sec:pre-sasak-struct}
Let us briefly explore a few properties in general dimension before returning to the $5$-dimensional pre-Sasakian case. The general-dimension results presented in this section provide Lie-theoretic descriptions of some pre-Sasakian structure equivalence classes, and they are based on similar ideas used for special symplectic connections in \cite{CS-special} and certain classes of Sasakian geometries in \cite{sykes2021classification}. 

Throughout the sequel, the term \emph{locally homogeneous CR manifold} refers to a CR manifold in which every point is in an open neighborhood that is CR equivalent to an open set on a common homogeneous CR manifold. For a point $p$ in a CR manifold $(S,D,J)$, consider the subset $Q_0$ of germs of local CR symmetries fixing $p$.  There is a natural group structure on $Q_0$, and it has a natural action on the algebra $\mathfrak{g}$ of infinitesimal CR symmetries in a neighborhood of $p$, given by applying the differential of  local diffeomorphisms representing germs in $Q_0$.  We refer to such $Q_0$ as the \emph{local isotropy group of} $(S,D,J)$ at $p$, and call this action its \emph{adjoint action} on $\mathfrak{g}$, as it coincides with the usual adjoint representation whenever $Q_0$ is the isotropy symmetry subgroup of a homogeneous manifold. 

If the CR structure underlying a pre-Sasakian structure $(S,X)$ is (locally) homogeneous then the germ of $(S,X)$ at a point $p\in S$ determines the pre-Sasakian structure in a neighborhood of $p$. This is a consequence of the following lemma.
\begin{lemma}\label{lemma: germ equivalence classes}
Let $(S,D,J)$ be a locally homogeneous CR manifold, $\mathfrak{g}$ its algebra of infinitesimal CR symmetries, and $p$ a point in $S$. The equivalence classes of germs at $p$ of pre-Sasakian structures modeled on $(S,D,J)$ are in one-to-one correspondence with the orbits of transverse symmetries (at $p$) in $\mathfrak{g}$ under the adjoint action on  $\mathfrak{g}$ of the local isotropy group of $(S,D,J)$ at $p$.
\end{lemma}
\begin{proof}
Let the pre-Sasakian structures $(S,X)$ and $(S,X^\prime)$ represent two germs of pre-Sasakian structures at $p$, where $X$ and $X^\prime$ are infinitesimal symmetries of $(S,D,J)$ transverse to $D$ at $p$. If these germs are equivalent then there is a local diffeomorphism $\varphi$ of a neighborhood of $p$ that is a CR symmetry and whose differential carries $X$ to $X^\prime$.

In the case where $(S,D,J)$ is locally homogeneous, we may describe it locally at $p$ as a left coset space $G/Q_0$ with $p=Q_0$, where $Q_0$ is the local isotropy group of $(S,D,J)$ at $p$, that is, given by the equivalence relation $g\sim g^\prime$ if $g=g^\prime h$ for some $h\in Q_0$. Its CR distribution is identified with a left invariant distribution on $G$ projected to $S\subset G/Q_0$, and infinitesimal CR symmetries are the right invariant vector fields on $G$ projected to $G/Q_0$. There is a basic fact from Lie theory that the left action of $G$ on $G/Q_0$  has an induced action of this symmetry algebra via the adjoint action of $G$ on $\mathfrak{g}$. The group $G$ can be taken in this setup so that all local symmetries in a sufficiently small neighborhood of $p\in S$ are restrictions of the left action by $G$ to $S$. Hence the aforementioned $X$ and $X^\prime$ are represented by right-invariant vector fields on $G$ in this local description, whereas $\varphi$ is given by the left action of a group element on $S\subset G/Q_0$ that fixes $p=Q_0$. That is, we may identify $\varphi=a\in Q_0$ and 
\begin{equation}\label{eqn: adjoint orbits}
X^\prime = \varphi_*X = \operatorname{Ad}_{a}X. 
\end{equation}
Note that since $\varphi$ is a CR automorphism, the isomorphism $\varphi_*$ cannot send vectors transverse to $D$ into $D$, and hence orbits generated by \eqref{eqn: adjoint orbits} for various $h\in Q_0$ consist only of infinitesimal CR symmetries that are transverse at $p$.
\end{proof}

\begin{corollary}\label{cor: germ determinacy}
 For every point $p$ in a pre-Sasakian manifold $(S,X)$ whose underlying CR structure is locally homogeneous, there exists an open neighborhood of $p$ in which the germs of $(S,X)$ at every other point in the neighborhood are determined by the germ at $p$.
\end{corollary}
\begin{proof}
By Lemma \ref{lemma: germ equivalence classes}, the germ at $p$ identifies a neighborhood of $p$ in $(S,X)$ with a neighborhood of the identity element in $(G/Q,V)$, where $G$ can be taken as the connected simply connected Lie group of the underlying CR symmetry algebra $\mathfrak{g}$, $Q$ its isotropy subgroup of the local action at $p$, and $V$ a vector field represented by some right invariant vector field on $G$. The lemma implies that ambiguity in the choice of $V$ is exhausted via transformations by the left action of $Q$ applied to $(G/Q,V)$.
\end{proof}
\begin{remark}[geometric equivalence]\label{rem: coarse equivalence}
In settings where such \emph{continuation phenomena} occur -- that is, where a structure on a manifold is determined in an open neighborhood by its germ at a point -- it is meaningful to consider a coarser equivalence relation on germs, whereby two germs $\alpha$ and $\alpha^\prime$ are equivalent if there is a sequence of germs $\alpha\cong \alpha_0,\alpha_1,\ldots,\alpha_n\cong \alpha^\prime$ such that each consecutive pair $(\alpha_{j-1},\alpha_j)$ both occur at different points on some path-connected manifold in the category, as this weaker equivalence roughly characterizes germs that can be connected by a natural structure continuation (such as analytic continuation for holomorphic structures). Such coarse equivalence classes have been studied, for example, in the category of Bochner-flat K\"{a}hler manifolds \cite{Bryant} and the more general category of special symplectic connections \cite{CS-special}. Corollary \ref{cor: germ determinacy} shows that this is a natural equivalence relation to consider for any class of pre-Sasakian structures whose underlying CR structures are locally equivalent to a common homogeneous model $(S,D,J)$. Each coarse equivalence class is comprised of exactly the pre-K\"{a}hler structures that can appear on a common connected manifold, which  we refer to as \emph{geometric equivalence} (distinct form standard germ equivalence). We will call the set of such equivalence classes for a given $(S,D,J)$ the \emph{moduli space of geometrically distinct pre-K\"{a}hler structures (arising from $(S,D,J)$)}. 
\end{remark}
To state the following corollary, for a symmetry algebra $\mathfrak{g}$ of a CR structure $(S,D,J)$, let $G$ be the connected simply connected Lie group of $\mathfrak{g}$, let $Q_0$ be the local isotropy group of $(S,D,J)$ at $p$. Let $\hat G$ denote the subgroup in $\operatorname{Aut}(\mathfrak{g})$ generated by the images of $G$ and $Q_0,$ denoted as
  \begin{equation}
    \label{eq:G-hat}
    \hat G=\langle \mathrm{Ad}_G,\mathrm{Ad}_{Q_0}\rangle\subseteq \mathrm{Aut}(\fg).
  \end{equation}
\begin{corollary}\label{cor: coarse equivalence via adjoint action}
Let $(S,D,J)$ be a finitely-nondegenerate locally homogeneous CR manifold. The moduli space of pre-K\"{a}hler structures arising from $(S,D,J)$ (defined in Remark \eqref{rem: coarse equivalence}) is in one-to-one correspondence with the space of orbits of $\hat G\subset \operatorname{Aut}(\mathfrak{g})$ acting on the CR symmetry algebra $\mathfrak{g}$ of $(S,D,J)$, where  $\hat{G}$ is as in \eqref{eq:G-hat}.
\end{corollary}
\begin{proof}
Let $G$ and $Q_0$ denote the same groups used in the definition of $\hat{G}$ as in \eqref{eq:G-hat}.
We first show that coarsely equivalent germs  $\alpha$ and $\alpha^\prime$ are represented by a unique $\hat{G}$-orbit. It will suffice to show this for each consecutive pair $(\alpha_{j-1},\alpha_j)$ from a sequence of germs $\alpha\cong \alpha_0,\alpha_1,\ldots,\alpha_n\cong \alpha^\prime$ as in Corollary \ref{cor: germ determinacy}, so let's assume $\alpha$ and $\alpha^\prime$ are respectively equivalent  to germs $\alpha_0$ and $\alpha_1$ at two points $p$ and $p^\prime$ in a path-connected pre-Sasakian manifold $(S,X)$ modeled on a locally homogeneous CR structure $(S,D,J)$. Any path connecting $p$ and $p^\prime$ can be covered by finitely many open neighborhoods on which the pre-Sasakian structure is equivalent to a neighborhood of the identity element in $(G/Q,V)$, where $G/Q$ is a homogeneous CR manifold (with $G$ connected and simply connected) and $V$ is a transverse infinitesimal CR symmetry on the neighborhood, because every point in $S$ has such a neighborhood. By chaining together such neighborhoods, we can reduce to the simplified case where $(S,D,J)$ is (globally) homogeneous, 
\[
S=G/Q,
\] 
and the infinitesimal CR symmetry $X$ is assumed only to be transverse at $p$ and $p^\prime$. Taking $a\in G$ satisfying $p= a p^\prime$, the differential of left translation by $a$ transforms $X$ to $\operatorname{Ad}_{a}X$, so the germ $\alpha_1$ of $(S,X)$ at $p^\prime$ is equivalent to the germ of $(S,\operatorname{Ad}_{a}X)$ at $p$ by construction. Therefore, coarsely equivalent germs $\alpha_0$ and $\alpha_1$ in this simplified case are represented by elements in $\mathfrak{g}$ belonging to the same $G$-orbit. By Lemma \ref{lemma: germ equivalence classes}, $\alpha_0$ and $\alpha_1$ are respectively related to $\alpha$ and $\alpha^\prime$ by the $\operatorname{Ad}_{Q_0}$ action, so $\alpha$ and $\alpha^\prime$ are in the same $\hat{G}$-orbit.

Conversely,  for $X$ in $\mathfrak{g}$, let $p\in G/Q$ be any point in the CR manifold. For a group element $a\in G$, if $X$ is transverse at $a^{-1}p$ then the pre-Sasakian germ at $a^{-1}p$ defined by $X$ transforms to $\operatorname{Ad}_{a}X$ at $p$ under the left action of $a$. So the germs of $(G/Q,X)$ and $(G/Q,\operatorname{Ad}_{a}X)$ at $p$, (i.e., germs in a common $G$-orbit) are coarsely equivalent. It is easily concluded from here and Lemma \ref{lemma: germ equivalence classes}  that $X$ and $\operatorname{Ad}_{h_2}\circ\operatorname{Ad}_{a}\circ\operatorname{Ad}_{h_1}X$ are also coarsely equivalent for any $h_j\in Q_0$, and hence the larger $\hat{G}$-orbit consists of coarsely equivalent germs. Finite non-degeneracy of the CR structure implies that $X$ will be transverse almost everywhere on $G/Q$, so for arbitrary $X\in \mathfrak{g}$ and $p\in G/G_0$ there will exist elements in the conjugacy class of $X$ that define a pre-Sasakian structure at $p$, thereby identifying the $G$-orbit with a unique coarse equivalence class of germs.  Finite non-degeneracy is essential for the converse direction because without it there can exist $G$ orbits in $\mathfrak{g}$ consisting of nowhere transverse infinitesimal CR symmetries.
\end{proof}

The classification of homogeneous $5$-dimensional hypersurface-type $2$-nondegenerate CR structures is derived in \cite{FK-CR}, a major development in Levi degenerate CR geometry. Consequently, it provides a large class of pre-Sasakian structures, as each of the classified CR hypersurfaces possess many transverse symmetries. Relating \cite{FK-CR} to the classification of affinely homogeneous surfaces in $\mathbb{R}^3$ from \cite{DKR-affine}, one gets that all homogeneous $2$-nondegenerate hypersurfaces in $\mathbb{C}^3$  are locally equivalent to hypersurfaces of the form
\begin{equation}\label{eqn: homogeneous CR realizations}
\{(w,z^1,z^2)\,|\, \Re(w)=\rho_a(z)-2\} 
\quad\quad\forall\, a\neq 0,1
\end{equation}
at the origin, where $\rho_a$ is as in \eqref{rho-Doubrov-example} from Example \ref{ex: light cone potential}.  For the other parameter values $a=0,1$, the structures are holomorphically degenerate.   From \eqref{eqn: homogeneous CR realizations}, we can calculate these hypersurfaces' CR symmetries, which in turn can be used to describe all pre-Sasakian structures on the hypersurfaces along with their respective symmetries.

For the $a$-parameterized family of hypersurfaces in \eqref{eqn: homogeneous CR realizations}, their infinitesimal symmetry algebras all contain
\begin{equation}\label{eqn: 5d. hom CR symmetries a}
\begin{gathered}
X_0 := \frac{\partial}{\partial v},
\quad
X_1 := \frac{\partial}{\partial y^1},
\quad
X_2 := \frac{\partial}{\partial y^2}, \\
X_3 := \left(\frac{x^{1}}{2} + \frac{1}{2}\right) \frac{\partial}{\partial x^{1}} + \frac{y^{1}}{2} \frac{\partial}{\partial y^{1}} + \frac{a \left(x^{2} + 1\right)}{2 \left(a - 1\right)} \frac{\partial}{\partial x^{2}} + \frac{a y^{2}}{2 \left(a - 1\right)} \frac{\partial}{\partial y^{2}},\quad \text{ and}\\
X_4 :=\left(\frac{x^{1}}{2} + \frac{1}{2}\right) \frac{\partial}{\partial x^{1}} + \frac{y^{1}}{2} \frac{\partial}{\partial y^{1}} + \frac{a \left(- x^{2} - 1\right)}{2 \left(a - 1\right)} \frac{\partial}{\partial x^{2}} -  \frac{a y^{2}}{2 a - 2} \frac{\partial}{\partial y^{2}} + a \left(u + 2\right) \frac{\partial}{\partial u} + a v \frac{\partial}{\partial v},
\end{gathered}
\end{equation}
in coordinates $w=u+\operatorname{i} v$ and $z^j = x^j+\operatorname{i} y^j$. Except for the values $a=-1,\tfrac{1}{2},2$ where this hypersurface has the flat $2$-nondegenerate structure, these symmetries span the symmetry algebra, as it is shown to be $5$-dimensional in  \cite{FK-CR}. We can apply Lemma \ref{lemma: germ equivalence classes}  to describe the spaces of germs of $2$-nondegenerate pre-Sasakian structures modeled on homogeneous CR hypersurfaces, and obtain the following dimension count.

\begin{proposition}\label{prop: germ equivalence class dimension}
Local equivalence classes define a codimension $5$ foliation on an open dense subset in the space of germs of pre-Sasakian structures modeled on any fixed $5$-dimensional $2$-nondegenerate locally homogeneous CR manifold.
\end{proposition}
\begin{proof}
If the underlying CR structure is not flat,
then the result follows from Lemma \ref{lemma: germ equivalence classes} because, as discussed above, such CR structures are simply transitive.  

If the underlying CR structure is flat, then, as described in Section \ref{sec:symm-reduct-flat}, the CR symmetry algebra is $\mathfrak{g}:=\mathfrak{sp}(4)$ with a matrix representation
\[
\left\{\left.
\left(
\begin{array}{cccc}
-s_1 & s_2+\operatorname{i}s_3 & s_2 - \operatorname{i}s_3 & s_4 \\
2(r_1 - \operatorname{i}r_2) & -\operatorname{i}s_5 & r_3-\operatorname{i}r_4 & 2\operatorname{i}(s_2 - \operatorname{i}s_3) \\
2(r_1 + \operatorname{i}r_2) & r_3+\operatorname{i}r_4 & \operatorname{i}s_5 & -2\operatorname{i}(s_2 + \operatorname{i}s_3) \\
r_5 & -\operatorname{i}(r_1 + \operatorname{i}r_2) & \operatorname{i}(r_1 - \operatorname{i}r_2) & s_1
\end{array}
\right)
\right|
s_j,r_j\in\mathbb{R}
\right\},
\]
where $(s_1,\ldots, s_5)$ parameterize the Lie subalgebra $\mathfrak{h}$ of an isotropy subgroup $H$. For each $g\in \mathfrak{g}$, the tangent space at $g$ to the orbit of the adjoint action of $H$ on $\mathfrak{g}$ through $g$ is $\{\operatorname{ad}_h(g)\,|\, h\in \mathfrak{h}\}$. It is a straightforward computation to verify that $\{\operatorname{ad}_h(g)\,|\, h\in \mathfrak{h}\}$ is $5$-dimensional for generic $g$, and hence orbits of $H$ foliate an open dense subset of $\mathfrak{g}$ with $5$-dimensional leaves.
\end{proof}
We can similarly describe the spaces of geometrically inequivalent pre-K\"{a}hler structures using Corollary \ref{cor: coarse equivalence via adjoint action}, which can be combined with Proposition \ref{prop: germ equivalence class dimension} to find lower bounds on symmetry group dimensions.

\begin{theorem}\label{thm: coarse moduli dim}
The space of geometrically inequivalent pre-K\"{a}hler structures arising from any fixed $5$-dimensional $2$-nondegenerate locally homogeneous CR manifold contains a $2$-dimensional open dense subset. Each such pre-K\"{a}hler structure has nontrivial infinitesimal symmetries.
\end{theorem}
\begin{proof}
Given the classification of such homogeneous CR structures, we may consider only the pre-Sasakian structures modeled on \eqref{eqn: homogeneous CR realizations} for various values of the parameter $a$. Assuming $2$-nondegeneracy precludes $a=0,1$.

If the CR structure is flat, then, due to Proposition \ref{prop:Bochner-flat-pre-Kahler}, this theorem reduces to \cite[Corollary C]{CS-special}. Because the argument is brief, however, we note how to get this theorem's first part directly: it follows from the fact that for $G=\mathrm{Sp}(4,\RR)$, as for any semi-simple Lie group $G$, a generic element in $\fg$ is $G$-conjugate to an element in a Cartan subalgebra. And hence, by Corollary \ref{cor: coarse equivalence via adjoint action}, the $2$-dimensional Cartan subalgebra in $\fg$ represents an open dense subset in the space of geometrically distinct structures. 

If the CR structure is not flat, on the other hand, then we may furthermore assume $a\neq -1,\tfrac{1}{2},2$ (see Example \ref{exa:special-homog-equiaaffine}).  For each such $a$, the CR symmetry algebra $\mathfrak{g}$ is the $5$-dimensional algebra represented by \eqref{eqn: 5d. hom CR symmetries a}.  Letting $(e_1,\ldots, e_5)$ be labels for those five vector fields in the order they are listed and letting $t=(t_1,\ldots, t_5)\in\mathbb{R}^5$ represent the element $\sum t_ie_i\in\mathfrak{g}$, we have
\[
\operatorname{ad}_{e_1}(t) = \begin{bmatrix} a t_5 \\ 0\\ 0\\ 0\\ 0 \end{bmatrix}, \quad
\operatorname{ad}_{e_2}(t) = \frac{1}{2}\begin{bmatrix} 0\\ t_4+t_5\\ 0\\ 0\\ 0 \end{bmatrix}, \quad
\operatorname{ad}_{e_3}(t) = \frac{a}{2(a-1)}\begin{bmatrix} 0 \\ 0\\ t_4-t_5\\ 0\\ 0 \end{bmatrix},
\]
\[
\operatorname{ad}_{e_4}(t) = \frac{-1}{2(a-1)}\begin{bmatrix} 0 \\ (a-1)t_2\\ at_3\\ 0\\ 0 \end{bmatrix}, \quad
\operatorname{ad}_{e_5}(t) = \frac{-1}{2(a-1)}\begin{bmatrix} 2(a-1)t_1 \\ (a-1)t_2\\ -at_3\\ 0\\ 0 \end{bmatrix}
\]

These five vectors span a $3$-dimensional space for generic $t$ representing  $\{\operatorname{ad}_x(t)\,|\, x\in \mathfrak{g}\}$, which is the tangent space of the equivalence class orbit passing through $t\in \mathfrak{g}$, by Corollary \ref{cor: coarse equivalence via adjoint action}. And in general, their spans are at most $3$-dimensional for all $t$. In fact, for al $t\in \mathfrak{g}$, their spans are contained in the tangents of affine $3$-spaces parallel to the span of $\{e_1, e_2, e_3\}$.

Thus, away from this foliation's singularity locus $\big\{t\,\big|\, \dim \{\operatorname{ad}_x(t)\,|\, x\in \mathfrak{g}\}<3\big\}$, the foliation is simply partitioning  $\mathfrak{g}$ into parallel affine $3$-spaces. Since $ \mathfrak{g}$ is $5$-dimensional, the leaf space of this foliation minus the singularity locus is a $2$-dimensional manifold.  By Corollary \ref{cor: coarse equivalence via adjoint action}, this leaf space is the corresponding space of geometrically distinct pre-K\"{a}hler structures.

Regarding symmetries in the non-flat CR case, consider the natural map from a pre-K\"{a}hler manifold $M$ to the space of germs of pre-K\"{a}hler structures modeled on the same locally homogenoeus CR manifold, which is identified with that CR manifold's Lie algebra $\fg$ of infinitesimal symmetries. Assuming the pre-K\"{a}hler manifold is connected, then, by definition, the image of this map is contained in one \emph{geometric equivalence class} $\mathcal{O}\subset \fg$ (i.e. an orbit of the adjoint action of $\hat{G}$ from Corollary \ref{cor: coarse equivalence via adjoint action}), which we have seen is a manifold of dimension at most $3$. Quotienting this  \emph{geometric equivalence class} $\mathcal{O}$ by the standard local equivalence relation $\sim$ on germs does not change the dimension, by Proposition \ref{prop: germ equivalence class dimension}. Thus we have a fibration $M\to \mathcal{O}/\!\!\sim$ whose fibers are at least $1$-dimensional because $\dim(M)=4$ and $\dim( \mathcal{O}/\!\!\sim)\leq 3$.  Since germs at every point along a fiber belong to the same local equivalence class, local symmetries act locally transitively on the fibers.

For completeness, we note how these symmetries can also be found through direct calculation: For a non-flat $2$-nondegenerate CR manifold $S$ of the form \eqref{eqn: homogeneous CR realizations} given by the parameter $a$, let $X_0,\ldots, X_{4}$ be the vector fields of \eqref{eqn: 5d. hom CR symmetries a} and let the infinitesimal $X = \sum \alpha_jX_j$ for some coefficient parameters $\alpha_j$ represent a pre-Sasakian structure (whereas for the flat CR case, one should take $X$ as a general element from its $10$-dimensional symmetry algebra instead). Infinitesimal symmetries of $(S,X)$ are the CR symmetries commuting with $X$, which one can compute directly. For all $a$, this centralizer is at least $2$-dimensional, so the pre-K\"{a}hler symtry algebra, which is this pre-Sasakian symmetry algebra quotiented by the span of $X$, is at least $1$-dimensional. 
\end{proof}

\begin{remark}
In Remark \ref{rmk:Q-quartic-examples} it was pointed out that pre-K\"ahler structures with potential $\rho_a$ in \eqref{rho-Doubrov-example}, which are clearly among symmetry reductions of homogeneous 2-nondegenerate CR 5-manifolds, satisfy the invariant relation $\bQ-c(\bR)^2=0,$ for some constant $c\in\RR.$ Furthermore, one can check that they correspond to moment flat symplectic connections with 3-dimensional symmetry algebra.  In the spirit of Propositions \ref{prop:Bochner-flat-pre-Kahler} and \ref{prop:crit-spec-sympl-1}, it would be interesting to determine whether every member of the 2-parameter families of pre-K\"ahler structure obtained via symmetry reductions of non-flat homogeneous CR structures correspond to a critical symplectic connection and whether the relation $\bQ-c(\bR)^2=0$  characterizes them.
\end{remark}

\subsection{Homogeneous pre-K\"{a}hler complex surfaces}
\label{sec:homog-pre-kahl}
Gathering this section's preceding results, we can now prove that every locally homogeneous $4$-dimensional $2$-nondegenerate pre-K\"{a}hler structure is flat.

\begin{proposition}\label{prop: hom implies symm red}
Every homogeneous $4$-dimensional $2$-nondegenerate pre-K\"{a}hler structure is obtained via symmetry reduction of a flat 2-nondegenerate CR 5-manifold.
\end{proposition}
\begin{proof}
By Theorem \ref{thm: correspondence theorem}, it suffices to describe all $5$-dimensional $2$-nondegenerate homogeneous pre-Sasakian structures, which as discussed above are all modeled locally on the CR manifolds in \eqref{eqn: homogeneous CR realizations}. For all but the flat CR structure, these CR hypersurfaces have $5$-dimensional symmetry algebras spanned by the vector fields in  \eqref{eqn: 5d. hom CR symmetries a}, which follows from direct calculation and the dimension for such structures symmetry algebras given in \cite{FK-CR}.

The algebras spanned by \eqref{eqn: 5d. hom CR symmetries a}, however, have trivial centers, so there is no choice of transverse symmetry that commutes with this full $5$-dimensional symmetry algebra. 
\end{proof}

\begin{proposition}\label{prop:symm-red-homog-CR-flat-pre-kahl}
  Among pre-K\"ahler structures obtained via symmetry reduction of a flat 2-nondege\-nerate CR 5-manifold, the flat pre-K\"ahler structure is the only locally homogeneous one. 
\end{proposition}
\begin{proof}
 It was obtained in \eqref{eq:invariants} that  pre-K\"ahler complex surfaces obtained via symmetry reduction of the flat 2-nondegenerate CR 5-manifold satisfy $T_2=0.$ The $\mathrm{U}(1)$-action on the function $T_1$ in \eqref{eq:cartan-curv-prekahler} can be found using the gauge transformation of the Cartan connection $\varphi(p)$  in  \eqref{eq:Cartan-conn-curv}, at every $p\in \cG,$ i.e.  $\varphi(p)\to \varphi(\bu^{-1}p)=\bu^{-1}\eta\bu+\bu^{-1}\exd\bu$ where $\bu$ is the diagonal matrix $\mathrm{diag}(1, e^{-\ri r},e^{\ri r}).$ One obtains  $T_1(\bu^{-1}p)=e^{-2\ri r}T_1(p).$ Infinitesimally, this transformation law is encoded in the first equation in \eqref{eq:Bianchies-closed-system}. By Theorem \ref{thm:cart-geom-descr}, if $\bT_2=0$ then non-flatness of  a pre-K\"ahler structure implies $\bT_1\neq 0.$ Writing $T_1=t_1+\ri t_2,$ the induced  $\mathrm{U}(1)$ action on  $t_2$ is given by $t_2(p)\to -2t_1\sin(2r)+2t_2\cos(2r).$ As a result, given a 3-adapted coframe on $M$ for such non-flat pre-K\"ahler structure,  by setting $r=\tan^{-1}(\tfrac{t_2}{t_1})$ one always obtains  a section $s\colon M\to \cG$ with respect to which $T_1=\ov{T_1}=t_1$ is a non-zero real-valued function on $M.$ The pull-back of the first equation of  \eqref{eq:Bianchies-closed-system} to this section gives
 \begin{equation}\label{eq:reduced-psi}
   s^*\psi=-\frac{T_3}{2t_1}(\ov{\theta^2}-\theta^2)+\frac{\ri}{4t_1}(\ov{T_{1;1}}\ov{\theta^1}-T_{1;1}\theta^1),
    \end{equation}
 where  the function  $s^* T_{1;1}$ can no longer be interpreted as the coframe derivative of $T_1$ on the section $s.$ 

 Inserting expression \eqref{eq:reduced-psi} in structure equations  \eqref{eq:dtheta12-on-cG}, the condition of being homogeneous corresponds to  those values of $(t_1,T_3,S)$  for which  the pull-back of the structure equations by $s$ correspond to Maurer-Cartan equations for some Lie algebra. In particular, the  coefficients of the reduced structure equations need to be constant.  For instance, the pull-back of $\exd \theta^1$ is expressed as
 \[\exd\theta^1\equiv  \frac{\ri T_3}{t_1}\theta^1\w(\ov{\theta^2}-\theta^2)\mod\{\theta^2\},\]
 wherein, by abuse of notation, we have suppressed $s^*$. Being locally homogeneous implies that either $T_3=0$ or  $\tfrac{T_3}{t_1}$ is constant. Using the pull-back of the first and second equations in \eqref{eq:Bianchies-closed-system} to $s,$ one obtains
 \[
   \begin{aligned}
     \exd\left(\frac{T_3}{t_1}\right)&\equiv \frac{T_3^2+t_1^2}{2t_1^2}\left(\theta^2+\ov{\theta^2}\right)\mod\{\theta^1,\ov{\theta^1}\},\\
     \exd T_3&\equiv \ri t_1\left(\theta^2+\ov{\theta^2}\right)\mod\{\theta^1,\ov{\theta^1}\}.
   \end{aligned}
 \]
 As a result, being homogeneous implies  that $t_1=0$ is a necessary condition  on  $s\colon M\to \cG$ which  contradicts the assumption $T_1\neq 0$ on   $s.$ Hence, for such pre-K\"ahler structures, being locally homogeneous implies flatness. 
\end{proof}

Propositions \ref{prop: hom implies symm red} and \ref{prop:symm-red-homog-CR-flat-pre-kahl} prove the following.
\begin{theorem}\label{thm:homog-prekahler}
A  2-nondegenerate pre-K\"ahler complex surface is locally homogeneous if and only if it is flat.
\end{theorem}
\section*{Acknowledgements}
The first author was supported by the Troms{\o} Research Foundation (project ``Pure Mathematics in Norway'') and the UiT Aurora project MASCOT. The second author was supported by the Institute for Basic Science (IBS-R032-D1).  The second author is grateful to Gerd Schmalz and Martin Kol\'{a}\v{r} for many illuminating discussions on Sasakian geometry. The  EDS calculations  are done using Jeanne Clelland's \texttt{Cartan} package in Maple and also by the python package \texttt{dgcv}  \cite{dgcv}.

%-------------BIBLIOGRAPHY--------
\bibliographystyle{alpha}      
\bibliography{MSBib}  
\end{document}